\documentclass[11pt]{article}

\usepackage[utf8]{inputenc}
\usepackage[T1]{fontenc}
\usepackage[a4paper]{geometry}
\usepackage{csquotes}
\usepackage[english]{babel}
\usepackage{amsmath, amsthm}
\usepackage{amsfonts}
\usepackage{amssymb}
\usepackage[all]{xy}
\usepackage{graphicx}
\usepackage{tikz}
\usetikzlibrary{arrows, cd}
\usepackage{wrapfig}
\usepackage{multirow}
\usepackage{a4wide}
\usepackage{enumerate}
\usepackage{hyperref}
\usepackage{stmaryrd}
\usepackage{genyoungtabtikz}
\usepackage{mathtools}
\usepackage{faktor, shuffle}
\usepackage{mathdots}
\usepackage{color}
\definecolor{cyanp}{rgb}{0.0, 0.55, 0.55}
\usepackage[nottoc, notlof, notlot]{tocbibind}

\allowdisplaybreaks

\newtheorem{theorem}{Theorem}[section]
\newtheorem{lemma}[theorem]{Lemma}
\newtheorem{proposition}[theorem]{Proposition}
\newtheorem{corollary}[theorem]{Corollary}
\newtheorem{definition}[theorem]{Definition}

\newtheorem{notation}[theorem]{Notation}

\theoremstyle{definition} 
\newtheorem{remark}[theorem]{Remark}
\newtheorem{example}[theorem]{Example}
\numberwithin{figure}{section}
\numberwithin{equation}{section}
\hypersetup{colorlinks=true, citecolor=blue, linkcolor=blue}

\newcommand{\eqdef}{\mbox{\ \raisebox{0.2ex}{\scriptsize\ensuremath{\mathrm:}}\ensuremath{=}\ }} 

\newcommand{\Des}{{\operatorname{Des}}}         
\newcommand{\cset}{{\operatorname{C}}}          
\newcommand{\compof}{\vDash}                    
\newcommand{\lon}{\operatorname{\ell}}          
\newcommand{\cont}{\operatorname{cont}}         
\newcommand{\Cont}{\operatorname{cont}}         
\newcommand{\interv}[2]{\llbracket #1, #2\rrbracket}

\newcommand{\C}{\mathbb{C} } 
\newcommand{\Z}{\mathbb{Z} } 
\newcommand{\N}{\mathbb{N} } 
\newcommand{\F}[1]{\mathbb{F}_{#1}}
\newcommand{\Id}{\mathrm{Id}}
\newcommand{\Fq}{\F{q}}

\newcommand{\SG}[1]{\mathfrak{S}_{#1}} 

\newcommand{\GL}[2]{\mathbf{GL}_#1 \left( #2 \right)}

\newcommand{\emptyw}{\varepsilon}

\DeclareMathOperator{\Stell}{Stell}
\DeclareMathOperator{\St}{St}

\DeclareMathOperator{\Res}{Res}
\DeclareMathOperator{\Ind}{Ind}
\DeclareMathOperator{\Hom}{Hom}
\DeclareMathOperator{\Split}{Split}

\newcommand{\Hnz}{H_n^0}
\newcommand{\Rnz}{R_n^0}
\newcommand{\Snz}{\Stell_n^0}

\DeclareMathOperator{\Inv}{Inv} 
\DeclareMathOperator{\InvB}{\overline{Inv}} 
\DeclareMathOperator{\supp}{supp} 
\newcommand{\MCR}[1]{\mathcal{MCR}_{#1}} 
\newcommand{\MCT}[1]{\mathcal{MCT}_{#1}}	
\DeclareMathOperator{\code}{code}
\DeclareMathOperator{\decode}{decode}
\DeclareMathOperator{\PZ}{FZ}
\DeclareMathOperator{\Lehmer}{Lehmer}
\DeclareMathOperator{\CyW}{CycleW}
\DeclareMathOperator{\ChW}{ChainW}
\newcommand{\sfc}{\mathsf{c}}
\newcommand{\sfb}{\mathsf{b}}
\newcommand{\col}[3][black]{\,\scalebox{0.6}{\color{#1}$\begin{array}{ | c | }
     \hline
     #2 \\
     \vdots \\
     #3 \\
     \hline
   \end{array}$\,\,} }
\newcommand{\cols}[3][black]{\,\scalebox{0.6}{\color{#1}$\begin{array}{|| c ||}
     \hline
     #2 \\
     \vdots \\
     #3 \\
     \hline
   \end{array}$\,} }
\DeclareMathOperator{\rad}{Rad} 
\newcommand{\Si}[1]{S_{\lbrace #1 \rbrace}}
\newcommand{\ev}[1]{{\left\langle #1 \right\rangle}}
\newcommand\decn[1]{\overline{#1}^{\scalebox{0.7}{n}}}
\newcommand{\tMonoid}{M}
\newcommand{\idempMon}[1][\tMonoid]{E(#1)} 
\newcommand{\JJ}{\mathcal{J}}
\newcommand{\LL}{\mathcal{L}}
\newcommand{\RR}{\mathcal{R}}
\newcommand{\KK}{\mathcal{K}}
\DeclareMathOperator{\rAut}{rAut}
\DeclareMathOperator{\rfix}{rfix}
\DeclareMathOperator{\lAut}{lAut}
\DeclareMathOperator{\lfix}{lfix}
\newcommand{\meR}{\wedge_\RR} 
\newcommand{\jnR}{\vee_\RR} 
\newcommand{\w}[1]{\mathbf{#1}}  
\newcommand{\un}[1]{\underline{#1}} 

\newcommand{\ov}[1]{\overline{#1}}

\newcommand{\qandq}{\qquad\text{and}\qquad}
\newcommand{\suchthat}{\mid}

\def\rsbox#1#2#3{\raisebox{#1}{\scalebox{#2}{$#3$}}}

\newcommand\ylwB{\Yfillcolour{yellow}}

\newcommand\redB{\Yfillcolour{red}}
\newcommand\bluB{\Yfillcolour{blue}}

\newcommand\limeB{\Yfillcolour{lime}}

\newcommand{\eqz}{\equiv_0}

\newcommand{\Sym}{\mbox{\bf Sym}}		
\newcommand{\QSym}{\mbox{\bf QSym}}		
\newcommand{\NCSF}{\mbox{\bf NCSF}}		

\title{The $0$-Rook Monoid and its Representation Theory}
\author{Joël~Gay \and Florent~Hivert}

\newcommand{\Addresses}{{
  \bigskip
  \footnotesize\noindent
  Joël Gay: \textit{E-mail address:} \url{joel.gay@lri.fr} \\\noindent
  Florent Hivert: \textit{E-mail address:} \url{florent.hivert@lri.fr} \\
  \textsc{Laboratoire de Recherche en Informatique (LRI, UMR CNRS 8623),
    Université Paris Sud, Université Paris-Saclay, CNRS},
  \par\nopagebreak\noindent
  Bâtiment 650, Université Paris Sud 11, 91405 ORSAY CEDEX, FRANCE.
  \par\nopagebreak%
}}

\begin{document}
\maketitle

\begin{abstract}
  We show that a proper degeneracy at $q=0$ of the $q$-deformed rook monoid of
  Solomon is the algebra of a monoid $R_n^0$ namely the $0$-rook monoid, in
  the same vein as Norton's $0$-Hecke algebra being the algebra of a monoid
  $H_n^0 \eqdef H^0(A_{n-1})$ (in Cartan type~$A_{n-1}$).  As expected,
  $R_n^0$ is closely related to the latter: it contains the $H^0(A_{n-1})$
  monoid and is a quotient of $H^0(B_{n})$. We give a presentation for this
  monoid as well as a combinatorial realization as functions acting on the
  classical rook monoid itself. On the way we get a Matsumoto theorem for the
  rook monoid a result which was conjectured by Solomon.

  The $0$-rook monoid shares many combinatorial properties with the Hecke
  monoid: its Green right preorder is an actual order, and moreover a lattice
  (analogous to the right weak order) which has some nice combinatorial, and
  geometrical features. In particular the $0$-rook monoid is $\JJ$-trivial.

  Following Denton-Hivert-Schilling-Thi\'ery, it allows us to describe its
  representation theory including the description of the simple and projective
  modules. We further show that $R_n^0$ is projective on $H_n^0$ and make
  explicit the restriction and induction functors along the inclusion map. We
  finally give a (partial) associative tower structures on the family of
  $(R_n^0)_{n\in \N}$ and we discuss its representation theory.
\end{abstract}
\tableofcontents

\section{Introduction}

This article is the first of a series of two on the degeneracy at $q=0$ of the
$q$-rook and more generally $q$-Renner monoids and their representation
theory~\cite{GayHivert.2018}. This first paper is focused on Cartan type $A$,
that is only on the rook case. We start by recalling
Iwahori's~\cite{Iwahori.1964} construction of the Iwahori-Hecke algebra, and
the importance of the $q=0$ degeneracy.

\subsection{Iwahori-Hecke algebra and its degeneracy at \texorpdfstring{$q=0$}{q=0}}
\label{sec_Hecke}

Let $\Fq$ be the finite field with $q$ elements. Let $G \eqdef \GL{n}{\Fq}$ be
its general linear group of invertible $n\times n$ matrices, and $B\subset G$
its subgroup of upper triangular matrices. Both groups $G$ and $B$ are finite
of respective cardinalities
$\vert G\vert = (q^n-1)(q^n-q)(q^n-q^2)\dots (q^n-q^{n-1})$ and
$\vert B\vert = (q-1)^n q^{\binom{n}{2}}$.  We denote~$\SG{n}$ the symmetric
group acting on
$\{1,\dots,n\}$ and identify a permutation with its associated permutation
matrix. The Bruhat decomposition~\cite{BjornerBrenti.2005} tells that for all
$M \in G$ there is a unique permutation $\sigma \in \SG{n}$ such that $M \in
B\sigma B$, that is :
\begin{equation}
  G = \bigsqcup_{\sigma \in \mathfrak{S}_n} B\sigma B.
\end{equation}
For $\sigma\in\SG{n}$, let $T_\sigma$ be the element of the
group algebra $\C[G]$ defined by:
\begin{equation}
  T_\sigma \eqdef \frac{1}{\vert B\vert}\sum_{x\in B\sigma B} x.
\end{equation}
The Hecke ring $\mathcal{H}(G,B)$ is the $\Z$-ring spanned by the $T_w$. Its
identity is $\varepsilon = T_\Id = \frac{1}{\vert B\vert}\sum_{b\in B}
b$. Furthermore, $\mathcal{H}(G,B) = \varepsilon\, \Z[G]\, \varepsilon$.
Let $S = \lbrace s_1, \dots, s_{n-1}\rbrace$ be the elementary transpositions
which generate $\SG{n}$ as a group. For $q\in \mathbb{C}$, let
$\mathcal{H}_{\Z}(q)$ denote the $\Z$-algebra defined by generators and
relations as follows:
\begin{alignat}{3}
  T_i^2 &= q\cdot 1 + (q-1)T_i    &\qquad& 1\leq i \leq n-1, \tag{H1}\label{relH1}\\
  T_iT_{i+1}T_i &= T_{i+1}T_iT_{i+1} &&  1\leq i \leq n-2,      \tag{H2}\label{relH2}\\
  T_iT_j &= T_jT_i                &&\vert i-j\vert \geq 2,   \tag{H3}\label{relH3}.
\end{alignat}
If $q$ is the cardinality of a finite field, Iwahori proved that the maps $T_i
\mapsto T_{s_i}$ extends to a full ring isomorphism from $\mathcal{H}_{\Z}(q)$
to $\mathcal{H}(G,B)$ and consequently, the equations above give a
presentation of $\mathcal{H}_{\Z}(q)$. By extending the scalar to $\C$ we get
a $\C$-algebra $\mathcal{H}_{\C}(q)$ which extends the definition of the Hecke
ring outside of prime powers. It is well known that when $q$ is neither zero
nor a root of the unity, the Iwahori-Hecke algebra is isomorphic to the
complex group algebra $\C[\SG{n}]$.
\bigskip

The degeneracy at $q=0$ of the Iwahori-Hecke algebra has many interesting
properties and applications. Its first appearance is perhaps in Demazure
character formula~\cite{Demazure.1974} through divided differences. Then, its
central role in Schubert calculus was discovered by
Lascoux~\cite{Lascoux.2001,Lascoux.2003.book,Lascoux.2003.slc}, with further
recent connection with $K$-theory through Grothendieck polynomials (see
\textit{e.g.}  \cite{Miller.2005,Lam_Schilling_Shimozono.2010}). Its
representation theory was first studied by Norton~\cite{Norton.1979} in type A
and Carter~\cite{Carter.1986} in the other types. In type $A$, Krob and
Thibon~\cite{KrobThibon.1997} explained how induction and restriction of these
modules give an interpretation of the products and coproducts of the Hopf
algebras of noncommutative symmetric functions and quasi-symmetric functions,
giving thus analogue of the well know Frobenius isomorphism from the character
ring of the symmetric groups to symmetric functions~(See
\textit{e.g.}~\cite{MacDonald.1995}). This was the main motivation for the
present work at the beginning. Two other important steps were further made by
Duchamp--Hivert--Thibon~\cite{DuchampHivertThibon.2002} for type $A$ and
Fayers~\cite{Fayers.2005} for other types, using the Frobenius structure to
get more results, including a description of the
Ext-quiver. Denton~\cite{Denton.2010} gave a family of minimal orthogonal
idempotents.

This degeneracy is defined by putting $q=0$ in the relation of the $q$-Iwahori-Hecke
algebra:
\begin{alignat}{2}
  T_i^2 &= -T_i    &\qquad& 1\leq i \leq n-1,\\
  T_iT_{i+1}T_i &= T_{i+1}T_iT_{i+1} &&  1\leq i \leq n-2\\
  T_iT_j &= T_jT_i                &&\vert i-j\vert \geq 2.
\end{alignat}
One interesting remark which as been discovered independently several times is
that this is the algebra of a monoid~\cite{DHST}. To see this, they are two
possibilities: define either $\pi_i \eqdef -T_i$ or $\pi_i \eqdef T_i+1$, and get the
following presentation of the Hecke monoid at $q=0$, which we denote $H_n^0$
(as opposite to its algebra denoted by $H_n(0)$):
\begin{alignat}{2}
  \pi_i^2 &= \pi_i    &\qquad& 1\leq i \leq n-1,
  \tag{M1}\label{relM1}\\
  \pi_i\pi_{i+1}\pi_i &= \pi_{i+1}\pi_i\pi_{i+1} &&  1\leq i \leq n-2,
  \tag{M2}\label{relM2}\\
  \pi_i\pi_j &= \pi_j\pi_i                &&\vert i-j\vert \geq 2.
  \tag{M3}\label{relM3}
\end{alignat}
For a permutation $\sigma$, one defines $\pi_\sigma \eqdef \pi_{i_1}\dots
\pi_{i_k}$ where $s_{i_1}\dots s_{i_k}$ is any reduced word (word of minimal
length) for $\sigma$. Thanks to the braid relations~\ref{relM2},\ref{relM3},
and Matsumoto's theorem the result is independent of the choice of the reduced
word. Then $H_n^0$ is nothing but the set $\{\pi_\sigma\mid\sigma\in\SG{n}\}$
and therefore of cardinality $n!$.

In general, being the algebra of a monoid helps a lot understanding the
representation theory. In this particular case, this is even more true since
the monoid has a very specific property: it is $\mathcal{J}$-trivial. Those
are the monoids which bears an order such that the product of $x$ and $y$ is
smaller than both $x$ and $y$. According to
Denton--Hivert--Schilling--Thiéry~\cite{DHST}, the representation theory of
these kinds of monoids is entirely combinatorial (see section~\ref{j-trivial}
for an overview of their properties). In particular, they showed that many of
the previous works about the representation theory of $\Hnz$ such
as~\cite{Norton.1979,Carter.1986,DuchampHivertThibon.2002,Fayers.2005} are
just particular cases of the general theory for $\mathcal{J}$-trivial monoids.

\subsection{Rook and \texorpdfstring{$q$}{q}-rooks}\label{sec_rooks}

In~\cite{Solomon.1990,Solomon.2004}, Solomon constructed an analogue of
Iwahori's construction replacing the general linear group by its full matrix
monoid $M \eqdef \mathbf{M}_n(\Fq)$. It goes as follows: recall that $B\subset M$
denotes the set of invertible upper triangular matrices. Then $M$ admits a
Bruhat decomposition~\cite{Renner.1995} too: the set of permutation matrices
is replaced by the set $R_n$ of so-called \emph{rook matrices} of size $n$,
that is a $n\times n$ matrices with entries $\lbrace 0,1\rbrace$ and at most
one nonzero entry in each row and column. Then
\begin{equation}
  M = \bigsqcup_{r\in R_n} BrB\,.
\end{equation}
The product of two rook matrices is still a rook matrix so that they form a
submonoid $R_n$ of $M$. For any $r \in R_n$, Solomon defined as in
Section~\ref{sec_Hecke} an element $T_r$ of the monoid algebra $\C[M]$ by
\begin{equation}
  T_r \eqdef \frac{1}{\vert B\vert} \sum_{x\in BrB} x.
\end{equation}
Those elements span a sub algebra $\mathcal{H}(M,B)$ which contains
$\mathcal{H}(G,B)$ with the same identity $\varepsilon$, and can also be
defined by $\mathcal{H}(M,B) = \varepsilon\, \C[M]\, \varepsilon$.

Halverson~\cite{Halverson.2004} further got a presentation of this ring. It is
generated by the two families $T_1, \dots, T_{n-1}$ and $P_1, \dots, P_n$
together with the relations of the Iwahori-Hecke algebra
(Equations~\ref{relH1}, \ref{relH2}, \ref{relH3}) and the following extra
relations:
\begin{alignat}{2}
  P_i^2  &= P_i                  &\qquad& 1\leq i \leq n,
  \tag{RH4}\label{relRH4}\\
  P_i P_j &= P_j P_i             &\qquad& 1\leq i,j \leq n,
  \tag{RH5}\label{relRH5}\\
  P_i T_j &= T_j P_i             &\qquad& i < j
  \tag{RH6}\label{relRH6}\\
  P_i T_j &= T_j P_i = q P_i     &\qquad& j<i
  \tag{RH7}\label{relRH7}\\
  P_{i+1} &= q P_i T_i^{-1} P_i   &\qquad& 1\leq i < n.
  \tag{RH8}\label{relRH8}
\end{alignat}
Note that the last relation can also be reformulated using the first as
\begin{equation}
P_{i+1} = P_iT_iP_i - (q-1) P_i
  \tag{RH8a}\label{relRH8a}
\end{equation}

The question whether there exists a proper degeneracy at $q=0$ of this ring
and if it exists, if it is the monoid-ring of a monoid, is therefore very
natural. The main goal of the present article is to construct such a monoid
denoted $\Rnz$, show that it is, as $\Hnz$, a $\JJ$-trivial monoid, which
allows us to analyze easily its representation theory.

\subsection{Outline of the paper}

The paper is organized as follows: in Section~\ref{sec-background}, after some
background on the rook matrices (or just \emph{rooks}) and their one-line
notations, we sketch out Denton-Hivert-Schilling-Thiéry work on
representation theory of $\JJ$-trivial monoids and how it applies to
$0$-Hecke monoids. We also briefly review Krob-Thibon's
work~\cite{KrobThibon.1997} linking representation theory of $0$-Hecke algebra
to the Hopf algebras of noncommutative symmetric functions and quasi-symmetric
functions.

In Section~\ref{sec-def}, we turn to the definition of the $0$-rook
monoid. We actually give two equivalent definitions:
The first definition is by generators and relations
(Subsection~\ref{sec-def-pres}): We show that a suitable rewriting of
Halverson's presentation when specialized at $q=0$ is actually a monoid
presentation~(Definition~\ref{def-rook_presentation}). We then study some
particular elements of this monoid which allows us to give a simpler
equivalent presentation~(Corollary~\ref{relRn0pi0}).

The second definition is as operators acting on the rook monoid
(Definition~\ref{def_Ro_fun}). To show that these two definitions are actually
equivalent (Corollary~\ref{egalitetouscardinaux}), we choose to go a somewhat
lengthy road, taking the following steps:
\begin{enumerate}
\item We first notice that the operators verify the relations of the
  presentation (Remark~\ref{presFnFn0}).
\item We generalize to rooks a variant of the notion of Lehmer code of
  permutations (Definition~\ref{def_encode}), building a bijection between
  rooks and the so-called $R$-code (Theorem~\ref{codedecode}).
\item After a little combinatorial detour (Section~\ref{subsubsec-count-rook}),
  we associate to each $R$-code $\sfc$, a canonical word
  $\pi_{\sfc}$ (Definition~\ref{def_word_code}) and its corresponding
  $s_{\sfc}$ in the classical rook monoid such that
  (Proposition~\ref{action1n}) for all rook $r \in R_n$ then $1_n\cdot
  \pi_{\code(r)} = 1_n\cdot s_{\code(r)}=r.$
\item We then translate on $R$-code $\sfc$ the action on rook
  (Definition~\ref{defactioncode}), and prove that, for any generator $t$, the
  element $\pi_\sfc t$ is equivalent to $\pi_{\sfc\cdot t}$ modulo the
  relations of the presentation (Theorem~\ref{actionsurcode}).
\item By induction this shows that any word is equivalent to a word
  $\pi_{\sfc}$, but since there are as many $R$-codes as rooks we will
  conclude that the two definitions are equivalent
  (Corollary~\ref{egalitetouscardinaux}).
\end{enumerate}
Note that we do not use the well-known presentation of the classical rook
monoid or of the $q$-rook algebra, but prove them again from scratch. Though
it is combinatorially technical, we argue that our approach has several
advantages. First it is self contained and purely monoidal, providing
arguments for monoid theory people which are not familiar with Coxeter group
theory.  Second, our approach is very explicit and algorithmic providing a
canonical reduced word for all rooks or $0$-rooks together with an explicit
algorithm transforming any word in its equivalent canonical one. Moreover, the
Lehmer code is central ingredient in the theory of Schubert polynomials whose
modern combinatorial incarnation is the pipedream theory. We find interesting
to provide such a combinatorial tool.  Finally, this allows us to have a much
finer understanding of the combinatorics of reduced words. In particular, we
get an analogue of Matsumoto's theorem (Theorem~\ref{theoreom-matsumoto}), an
ingredient which was noticed missing by Solomon~\cite{Solomon.2004}. As a
consequence, all the previous proof of presentation had to rely on some
dimension argument so that they were only valid on a field. Notice that, if we
had this theorem from the beginning, we could have worked only on reduced
words as for the classical case of Hecke algebras.  \bigskip

Section~\ref{sec-r-order} is devoted to the study of the analogue of the weak
permutohedron order on rooks or equivalently to Green's $\RR$-order of the
$0$-rook monoid. Using a generalization of the notion of inversion sets
(Definition~\ref{def-rook-triple}), we provide an algorithm to compare two
rooks (Definition~\ref{def-rook-order} and Theorem~\ref{ordreR0}). A very
important consequence in particular for the representation theory is that
$\Rnz$ is $\RR$-trivial, $\LL$-trivial and thus $\JJ$-trivial
(Corollary~\ref{Jtrivial}). We then show that the right order, as for
permutations, is actually a lattice (Corollary~\ref{cor-order-lattice}),
giving algorithms to compute the meet and the join (Theorem~\ref{meet-r-order}
and \ref{join-r-order}). We moreover provide a formula enumerating the meet
irreducible (Proposition~\ref{prop-meet-irred}), give a bijection for a
certain subposet with the subposet of singletons in the Tamari lattice
(Section~\ref{subsec:chains}) and conclude this section by some geometric
remarks.  \bigskip

Section~\ref{sec-rep-theo} deals with the representation theory of the $0$-rook
monoid. It heavily uses the fact that $\Rnz$ is $\JJ$-trivial through the
theory of Denton--Hivert--Schilling--Thiéry~\cite{DHST}. We describe the set
of idempotents and their lattice structure (Proposition~\ref{prop-Rn0-idemp}
and~\ref{prod_of_idemp}). As for any $\JJ$-trivial monoids, we show that the
simple modules are all $1$-dimensional (Theorem~\ref{theo-Rn0-simple}),
describe the indecomposable projective module as some kind of descent classes
(Theorem~\ref{theo-Rn0-projective}) and describe the quiver
(Theorem~\ref{theo-Rn0-quiver}). We then study how the representation theory
of $\Hnz$ and $\Rnz$ are related. The main result here is that the later is
projective on the former (Theorem~\ref{Rn_proj_Hn}). We moreover give the
decomposition functor (Theorem~\ref{decompo_H_R}).

Finally Section~\ref{sec-tower}, is devoted to the tower of monoids structure
on the sequence of $0$-rook monoids. Recall that
Bergeron-Li~\cite{BergeronLi.2009} gave some necessary condition to get Hopf
algebra structure on the Grothendieck groups generalizing the algebras of
symmetric~\cite{MacDonald.1995}, noncommutative symmetric and quasi-symmetric
functions~\cite{KrobThibon.1997}. This was the main motivation for this work,
but unfortunately, it does not work as nicely as expected. We present such an
associative structure but it does not fulfill all the requirement of
Bergeron-Li. In particular, $R_{m+n}^0$ is not projective over $R_{m}^0\times
R_{n}^0$. We nevertheless explicit some structure and in particular the
induction rule for simple modules (Theorem~\ref{ind_simple_tour}).

\subsection{Aknowledgment}

We thank Vincent Pilaud for numerous fruitful discussions and suggestions. We
also would like to thank Nicolas M. Thiéry, Jean-Christophe Novelli for
various suggestions about representation theory and combinatorics. We are
grateful to Jean-Yves Thibon who suggested the problem. We also would like to
thanks James Mitchell for his help using
\texttt{libsemigroup}~\cite{libsemigroup} setting up some advanced difficult
computations related to this work.

J.~Gay was founded by Fondation DIGITEO, project TRAGIC,
grant~\#2015-3181D. The computation where made using the
Sagemath~\cite{Sagemath} software. Development is supported by the
OpenDreamKit Horizon 2020 European Research Infrastructures project grant
\#676541.

\section{Background}\label{sec-background}

\subsection{Rook monoids}

We start by recalling some basic combinatorial facts about rooks.

\begin{definition} A rook matrix is a $n\times n$ matrix with entries $\lbrace
  0,1\rbrace$ and at most one nonzero entry in each row and column.
\end{definition}
Enumeration of rook matrices has received a considerable research effort in the past
(See \textit{e.g.}~\cite{riordan2002introduction, Butler_rooktheory} and the
references therein) and has recently be renewed by connection with
PASEP~\cite{JosuatVerges.2011}. The product of two rook matrices is still a
rook matrix. Thus the following definition:
\begin{definition}
  The rook monoid of size $n$ is the submonoid $R_n$ of the matrix monoid
  containing the rook matrices of size $n$.
\end{definition}
Identifying permutations with their matrices, we see that $\mathfrak{S}_n$ is
a submonoid of $R_n$. To deal with rook matrices, it is easier to have an analogue of
the so-called one line notation for permutations as in~\cite{CanRenner.2012}:
\begin{notation} We encode a rook matrix by its \emph{rook vector} (or just
  \emph{rook}) of size $n$ whose $i$-th coordinate is $0$ if there is no $1$
  in the $i$-th column of $r$, and the index of the row containing the $1$ in the
  $i$-th column otherwise.
\end{notation}
\begin{example}
Here are two matrices with their associated rook vector:
$$
\begin{array}{cc}
\begin{psmallmatrix}
0 & 0 & 0 & 0 & 1\\
0 & 0 & 1 & 0 & 0\\
0 & 0 & 0 & 1 & 0\\
0 & 1 & 0 & 0 & 0\\
0 & 0 & 0 & 0 & 0
\end{psmallmatrix}&\begin{psmallmatrix}
0 & 0 & 0 & 0 & 1\\
0 & 0 & 0 & 0 & 0\\
0 & 1 & 0 & 0 & 0\\
0 & 0 & 0 & 1 & 0\\
0 & 0 & 0 & 0 & 0
\end{psmallmatrix}\\
\begin{smallmatrix}0&4&2&3&1\end{smallmatrix} &
\begin{smallmatrix}0&3&0&4&1\end{smallmatrix}
\end{array}
$$
\end{example}

In the sequel, we identify rooks matrices and rook vectors and speak about
rooks when there is no ambiguity.
\begin{definition}\label{def_matr_Rn}
  In the monoid $R_n$, let $(s_i)_{i=1\dots n-1}$ denotes the rook matrices of
  the elementary transpositions $(i, i+1)$. Let $P_i$ also denote the
  diagonal $n\times n$ matrix with the $i$ first diagonal entries nul and
  the remaining one as $1$.
\end{definition}
For example with $n=4$, here are the matrices of $s_1,s_2,s_3,P_1,P_2,P_3,P_4$
and their associated vectors
\[
\begin{array}{ccccccc}
\begin{psmallmatrix}0&1&0&0\\1&0&0&0\\0&0&1&0\\0&0&0&1\end{psmallmatrix}&
\begin{psmallmatrix}1&0&0&0\\0&0&1&0\\0&1&0&0\\0&0&0&1\end{psmallmatrix}&
\begin{psmallmatrix}1&0&0&0\\0&1&0&0\\0&0&0&1\\0&0&1&0\end{psmallmatrix}&
\begin{psmallmatrix}0&0&0&0\\0&1&0&0\\0&0&1&0\\0&0&0&1\end{psmallmatrix}&
\begin{psmallmatrix}0&0&0&0\\0&0&0&0\\0&0&1&0\\0&0&0&1\end{psmallmatrix}&
\begin{psmallmatrix}0&0&0&0\\0&0&0&0\\0&0&0&0\\0&0&0&1\end{psmallmatrix}&
\begin{psmallmatrix}0&0&0&0\\0&0&0&0\\0&0&0&0\\0&0&0&0\end{psmallmatrix}\\
\begin{smallmatrix}2&1&3&4\end{smallmatrix}&
\begin{smallmatrix}1&3&2&4\end{smallmatrix}&
\begin{smallmatrix}1&2&4&3\end{smallmatrix}&
\begin{smallmatrix}0&2&3&4\end{smallmatrix}&
\begin{smallmatrix}0&0&3&4\end{smallmatrix}&
\begin{smallmatrix}0&0&0&4\end{smallmatrix}&
\begin{smallmatrix}0&0&0&0\end{smallmatrix}
\end{array}
\]
It is well-known that the $(s_i)_i$ generate the symmetric group as a the group
of permutation matrices and $(s_i)_i, P_1$ generate the rook monoid. We will
later give a presentation (Remark~\ref{presFnFn0}).

\subsection{\texorpdfstring{$\JJ$}{J}-trivial monoids}
\label{j-trivial}

We present here basic facts about monoids.  We refer to \cite{Pin.2010} or
\cite{Steinberg.2016} for more details. Through this paper, all monoids are
supposed to be \emph{finite}.
\medskip

Recall that the left (resp.\ right, resp.\ bi-sided) ideal of $\tMonoid$
generated by $x$ is the set $Mx\eqdef\{m x\mid m\in\tMonoid\}$ (resp.\
$xM\eqdef\{x m\mid m\in\tMonoid\}$, resp.\
$MxM\eqdef\{m x n\mid m, n\in\tMonoid\})$.  In 1951, Green~\cite{Green.1951}
introduced several preorders on monoids related to inclusion of ideals.  The
standard terminology is to write $\RR$ for right ideal, $\LL$ for left and
$\JJ$ for bi-sided. Let $\KK\in\lbrace\RR,\LL,\JJ\rbrace$ and $\tMonoid$ be a
monoid. For $x,y\in \tMonoid$, we write $x \leq_\KK y$ when the $\KK$-ideal
generated by $x$ is contained in the $\KK$-ideal generated by $y$.  For
example, if $\KK=\LL$, this means that $x \leq_\LL y$ if
$\tMonoid x\subseteq \tMonoid y$ or equivalently if $x=uy$ for some
$u\in \tMonoid$. These relations are clearly preorders (reflexive and
antisymmetric) and naturally give rise to equivalence relations denoted simply
by $\KK$: for example $x \, \LL \, y$ if and only if
$\tMonoid x = \tMonoid y$.
\begin{definition}
  A monoid $\tMonoid$ is called $\KK$-trivial if all $\KK$-classes are of
  cardinality one, that is if the $\KK$-preorder is antisymmetric and
  therefore an actual order.
  Specifically, $\tMonoid$ is $\JJ$-trivial if $\tMonoid x \tMonoid = \tMonoid
  y \tMonoid$ implies $x = y$.
\end{definition}
For the reader which is more familiar with Cayley graph, this means that the
$\JJ$-sided Cayley graphs has only trivial (i.e. singletons) strongly connected
components.  Examples of $\JJ$-trivial monoid of interest for this work
include the $0$-Hecke algebra for any Coxeter group~\cite{DHST}. Beware that
$1$ is the largest element of those (pre)-orders. This is the usual convention
in the semigroup community, but is the converse convention from the closely
related notions of left and right weak order in a Coxeter group.  \medskip

Finally, for finite monoids, $\RR, \LL$ and $\JJ$ are related as follows:
\begin{lemma}[\cite{Pin.2010} V. Theorem 1.9]\label{R_and_L_donc_J}
  A \emph{finite} monoid is $\JJ$-trivial if and only if it is both
  $\RR$-trivial and $\LL$-trivial.
\end{lemma}

\subsection{Representation theory of \texorpdfstring{$\JJ$}{J}-trivial monoids}

The representation theory of $\JJ$-trivial monoids has been well studied by
Denton, Hivert, Schilling and Thi\'ery~\cite{DHST}. It turns out that it is
combinatorial: more precisely, one can compute the simple, projective modules,
the Cartan matrix and even the quiver by computing only in the monoid, without
requiring linear combinations. For example, the representation theory of any
algebra $A$ is largely governed by its idempotents (elements such that
$e^2=e$). However, when dealing with a finite $\JJ$-trivial monoid $\tMonoid$,
it is sufficient to look for idempotents in the monoid $\tMonoid$
itself rather than in its monoid algebra $\C[\tMonoid]$.

In this subsection, $\tMonoid$ will always by a finite $\JJ$-trivial monoid
and we will denote by $\idempMon$ the set of idempotents of $\tMonoid$. They
parameterize the simple $M$-modules:
\begin{theorem}[{\cite[Proposition~3.1
    and~3.3]{DHST}}]\label{proposition.simple}
  There are as many as (isomorphism classes of) simple modules $S_e$ as
  idempotents $e\in\idempMon$, all of dimension $1$. Their structure is as
  follows: $S_e$ is spanned by some vector $\epsilon_e$ with the action of any
  $m\in\tMonoid$ given by
  \begin{equation}
    m\cdot\epsilon_e=
    \begin{cases}
      \epsilon_e &\text{ if $me=e$}, \\
      0 &\text{ otherwise.}
    \end{cases}
  \end{equation}
\end{theorem}
We now describe the structure of the radical. Given $x\in\tMonoid$, the
sequence $(x^i)_{i\in\N}$ is decreasing for the $\JJ$-order, therefore it must
eventually stabilize to an idempotent element which is usually denoted $x^\omega$.
\begin{theorem}[{\cite[Theorem~3.4 and~3.7]{DHST}}]
  Define a product $\star$ on $\idempMon$ by $x \star y \eqdef (x y)^\omega$. Then
  the restriction of $\leq_\JJ$ to $\idempMon$ is a lower semi-lattice such
  that $x \wedge_\JJ y = x \star y$ where $x \wedge_\JJ y$ is the meet of $x$
  and $y$. In particular, $(\idempMon, \star)$ is a commutative monoid.

  Moreover $(\C[\idempMon], \star)$ is isomorphic to
  $\C[\tMonoid]/\rad(\C[\tMonoid])$ and the mapping $\phi : x \mapsto
  x^\omega$ is the canonical algebra morphism associated to this quotient.
\end{theorem}
Finally, we also describe the projective module: Define
\begin{equation}
  \label{eq.rfix}
  \rfix(x) \eqdef \min \{e\in \idempMon \mid xe=x\},
  \qandq
  \lfix(x) \eqdef \min \{e\in \idempMon \mid ex=x\},
\end{equation}
the $\min$ being taken for the $\JJ$-order (which exists according to
\cite[Proposition~3.16]{DHST}).
\begin{theorem}[{\cite[Theorem 3.23]{DHST}}]
  \label{theorem.projective_modules}
  For any idempotent $e$ denote by $L(e) \eqdef Me$, and set
  \begin{equation}
    L_=(e) \eqdef \{x\in Me\mid \rfix{x} = e\}
    \qandq
    L_<(e) \eqdef \{x\in Me\mid \rfix{x} <_\LL e\} \, .
  \end{equation}
  Then, the projective module $P_e$ associated to $S_e$ is isomorphic to $\mathbb{K}
  L(e)/\mathbb{K} L_<(e)$. In particular, taking as basis the image of $L_=(e)$ in the
  quotient, the action of $m\in\tMonoid$ on $x\in L_=(e)$ is given by: $m\cdot
  x = mx$ if $\rfix(mx) = e$ and $0$ otherwise.
\end{theorem}
Of course the corresponding statement holds on the right. Then
\cite[Theorem~3.20]{DHST} further give a formula for the Cartan invariant
matrix: for $i,j\in \idempMon$ is given by:
\begin{equation}
c_{i,j}=|\{x\in \tMonoid \suchthat i=\lfix{x}\text{ and } j=\rfix{x} \}|.
\end{equation}

\subsection{Descent sets, compositions and ribbons}

Before applying the preceding theory to the $0$-Hecke monoid, we recall some
classical combinatorial ingredient: each subset $S$ of $\interv{1}{n-1}$ of
cardinality $p$ can be uniquely associated with a so called \emph{composition
  of $n$} of length $p+1$ that is a tuple $I \eqdef (i_1,\dots,i_{p+1})$ of
positive integers of sum $n$:
\begin{equation}
  \label{eq.set.to.comp}
  S=\{s_1 < s_2 < \dots < s_p\}
  \longmapsto \cset(S)\eqdef(s_1,s_2-s_1,s_3-s_2,\ldots,n-s_p)\,.
\end{equation}
The converse bijection, sending a composition to its \emph{descent set}, is
given by:
\begin{equation}
  \label{eq.comp.to.set}
  I = (i_1,\dots,i_p) \longmapsto
  \Des(I) = \{i_1+\dots+i_j \suchthat j=1,\dots,p-1\}\,.
\end{equation}
we write $I\compof n$ when $I$ is a composition of $n$ and write $\lon(I)$ the
length of $I$.  We will sometimes extend this definition to subsets $J\subset
\interv{0}{n-1}$ by prepending a $0$ to $\cset(S)$ when $0\in S$.

For instance, the composition $(3,1,2,1,2,2)\compof 11$ corresponds to the
subset $\{3,4,6,7,9\}$ of $\interv{0}{10}$ and $(0,3,4,1)\compof 8$
corresponds to the subset $\{0,3,7\}$ of $\interv{0}{7}$.

Compositions can be pictured as a ribbon diagram, that is, a set of rows
composed of square cells of respective lengths $i_j$, the first cell of each
row being attached under the last cell of the previous one. $I$ is called the
\emph{shape} of the ribbon diagram. Recall also that the \emph{descent set}
$\Des(\sigma)$ of a permutation $\sigma$ is the set of $i$ such that
$\sigma(i)>\sigma(i+1)$ (the \emph{descents} of $\sigma$), and the
\emph{(right) descent composition} $\cset(\sigma)$ of $\sigma$ is the unique
composition $I$ of $n$ such that $\Des(I)=\Des(\sigma)$, that is the shape of
a filled ribbon diagram whose row reading is $\sigma$ and whose rows are
increasing and columns decreasing.  For example, Figure~\ref{perm-compo} shows
that the descent composition of $(3,5,4,1,2,7,6)$ is $I=(2,1,3,1)$.
\begin{figure}[ht]
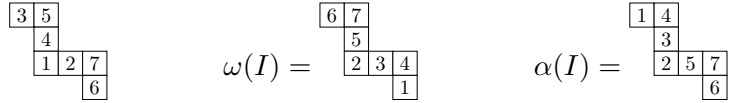

  \[
  \raisebox{20pt}{\scalebox{0.7}{$\gyoung(35,:;4,:;127,:::;6)$}}
  \qquad\qquad
  \omega(I)=\raisebox{20pt}{\scalebox{0.7}{$\gyoung(67,:;5,:;234,:::;1)$}}
  \qquad\qquad
  \alpha(I)=\raisebox{20pt}{\scalebox{0.7}{$\gyoung(14,:;3,:;257,:::;6)$}}
  \]

\caption{\label{perm-compo}The ribbon diagram of the permutation
$3541276$.}
\end{figure}

Conversely, with a composition $I$, associate its \emph{maximal permutation}
$\sigma=\omega(I)$ as the permutation with descent composition $I$ and maximal
inversion number.  Similarly, the minimal permutation $\alpha(I)$ is the
permutation with descent composition $I$ and minimal inversion number.  It is
well known that the set of permutations whose descent composition is $I$ is
the weak order right interval $[\alpha(I),\omega(I)]$ (see
\textit{e.g.}~\cite[Lemma 5.2]{KrobThibon.1997}).
For example, if $I=(2,1,3,1)$, $\omega(I)=6752341$ and $\alpha(I)=1432576$.

\subsection{Representation theory of \texorpdfstring{$0$}{0}-Hecke monoids and algebras}
\label{subsec-reptheo-hn0}

We now shortly explain how the previous theory applies to $H_n^0$. First of all
$H_n^0$ is $\RR$-trivial, the corresponding order being defined as
$\pi_\sigma\leq_\RR\pi_\mu$ if and only if $\mu\leq_R\sigma$ where $\leq_R$ is
the right weak order of the symmetric group. The same holds on the left, and
actually $H_n^0$ is isomorphic to its own opposite. Thanks to
Lemma~\ref{R_and_L_donc_J}, it is then $\JJ$-trivial.

For any composition $I = (i_1,\dots,i_p)\compof n$, we consider the parabolic
submonoid $H_I^0$ generated by $\{\pi_i \mid i \in \Des(I)\}$. It is
isomorphic to the direct product $H_{i_1}^0\times H_{i_2}^0 \times\dots\times
H_{i_p}^0$. Each parabolic submonoid contains a unique zero element
$\pi_J=\pi_{\omega_J}$ where $\omega_J$ is the maximal element of the
parabolic Coxeter subgroup $\SG{J}$. The collection $\{\pi_J \mid J \compof
n\}$ is exactly the set of idempotents in $H_n^0$.

Recall that the length $\ell(\sigma)$ of a permutation $\sigma$ is the minimal
length of a word in the $(s_i)_i$ whose product is $\sigma$. It is also equal to
the number of inversions of $\sigma$. Recall also that such a minimal length
word is called \emph{reduced}. The \emph{left and right descent sets} and
\emph{content} of $w\in \SG{n}$ are respectively defined by:
\begin{gather*}
  D_L(w) = \{ i\in I \mid \ell(s_iw) < \ell(w) \}, \qandq
  D_R(w) = \{ i\in I \mid \ell(w s_i) < \ell(w) \}, \\
  \cont(w) = \{i \in I \mid \text{$s_i$ appears in some reduced word for $w$} \},
\end{gather*}
Write $C_L$, $C_r$ and $\Cont$ the associated compositions.  In this last
condition ``some'' may be replaced by ``any''. Moreover, the above conditions
on $s_iw$ and $ws_i$ are respectively equivalent to $\pi_i \pi_w =\pi_w$ and
$\pi_w\pi_i = \pi_w$. One has $\cont(\pi_J)=D_L(\omega_J)$, or equivalently
$\cont(\pi_J)=D_R(\omega_J)$. Then, for any $\sigma\in\SG{n}$, we have
$\pi_\sigma^\omega=\pi_{\Cont(\sigma)}$,
$\lfix{\pi_\sigma}=\pi_{C_L(\sigma)}$, and
$\rfix{\pi_\sigma}=\pi_{C_R(\sigma)}$.

The left projective module $P_J$ corresponding to the idempotent $\pi_J$ has
its basis $b_w$ indexed by the elements of $w$ having $J$ as right descent
composition. The action of $\pi_i$ coincides with the usual left action,
except that $\pi_i\cdot b_w=0$ if $\pi\cdot w$ has a different right descent
composition than $w$.

\subsection{Induction and restriction of \texorpdfstring{$\Hnz$}{Hn0}-modules}

It is well known that character theory of the family of symmetric groups
$(\SG{n})_n$ can be encoded into symmetric functions via the Frobenius
isomorphism~\cite{MacDonald.1995}. Under this morphism, irreducible characters
$\chi_\lambda$ of $\SG{n}$ are mapped to Schur functions $s_\lambda$ of degree
$n$, induction and restriction along the natural inclusion
$\SG{m}\times\SG{n}\longrightarrow\SG{m+n}$ correspond respectively to
product and coproduct (the so called Littlewood-Richardson rule) of the Hopf
algebra $\Sym$ of symmetric function.

According to Krob-Thibon~\cite{KrobThibon.1997,Thibon.1998}, this construction
has an analogue for the $0$-Hecke monoids $(H_n^0)_n$. However, due to non semi-simplicity of
$H_n^0$, the situation is a little more complicated.
Note that the classical presentation deals with the algebra
$H_n(0)$ rather than the monoid. First of all, the maps
\begin{equation}
  \rho_{m,n}:\left\{
  \begin{array}{ccc}
    H_m^0\times H_n^0&\longrightarrow&H_{m+n}^0\\
    (\pi_i,\pi_j)&\longmapsto&\pi_i\pi_{j+m}=\pi_{j+m}\pi_i
  \end{array}\right.
\end{equation}
are injective monoid morphisms which moreover verify some associativity
conditions endowing $(H_n^0)_n$ with a tower of monoid structure
(see~\cite{BergeronLi.2009,Virmaux.2014} for a precise definition). One can
build two analogues of character rings, namely
$\mathcal{G}_0\eqdef\sum_{n}\C\mathcal{G}_0(H_n^0)$ the direct sum of the
(complexified) Grothendieck groups of $H_n^0$-modules on one hand, and
$\mathcal{K}_0\eqdef\sum_{n}\C\mathcal{K}_0(H_n^0)$ the direct sum of the
Grothendieck groups of projective $H_n^0$-modules on the other hand.  Recall
that $\mathcal{G}_0$ is the free $\Z$-module generated by simple module~$S_I$,
whereas $\mathcal{K}_0$ is the free $\Z$-module generated by the
indecomposable projective modules~$P_I$.

Now for two integers $m$ and $n$, we denote by $\Res_{m,n}$ the restriction
functor from the category of~$H_{m+n}^0$-modules to $H_m^0\times H_n^0$-modules along the
morphism $\rho_{m,n}$. It turns out that this defines proper co-products on
$\mathcal{G}_0$ and $\mathcal{K}_0$. In particular, $H_{m+n}^0$ is projective
over $H_m^0\times H_n^0$. Dually, the induction $\Ind_{m,n}$ defines products
on $\mathcal{G}_0$ and $\mathcal{K}_0$. These products and coproducts are
compatible giving the structure of a Hopf algebra. The analogue of Frobenius
isomorphism goes as follows: let $\QSym$ denote Gessel's~\cite{Gessel.1984}
Hopf algebra of quasi-symmetric functions, and $\NCSF$ denote the Hopf algebra
of noncommutative symmetric functions~\cite{NCSF}. Recall that these two dual
Hopf algebras have their bases indexed by compositions. Then the map
sending the simple module $S_I$ to the element $F_I$ of the fundamental basis
is a Hopf algebra morphism from $\mathcal{G}_0$ to $\QSym$. Dually, the map
sending the indecomposable projective module $P_I$ to the so-called ribbon
basis element $R_I$~\cite{NCSF,KrobThibon.1997} is a Hopf algebra morphism
from $\mathcal{K}_0$ to $\NCSF$. The duality between $\QSym$ and $\NCSF$
mirrors Frobenius duality between $\mathcal{G}_0$ and $\mathcal{K}_0$, the
commutative image $c:\NCSF\to\QSym$ being nothing but the Cartan map.

As an illustration, we give the induction rule of indecomposable projective
$H_n^0$-modules. For any two compositions $I\compof m$ and $J\compof n$:
\begin{equation}\label{induct_proj_hn0}
  \Ind_{m,n}(P_I\otimes P_J) \simeq P_{I \cdot J} \oplus P_{I \triangleright J}\,,
\end{equation}
where $I \cdot J$ is the concatenation of $I$ and $J$ and $I \triangleright J
\eqdef (i_1, \dots i_{k-1}, i_k+j_1, j_2, \dots j_\ell)$. For example,
$\Ind_{6,7}(P_{(3,1,2)}\otimes P_{(3,2,2)}) = P_{(3,1,2,3,2,2)}\oplus
P_{(3,1,5,2,2)}$. This is the same rule as the multiplication rule of the
ribbon basis of $\NCSF$~\cite{NCSF}.

As already said, the main motivation for the present paper was to understand
how this picture translate to rook monoids. Unfortunately, it turns out that
everything does not work as nicely as expected, but this may be because we
did not choose the right tower of monoids structure.

\section{The \texorpdfstring{$0$}{0}-rook monoid}\label{sec-def}
\subsection{Definition of \texorpdfstring{$\Rnz$}{R0} by generators and relations}\label{sec-def-pres}

To define the $0$-rook monoid, we take back Halverson's relations
(Equations~\ref{relH1} to \ref{relH3} and \ref{relRH4} to \ref{relRH8}) and
put $q=0$. In order to get a monoid, we write Equation~\ref{relRH8} as
\begin{equation}
  P_{i+1} = P_iT_iP_i + P_i = P_iT_iP_i+P_iP_i = P_i(T_i+1)P_i\,.
\end{equation}
Setting $\pi_i \eqdef T_i+1$, we finally obtain:
\begin{definition}\label{def-rook_presentation}
  We denote by $G_n^0$ the monoid generated by the two families
  $\pi_1, \dots, \pi_{n-1}$ and $P_1,\dots,P_n$ together with relations
  \begin{alignat}{2}
  \pi_i^2 &= \pi_i    &\qquad& 1\leq i \leq n-1,
  \tag{R1}\label{R1}\\
  \pi_i\pi_{i+1}\pi_i &= \pi_{i+1}\pi_i\pi_{i+1} &&  1\leq i \leq n-2,
  \tag{R2}\label{R2}\\
  \pi_i\pi_j &= \pi_j\pi_i                &&\vert i-j\vert \geq 2.
  \tag{R3}\label{R3}\\
  P_i^2  &= P_i                    &\qquad& 1\leq i \leq n,
  \tag{R4}\label{R4}\\
  P_i P_j &= P_j P_i               &\qquad& 1\leq i,j \leq n,
  \tag{R5}\label{R5}\\
  P_i \pi_j &= \pi_j P_i           &\qquad& i < j,
  \tag{R6}\label{R6}\\
  P_i \pi_j &= \pi_j P_i = P_i     &\qquad& j < i,
  \tag{R7}\label{R7}\\
  P_{i+1} &= P_i \pi_i P_i    &\qquad& 1\leq i < n.
  \tag{R8}\label{R8}
  \end{alignat}
\end{definition}
Using Relation~\ref{R8} we note that it is generated only by $\pi_1,
\dots, \pi_{n-1}$ and $P_1$.

\begin{notation}
  To state that two words are equal in $G_n^0$,
  we rather write explicitely that they are equivalent modulo the relations
  above as $e\eqz f$.
\end{notation}
We recall here the plan we introduced in the summary. Definition~\ref{def-rook_presentation} introduces a monoid defined by generators and relations. The $G$ stands for ``generators''. We will later give a definition of the monoid $F_n^0$ (Definition~\ref{def_Ro_fun}) as a monoid of operators acting on rooks ($F$ stands for ``functions'').  We will actually prove in Corollary~\ref{egalitetouscardinaux} that the two definitions actually coincide. We will then call this monoid the $0$-rook monoid, and denote it by $\Rnz$. 
\medskip

We start by focusing on the monoid generated by the $(P_i)$:
\begin{lemma}\label{prodPi}
  $P_i P_k \eqz P_{\max(i,k)}$.
\end{lemma}
\begin{proof}
  Thanks to Relation~\ref{R5}, we may assume that $k\geq i$. Relation~\ref{R8}
  shows us that there is a word for $P_k$ beginning with
  $P_i$. Relation~\ref{R4} says that $P_i$ is an idempotent.
\end{proof}

\begin{lemma}\label{Pn}
  The element $P_n$ is the unique zero of the monoid $G_n^0$, that is for any
  $e\in G_n^0$ then $eP_n \eqz P_ne \eqz P_n$. Furthermore $P_n$ have the two
  following expressions:
  \begin{equation}
    \begin{aligned}
      P_n
      & \eqz P_1 \pi_1P_1\pi_2\pi_1 P_1 \pi_3\pi_2\pi_1 P_1\dots P_1\pi_{n-1}\pi_{n-2}\dots\pi_1 P_1\\
      & \eqz P_1\pi_1\pi_2\dots\pi_{n-2}\pi_{n-1} P_1 \dots P_1 \pi_1\pi_2\pi_3 P_1 \pi_1\pi_2 P_1 \pi_1P_1.
    \end{aligned}
  \end{equation}
\end{lemma}
\begin{proof}
  We prove this by induction on $n\geq 1$. It is obvious that $P_2 \eqz
  P_1\pi_1P_1$ by Relation~\ref{R8}. To show that $P_2$ is a zero, it is
  enough to prove that the generators $\pi_1$ et $P_1$ stabilize it. It is
  clear for $P_1$ which is idempotent, and $\pi_1P_1\pi_1P_1 \eqz \pi_1 P_2
  \eqz P_2$ by the Relation~\ref{R7}.

  Assume that the result is proven for all $1\leq k \leq n$. Let us prove it
  for $n+1$:
  \begin{align*}
    P_{n+1} \eqz P_n \pi_n P_n & \eqz P_n \pi_n P_{n-1} \pi_{n-1}\pi_{n-2}\dots \pi_3\pi_2\pi_1 P_1 \text{ (by induction)}\\
    & \eqz P_n P_{n-1} \pi_n \pi_{n-1}\pi_{n-2}\dots \pi_3\pi_2\pi_1 P_1
    \text{ (by~\ref{R6})}\\
    & \eqz P_n \pi_n \pi_{n-1}\pi_{n-2}\dots \pi_3\pi_2\pi_1 P_1 \text{ (by
      Lemma \ref{prodPi}).}
  \end{align*}
  Thus the result holds. Since all the relations are symmetric, we get the
  other formula.

  To show that $P_{n+1}$ is a zero we prove that it is stabilized under
  multiplication by any generator among $\pi_1, \dots, \pi_n, P_1$. The
  stability by $P_1$ is obvious by Lemma \ref{prodPi}. For all the others, we
  deduce from Relation~\ref{R7} that $\pi_i P_n \eqz P_n$ since $i\leq n-1$.

  Finally, the uniqueness of the zero holds in any semigroup.
\end{proof}

\begin{corollary}\label{relRn0}
  In the presentation of $G_n^0$ one can replace the
  Relations~\ref{R4}, \ref{R5}, \ref{R6} and \ref{R7} by the following three
  and still get the same monoid:
  \begin{alignat}{2}
    P_1^2 &= P_1\,, &\quad&
  \tag{R4.1}\label{R4.1}\\
  P_1 \pi_j &= \pi_j P_1           && j \neq 1,
  \tag{R5.1}\label{R5.1}\\
  \pi_1P_1\pi_1P_1&=P_1\pi_1P_1=P_1\pi_1P_1\pi_1\,. &&
  \tag{R6.1}\label{R6.1}
  \end{alignat}
  In particular the monoid $G_n^0$ is generated by $(\pi_i)_{1\leq i \leq
    n-1}$ and $P_1$ subject to Relations~\ref{R1} to~\ref{R3}
  and \ref{R4.1} to \ref{R6.1}; Relation~\ref{R8} being seen as a
  definition for $P_i$ for $i>1$.
\end{corollary}

\begin{proof}
  Deducing Relations~\ref{R5.1} and \ref{R6.1} from
  Relations~\ref{R1} to \ref{R8} is obvious: Relation~\ref{R6.1} is only
  Relation~\ref{R7} applied with $i=2$ and $j=1$.

  Let us prove the converse: Relations~\ref{R1} to \ref{R8} can be deduced from
  Relations~\ref{R1} to \ref{R4}, \ref{R5.1}, \ref{R6.1} and \ref{R8} seen as
  a definition.
  We will now prove that Lemma \ref{prodPi} and Lemma \ref{Pn} (and Relation
  \ref{R4}) are still true. We prove simultaneously by induction on $n$ the
  following statements
  \begin{itemize}
  \item for all $k\leq n$, the element $P_k$ is given by the relation of Lemma
    \ref{Pn}.
  \item for all $i,k \leq n$, then $P_k^2 \eqz P_k$ and $P_iP_k\eqz P_{\max(i,k)}$.
  \end{itemize}
  The case $n=1$ is obvious with Relation~\ref{R4}.

  We now assume the statements for $n\geq 1$.  We only have to prove that two
  words for $P_{n+1}$ are given by Lemma \ref{Pn}, that $P_{n+1}^2 \eqz
  P_{n+1}$ and that $\forall i\leq n+1$, $P_{n+1}P_i \eqz P_{n+1}$.

  Regarding the words for $P_{n+1}$, a close look to the proof of Lemma
  \ref{Pn} shows that we use only Relation~\ref{R6.1} (for the basis step),
  Relation~\ref{R6} when $i<j\leq n$, Relation~\ref{R4} when $i\leq n$ and
  Lemma~\ref{prodPi} for $i,k\leq n$. But all these relations have already
  been proved by induction. Consequently we have the two expressions for
  $P_{n+1}$.

  From there, the relation $P_iP_{n+1} \eqz P_{n+1}P_i \eqz P_{n+1}$ for
  $i\leq n$ is clear using these words and the fact that $P_i^2 = P_i$. It
  remains only to prove that $P_{n+1}$ is idempotent.
  \begin{align*}
    P_{n+1}^2 & \eqz P_1\pi_1\pi_2\dots\pi_{n-1}\pi_{n} P_1 \dots P_1 \pi_1\pi_2 P_1 \pi_1P_1 \cdot P_1 \pi_1P_1\pi_2\pi_1 P_1 \dots P_1\pi_{n}\pi_{n-1}\dots\pi_1 P_1\\
    & \eqz P_1\pi_1\pi_2\dots \pi_{n-1}\pi_{n}P_1P_nP_n \pi_n\pi_{n-1}\dots
    \pi_2\pi_1P_1\\
    & \eqz P_1\pi_1\pi_2\dots \pi_{n-1}\pi_{n}P_1P_n \pi_n\pi_{n-1}\dots \pi_2\pi_1P_1\,,\\
  \intertext{by induction. Now using \ref{R3} and \ref{R5.1}:}
    P_{n+1}^2 & \eqz P_1\pi_1\pi_2\dots\pi_{n-1}\pi_{n}
    P_1\pi_1\pi_2\dots\pi_{n-1}\pi_{n} P_1 \dots P_1 \pi_1\pi_2\pi_3 P_1
    \pi_1\pi_2 P_1 \pi_1P_1\,.\tag{*}\label{Rstar}
  \end{align*}
  Now, call $\rho$, the first
  part of the previous calculation:
  \begin{alignat*}{2}
    \rho & \eqdef P_1\pi_1\pi_2\dots\pi_{n-1}\pi_{n}
    P_1\pi_1\pi_2\dots\pi_{n-1}\pi_{n}\,.
    \intertext{Then}
    \rho & \eqz P_1\pi_1\pi_2 P_1\pi_1\pi_2\dots\pi_{n-1}\pi_{n} \pi_2\pi_3\dots\pi_{n-2}\pi_{n-1} &\quad&\text{ (by \ref{R2}, \ref{R3} and \ref{R5.1})}\\
	& \eqz P_1\pi_1P_1 \pi_2 \pi_1\pi_2\dots\pi_{n-1}\pi_{n} \pi_2\pi_3\dots\pi_{n-2}\pi_{n-1} &&\text{ (by \ref{R5.1})}\\
	& \eqz P_1\pi_1P_1 \pi_1 \pi_2\pi_1\pi_3\dots\pi_{n-1}\pi_{n} \pi_2\pi_3\dots\pi_{n-2}\pi_{n-1} &&\text{ (by \ref{R2})}\\
	& \eqz P_1\pi_1P_1 \pi_1 \pi_2\pi_3\dots\pi_{n-1}\pi_{n} \pi_1\pi_2\pi_3\dots\pi_{n-2}\pi_{n-1} &&\text{ (by \ref{R3})}\\
	& \eqz P_1\pi_1P_1 \pi_2\pi_3\dots\pi_{n-1}\pi_{n} \pi_1\pi_2\pi_3\dots\pi_{n-2}\pi_{n-1} &&\text{ (by \ref{R6.1})}\\
	& \eqz P_1\pi_1\pi_2\dots\pi_{n-1}\pi_{n} P_1\pi_1\pi_2\dots\pi_{n-2}\pi_{n-1}  &&\text{ (by \ref{R5.1})}.
  \end{alignat*}
  Taking back Relation~(\ref{Rstar}) we thus have:
  \begin{multline*}
    P_{n+1}^2 \eqz P_1\pi_1\pi_2\dots\pi_{n-1}\pi_{n}\\
    P_1\pi_1\pi_2\dots\pi_{n-2}\pi_{n-1} P_1\pi_1\pi_2\dots\pi_{n-2}\pi_{n-1}P_1
    \dots P_1 \pi_1\pi_2\pi_3 P_1 \pi_1\pi_2 P_1 \pi_1P_1\,.
  \end{multline*}
  We recognize the end of the left term to be Equation~\ref{Rstar} for $n$ instead
  of $n+1$. Thus:
  \begin{equation*}
    P_{n+1}^2 \eqz P_1\pi_1\pi_2\dots\pi_{n-1}\pi_{n} P_n P_n \eqz
    P_1\pi_1\pi_2\dots\pi_{n-1}\pi_{n} P_n \eqz P_{n+1}\,.
  \end{equation*}
  Finally we have proved that the statement holds for $n+1$:
  indeed, we have thus Relations~\ref{R1} to \ref{R4} and the two
  Lemmas~\ref{prodPi} and \ref{Pn}. Relation \ref{R5} follows directly from
  Lemma~\ref{prodPi}, and Relation~\ref{R6} can be deduced from Lemma~\ref{Pn}
  using \ref{R5.1} and \ref{R3}.

  It remains to prove \ref{R7} using only \ref{R6.1} and Lemma \ref{Pn}. Since
  Lemma \ref{Pn} and all the relations are symmetric, we only need to show that
  $\pi_jP_i \eqz P_i$ for $j<i$, the proof of the other case could be
  conducted the same way.

  For $j=1$ and $i=2$ it is exactly Relation~\ref{R6.1}. For $j=1$ without
  condition on $i$, it comes from the fact that, because of Lemma~\ref{Pn}, a
  word for $P_i$ begin with $P_1\pi_1P_1$, and we conclude with~\ref{R6.1}.

  Otherwise, for $j\geq 2$ and $i>j$, we get:
\begin{align*}
  \pi_j P_i &
  \eqz \pi_j P_1 \pi_1P_1\pi_2\pi_1 P_1 \dots P_1 \pi_{j-1}\pi_{j-2}\dots \pi_2\pi_1 P_1 \pi_j\pi_{j-1}\dots \pi_2\pi_1 P_1\dots P_1\pi_{i-1}\pi_{i-2}\dots\pi_1 P_1 \\
  \intertext{with \ref{R3} and \ref{R5.1}:}
  & \eqz P_1 \pi_1P_1\pi_2\pi_1 P_1 \dots P_1 \pi_j \pi_{j-1}\pi_{j-2}\dots
  \pi_2\pi_1 P_1 \pi_j\pi_{j-1}\dots \pi_2\pi_1 P_1\dots
  P_1\pi_{i-1}\pi_{i-2}\dots\pi_1 P_1 \\
  & \eqz P_1 \pi_1P_1\pi_2\pi_1 P_1
    \dots P_1\pi_{i-1}\pi_{i-2}\dots\pi_1 P_1 = P_i \text{ (with $\rho$).}
\end{align*}
Hence the result.
\end{proof}

We finally get a new shorter presentation for $G_n^0$, by setting $\pi_0\eqdef P_1$.
\begin{corollary}\label{relRn0pi0}
  The monoid $G_n^0$ is generated by $\pi_0, \dots, \pi_{n-1}$ subject to
  the relations:
  \begin{alignat}{2}
  \pi_i^2 &= \pi_i    &\qquad& 0\leq i \leq n-1,
  \tag{RB1}\label{RB1}\\
  \pi_i\pi_{i+1}\pi_i &= \pi_{i+1}\pi_i\pi_{i+1} &&  1\leq i \leq n-2,
  \tag{RB2}\label{RB2}\\
  \pi_1\pi_0\pi_1\pi_0&=\pi_0\pi_1\pi_0=\pi_0\pi_1\pi_0\pi_1\,,
  \tag{RB3}\label{RB3}\\
  \pi_i\pi_j &= \pi_j\pi_i &&
  0\leq i,j \leq n-1,\qquad\vert i-j\vert \geq 2.
  \tag{RB4}\label{RB4}
  \end{alignat}
\end{corollary}
\begin{proof}
It is obvious from Corollary~\ref{relRn0} by letting $\pi_0 = P_1$.
\end{proof}

\begin{remark}\label{quotient_type_B}
  We can see that $G_n^0$ is a quotient of the Hecke monoid of type $B_n$ at
  $q=0$ (see~\cite{Fayers.2005} for a study of the representation theory of it).
\end{remark}

\subsection{Definition by action and \texorpdfstring{$R$}{R}-codes}
\label{sec-def-act}

The goal of this section is to construct a bijection between $R_n$ and $R_n^0$
which generalizes the bijection between $\SG{n}$ and $\Hnz$. In the case of
permutations, one argues using Matsumoto's theorem: recall that it says that two
reduced words (words of minimal length) generate the same permutation if and
only if they are congruent using only braid Relations~\ref{relM2}, \ref{relM3}
and not the quadratic one. Then, for a permutation $\sigma$, one defines
$\pi_\sigma \eqdef \pi_{i_1}\dots \pi_{i_k}$ where $s_{i_1}\dots s_{i_k}$ is any
reduced word for $\sigma$. Thanks to Matsumoto's theorem the result is
independent of the choice of the reduced word. One concludes that $H_n^0$ is
nothing but the set $\{\pi_\sigma\mid\sigma\in\SG{n}\}$ and therefore of
cardinality $n!$. The same argument is in fact valid for the algebras
$H_n(q)$ and is often found in this case in the literature.

Unfortunately, as noticed by Solomon~\cite[p.~209, bottom of the middle
paragraph]{Solomon.2004}, such a theorem is not known for the rook monoid. So
we choose a different path (see the discussion in the outline) effectively
ending up proving the generalization of Matsumoto's theorem. We introduce another monoid defined in term of a faithful action of
$R_n^0$ on $R_n$. It will turns out (Corollary~\ref{action_mulr_R0}) that this
action is nothing but the right multiplication.

\begin{definition}\label{def_Ro_fun}
  We denote $F_n^1$ the submonoid of the monoid of functions on $R_n$ generated by
  $s_1, \dots, s_{n-1}, P_1$ acting on $R_n$ by right multiplication of
  matrices. Namely, if $(r_1, \dots, r_n)$ is a rook then:
  \begin{align}
    \label{act_rook_sk}
    (r_1 \dots r_n)\cdot s_k &= r_1 r_2 \dots r_{k-1} r_{k+1} r_k r_{k+2} \dots r_n\,,\\
    \label{act_rook_P1}
    (r_1 \dots r_n)\cdot P_1 &= 0\, r_2  \dots r_n\,.
  \end{align}

  We denote $F_n^0$ the submonoid of the monoid of functions generated by
  $\pi_1, \dots, \pi_{n-1}, P_1$ acting on $R_n$ by the action:
  \begin{equation}
    \label{act_rook_pik}
    (r_1 \dots r_n)\cdot \pi_k \eqdef
      \begin{cases}
        (r_1 \dots r_n)\cdot s_k & \text{ if } r_k < r_{k+1}, \\
        (r_1 \dots r_n) & \text{ otherwise.}
      \end{cases}
    \end{equation}
\end{definition}
\begin{remark}\label{presFnFn0}
  A simple calculation shows that the generators of $F_n^0$ satisfy the
  Relations~\ref{R1} to~\ref{R3} and \ref{R4.1} to \ref{R6.1}. Similarly,
  the generators of $F_n^1$ satisfy
  \begin{alignat}{2}
    s_i^2 &= 1    &\qquad& 1\leq i \leq n-1,
  \tag{Rs1}\label{Rs1}\\
  s_is_{i+1}s_i &= s_{i+1}s_is_{i+1} &&  1\leq i \leq n-2,
  \tag{Rs2}\label{Rs2}\\
  s_is_j &= s_js_i                &&\vert i-j\vert \geq 2.
  \tag{Rs3}\label{Rs3}\\
  P_1^2 &= P_1, &\quad&
  \tag{Rs4.1}\label{Rs4.1}\\
  P_1 s_j &= s_j P_1           && j \neq 1
  \tag{Rs5.1}\label{Rs5.1}\\
  s_1P_1s_1P_1&=P_1s_1P_1=P_1s_1P_1s_1. &&
  \tag{Rs6.1}\label{Rs6.1}
\end{alignat}
We denote by $G_n^1$ the monoid generated by $\lbrace s_1, \dots, s_{n-1},
P_1\rbrace$ with the relations above.  We can rephrase Remark~\ref{presFnFn0}
as follows: there are two surjective morphisms of monoids:
\begin{equation}
  \Phi_1 : G_n^1 \twoheadrightarrow F_n^1
  \qquad\text{ and }\qquad
 \Phi_0 : G_n^0 \twoheadrightarrow F_n^0.
\end{equation}
Furthermore, these two morphisms give us an action of $G_n^1$ and $G_n^0$ over
$R_n$.
\end{remark}
\begin{remark}\label{kill-other-0}
  The map $(r_1r_2) \mapsto(r_10)$ is equal to the composition $s_1P_1s_1$ and
  therefore belongs to $F_2^1$. However, it can be checked that it does not
  belong to $F_2^0$, neither to its algebra $\C[F_2^0]$. More generally, in
  $F_n^0$, for any subset $I\subset\interv{1}{n}$ which is not of the form
  $\interv{1}{k}$ the maps replacing the letter in position $i$ by $0$,
  does not belong to $F_n^0$ or $\C[F_n^0]$.
\end{remark}
Our goal is now to show that $\Phi_1$ and $\Phi_0$ are actually isomorphisms.

\subsubsection{\texorpdfstring{$R$}{R}-code and rooks}\label{sec_code}

In this subsection, we build a combinatorial tool, namely the \emph{$R$-code},
which allows us to define for any rook a canonical reduced word. A classical
way to do that for permutations is to proceed by induction along the chain of
inclusions
$\SG1\subset\SG2\subset\dots\subset\SG{n-1}\subset\SG{n}\subset\dots$ noticing
that the number of cosets in $\SG{n-1} \backslash \SG{n}$ is exactly $n$. One
can for example take $\{1, s_{n-1}, s_{n-1}s_{n-2},
s_{n-1}s_{n-2}s_{n-3},\dots\}$ as a cross-section. In a more combinatorial
setting, this is equivalent to say that given a permutation
$\sigma\in\SG{n-1}$ there are exactly $n$ permutations which give back
$\sigma$ when erasing the letter $n$. Therefore any permutation can be encoded
by a sequence $\sfc=(\sfc_1\dots\sfc_n)$ satisfying $0\leq\sfc_i< i$. This can
be done by the \emph{Lehmer code}~(\cite[Page~330]{Lothaire.2002}) of the
permutation, or a variant of thereof.  See Remark~\ref{remark-compare-codes}
for a definition of the Lehmer code and how it relates to our generalized
$R$-code.  \medskip

The case of rooks is more involved because some times $n$ does not appear in
the rook vector and to go from $R_n$ to $R_{n-1}$ one has to erase a $0$. It
turns out that the right choice to minimize the number of moves (since we are
looking for a reduced word) is to remove the \emph{first} $0$. However, this
means that, given a rook $r$ of size $n-1$, the number of rook of size $n$
which give back $r$ depends on $r$ and more precisely on the position of its
first $0$.  We now unravel the corresponding combinatorics, starting with some
notations:
\begin{notation}[Word and Letter]
  The length of a word $w$ is denoted by $\ell(w)$.  The empty word (the only
  word of length $0$) will be denoted by $\emptyw$. When we need to
  distinguish between words and letters (for example when matching a word), we
  use the convention that words will be underlined as in $\un{w}$, while $i$
  will rather be a single letter. If the letter $i\in \Z$ appears in the word
  $\un{w}$ we write it $i\in \un{w}$; it means for example that $\un{w}$ can
  be written as $\un{w}=\un{a}i\un{b}$.
\end{notation}
\begin{definition}\label{def_encode}
  For a rook $r$ of length $n$, we call the \emph{code of $r$} and denote
  $\code(r)$ the word on $\Z$ of length $n$ defined recursively by:
  \begin{enumerate}
  \item If $n=0$ then $\code(\emptyw) \eqdef \emptyw$.
  \item Otherwise, if $n\in r$, then $r$ can be written uniquely
    $r=\un{b}n\un{e}$. Let $r'\eqdef\un{b}\un{e}$ (the subword of
    $r$ where the unique occurrence of $n$ is removed). Then $\code(r)
    \eqdef \code(r')\cdot (\ell(\un{b})+1)$.
  \item Otherwise, $n\notin r$, and therefore $r$ can be written uniquely
    $r=\un{b}0\un{e}$ with $0\notin \un{b}$. Let $r'\eqdef\un{b}\un{e}$ (the
    subword of $r$ where the first $0$ is removed). Then $\code(r) \eqdef
    \code(r')\cdot (-\ell(\un{b}))$.
  \end{enumerate}
\end{definition}
\begin{notation}
  When writing a code, $\ov{i}$ stands for $-i$ for $i\in \N$.
\end{notation}
\begin{example}
Let $r= 02401$. Then:
$$\code(02401) = \code(2401)\mathsf{0} = \code(201)\mathsf{20 }=
\code(21)\mathsf{\ov{1}20} = \code(1)\mathsf{1\ov{1}20} =
\mathsf{11\ov{1}20}.$$
An easy remark is that $r$ is a permutation if and only if its code contains
only positive letters.
\end{example}

\begin{remark}\label{remark-compare-codes}
  Recall that the Lehmer code~\cite[Page~330]{Lothaire.2002} of a permutation is
  defined by
  \begin{equation}
    \Lehmer(\sigma)=\sfc_1\dots \sfc_n
    \qquad\text{with}\qquad
    \sfc_i \eqdef |\{ j > i \mid \sigma(i) > \sigma(j) \}|\,.
  \end{equation}
  When $r$ is actually a permutation $\sigma$, the codes are related as
  follows: write the code as $\code(\sigma) = r_1\dots r_n$ and the Lehmer
  code as $\Lehmer(\sigma)=\sfc_1\dots \sfc_n$. Then
  $\sfc_i=\sigma(i)-r_{\sigma(i)}$. For example taking $\sigma = 516432$, then
  $\code(\sigma)=122213$ and $\Lehmer(\sigma)=403210$.
\end{remark}

We now describe a subset $C_n$ of $\Z^n$ that we call the set of $R$-codes. We
will see in Proposition~\ref{PZcode} and Theorem~\ref{codedecode} that it is
exactly the set of codes of a rook.
\begin{definition}\label{def_r_code}
  To each word $\un{w}$ over $\Z$, we associate a nonnegative number
  $m(\un{w})$ defined recursively by: $m(\emptyw) = 0$ and for any word $\un{w}$
  and any letter $d$,
  \begin{equation}
    m(\un{w} d) \eqdef
    \begin{cases}
      -d  & \text{if $d\leq0$\,,} \\
      m(\un{w})+1 & \text{if $0<d\leq m(\un{w})+1$\,,} \\
      m(\un{w})   & \text{if $d> m(\un{w})+1$\,.}
    \end{cases}
  \end{equation}

  A word on $\Z$ is an \emph{$R$-code} if it can be obtained by the following
  recursive construction: the empty word $\emptyw$ is a code, and $\un{w}d$ is
  a code if $\un{w}$ is a code and $-m(\un{w}) \leq d \leq n$. We denote by
  $C_n$ the set of $R$-codes of size $n$.
\end{definition}

\begin{notation}
  In order to make the difference between the rook $1234$ and the code
  $\mathsf{1234}$, we make the convention to write codes in sans-serif
  font.
\end{notation}

\begin{example}
  $m(\mathsf{12836427}) = 5$: there is no negative letter, thus it only
  increments on integers $1$, $2$, $3$, $4$ and $2$ in this order.
  $m(\mathsf{364\ov{4}294\ov{3}52538}) = 6$. Indeed, the last negative letter
  is $-3$, thus $m(\mathsf{364\ov{4}294\ov{3}}) = 3$ and it increments on
  letters $2$, $5$ and $3$ in this order. Similarly, $m(\mathsf{021\ov{1}1254}) = 4$.
\end{example}

\begin{example} Here are the first $R$-codes:
    $C_1=\{\sf 0, 1\}, C_2=\{\sf 00, 01, 02, 1\ov{1}, 10, 11, 12\}$ and
    \begin{multline*}
      C_3=\{\sf 000, 001, 002, 003, 01\ov{1}, 010, 011, 012, 013, 020, 021,
        022, 023,1\ov{1}\ov{1}, 1\ov{1}0, 1\ov{1}1, 1\ov{1}2, \\
        \qquad\sf 1\ov{1}3, 100, 101, 102, 103, 11\ov{2}, 11\ov{1}, 110, 111, 112,
        113, 12\ov{2}, 12\ov{1}, 120, 121, 122, 123\}\,.
    \end{multline*}
    The $R$-codes of $C_9$ with prefix $\mathsf{021\ov{1}1254}$ are
    ${\sf021\ov{1}1254\ov{4}}$ , $\mathsf{021\ov{1}1254\ov{3}}$ , $\dots$,
    ${\sf021\ov{1}12549}$.
\end{example}
\begin{remark}
  If $\sfc\in C_n$, then necessarily we have $m(\sfc) \leq \ell(\sfc)$.
\end{remark}

\begin{definition}
  We note $\PZ$ (standing for First Zero) the function defined for any rook
  $r=r_1\dots r_n$ by
  \begin{equation}
    \PZ (r) \eqdef \min\{ j\leq n \, \vert \, r_j = 0 \} - 1 \,,
  \end{equation}
  with the convention that if there is no zero among the $r_j$ (that is $r$ is in
  fact a permutation), we set $\PZ(r)=n$.
\end{definition}
We now show that $\code$ is a bijection between $R$-codes and rook vectors
of the same length.
\begin{proposition}\label{PZcode}
  If $r\in R_n$ then $\code(r) \in C_n$ and $\PZ(r) = m(\code(r)).$
\end{proposition}
\begin{proof}
  We show the result by induction on $n$: it is trivial for $n = 0$. We now
  show the induction step, assuming that it holds for $n-1$. Let $r \in
  R_n$. Let us first prove the case $n \in r$. We then write $r =
  \un{b}n\un{e}$ and $r' = \un{b}\un{e}$. By induction $\code(r') \in C_{n-1}$
  and $\code(r) = \code(r')\cdot (\ell(\un{b})+1)$ with $(\ell(\un{b})+1) \in
  \interv{1}{n} \subset \interv{-m(\code(r'))}{n}$ so
  that $r \in R_n$.

  The only remaining case is $n\notin r$. We write $r = \un{b}0\un{e}$ with
  $0\notin \un{b}$, $r' = \un{b}\un{e}$. By induction $\code(r') \in C_{n-1}$
  and $\code(r) = \code(r') \cdot -\ell(\un{b})$. By definition of $\PZ$ we
  have $\ell(\un{b}) = \PZ(r')$, and $\PZ(r') = m(\code(r'))$ by induction. So
  $-\ell(\un{b}) \in \interv{ -m(\code(r')) }{0}\subset
  \interv{-m(\code(r'))}{n}$ and so $r \in R_n$.

  We have proven the first part of the statement in every case. Let us now
  focus on the second part. First of all, if $0\notin r$, then $r$ is a
  permutation and its code $\sfc_1\dots\sfc_n$ is such that $0<\sfc_i\leq i$. As a
  consequence $m(\code(r)) = n = \PZ(r)$.

  We finally need to prove that when $0\in r$ then $\PZ(r)=m(\code(r))$,
  knowing by induction that $\PZ(r') = m(\code(r'))$.  We distinguish 
  the two nontrivial cases:
  \begin{itemize}
  \item[$\bullet$] If $n \in r$ then $r = \un{b}n\un{e}$ and
    $r' = \un{b}\un{e}$. The number of $0$ of $r$ is the same that~$r'$. We have
    two possibilities:
    \begin{itemize}
    \item If $0 \notin \un{b}$ then the first zero of $r'$ is in
      $\un{e}$. Thus $\PZ(r) = \PZ(r')+1$. But also $\code(r) = \code(r')\cdot
      (\ell(\un{b}) +1)$ with $\ell(\un{b})+1 \leq m(\code(r')) = \PZ(r')$. So,
      by definition of $m$, $m(\code(r)) = m(\code(r'))+1$. Hence the
      equality.
    \item If $0 \in \un{b}$ then $\PZ(r) = \PZ(r')$. Furthermore $m(\code(r))
      = m(\code(r'))$ by definition of $m$. So that we get $\PZ(r) = \PZ(r') =
      m(\code(r')) = m(\code(r))$.
    \end{itemize}
  \item[$\bullet$] If $n \notin r$, then $r = \un{b}0\un{e}$ with $0\notin \un{b}$,
    $r' = \un{b}\un{e}$ and $\code(r) = \code(r')\cdot -\ell(\un{b})$. Since
    ${0 \notin \un{b}}$ we have $\PZ(r) = \ell(\un{b})$. We write $\code(r) =
    \sfc_1\dots\sfc_n$ then $\PZ(r) = -\sfc_n$ by definition of
    $\code$. Furthermore $m(\code(r)) = -\sfc_n$ so that
    $\PZ(r)=m(\code(r))$. \qedhere
\end{itemize}
\end{proof}

We now define a candidate for the converse bijection.
\begin{definition}\label{def_decode}
  For $\sfc=\sfc_1\dots \sfc_n\in C_n$, we define inductively a vector
  $\decode(\sfc)$ as follows: first, set $\decode(\mathsf{\emptyw}) \eqdef \emptyw$.
  Then, let $r' \eqdef \decode(\sfc_1\dots \sfc_{n-1})$. If $\sfc_n$ is
  nonnegative, insert the letter $n$ in $r'$ at the position
  $\sfc_n$. Otherwise insert $0$ at $-\sfc_n+1$.
\end{definition}
\begin{proposition}
If $\sfc \in C_n$ then $\decode(\sfc) \in R_n$. 
\end{proposition}
\begin{proof}
  It is clear that we get a rook, since only $0$ can be repeated. The size is also clear. 
\end{proof}

\begin{example}
  Let $\sfc = \mathsf{11\ov{1}20}$. Then $\decode(\mathsf{1})=1$.
  $\decode(\mathsf{11}) = 21$.
  $\decode(\mathsf{11\ov{1}}) =201$.
  $\decode(\mathsf{11\ov{1}2}) = 2401$.
  Finally  $\decode(\mathsf{11\ov{1}20}) = 02401$.
\end{example}

\begin{proposition}\label{PZdecode}
  Let $\sfc = \sfc_1\dots \sfc_n \in C_n$. Then $\PZ(\decode(\sfc)) =
  m(\sfc)$. In particular, if $\sfc_n \leq 0$, $\PZ(\decode(\sfc))=-\sfc_n$.
\end{proposition}
\begin{proof}
  We prove it by induction on $n$. The assertion is clear for words of length
  $0$.  Otherwise, assume that we have proved the result for all words
  of length strictly less than $n$. Let $\sfb \eqdef \sfc_1\dots \sfc_{n-1}$.
  \begin{itemize}
  \item[$\bullet$] If $\sfc_n > 0$: by induction $\PZ(\decode(\sfb)) =
    m(\sfb)$. But $\PZ(\decode(\sfc)) = \PZ(\decode(\sfb))+1$ if $\sfc_n \leq
    \PZ(\decode(\sfb)) +1$ and $\PZ(\decode(\sfc)) = \PZ(\decode(\sfb))$
    otherwise. By definition of function $m$ we get $\PZ(\decode(\sfc)) =
    m(\sfc)$.
  \item[$\bullet$] If $\sfc_n \leq 0$ we have two possibilities:
    \begin{itemize}
    \item If $\forall i \leq n-1$, $\sfc_i > 0$ then $0 \notin \decode(\sfb)$
      by definition, and so $\decode(\sfc)$ has a single zero which is the one
      inserted between $\decode(\sfb)$ and $\decode(\sfc)$, and is thus at
      position $(-\sfc_n +1) -1 = m(\sfc)$.
    \item Otherwise, by induction $m(\sfb) = \PZ(\decode(\sfb))$. By definition of $m$, $m(\sfc) = -\sfc_n$. By definition of $R$-codes we get $-\sfc_n \leq m(\sfb) = \PZ(\decode(\sfb))$. Thus the zero inserted at position $-\sfc_n+1$ is left to the former first zero. \\
      Finally $\PZ(\decode(\sfc)) = -\sfc_n = m(\sfc)$.  \qedhere
    \end{itemize}
  \end{itemize}
\end{proof}

\begin{theorem}\label{codedecode}
  The functions $\code$ and $\decode$ are inverse one from the other: for all
  $\sfc \in C_n$ and $r \in R_n$ then
  \begin{equation}
    \code(\decode(\sfc)) = \sfc
    \qandq
    \decode(\code(r)) = r.
  \end{equation}
\end{theorem}
\begin{proof}
  We proceed by induction on the size $n$ of $r$ and $\sfc$.  The result is
  clear if $n = 0$. Assume now that we have proved the result up to $n-1$. We
  begin with rooks. Let $r\in R_n$.
  \begin{itemize}
  \item If $n \in r$, write $r = \un{b}n\un{e}$ and $r' =
    \un{b}\un{e}$ with $\decode(\code(r')) = r'$ by induction.
    Since ${\code(r) = \code(r')
      \cdot (\ell(\un{b})+1)}$, $\code(r)$ is the word $\code(r')$ with the
    position of $n$ as final letter. Since $\decode(\code(r))$ inserts in
    $\decode(\code(r')) = r'$ the $n$ at this position, we have the result.
  \item Otherwise $\code(r)$ is the word $\code(r')$ with at the end the
    opposite of the position minus $1$ of the first zero of $r$. But
    $\decode(\code(r))$ insert a zero in $\decode(\code(r')) = r'$ at this
    position.
  \end{itemize}
  We now do the proof for $R$-codes in a similar way: Let
  $\sfc = \sfc_1\dots \sfc_n \in C_n$ and $\sfc' = \sfc_1\dots \sfc_{n-1}$,
  and assume that $\code(\decode(\sfc')) = \sfc'$.
  \begin{itemize}
  \item If ${c_n} > 0$ then $\decode(\sfc)$ inserts in $\decode(\sfc')$ a
    letter $n$ at position ${c_n}$. Computing further $\code(\decode(\sfc))$
    adds at the end of $\code(\decode(\sfc')) = \sfc'$ this position.
  \item Otherwise, $\decode(\sfc)$ insert in $\decode(\sfc')$ a letter $0$ in
    position $-{c_n}+1$. Since it is the first zero of $\decode(\sfc)$ by
    Proposition \ref{PZdecode}, $\code(\decode(\sfc))$ add ${c_n}$ at the end
    of ${\code(\decode(\sfc')) = \sfc'}$. \qedhere
  \end{itemize}
\end{proof}
In particular, there are as many $R$-codes of size $n$ as rooks:
\begin{corollary}\label{CnRn}
For all $n$: $\vert C_n \vert = \vert R_n \vert$.
\end{corollary}

\subsubsection{Counting rook according to the position of the first \texorpdfstring{$0$}{0}}
\label{subsubsec-count-rook}

This subsection is a little detour through enumerative combinatorics and
permutations statistics.  It is interesting to count rooks of size $n$
according to the position of the first zero. We denote
$R(n,k)\eqdef\{r\in R_n\mid\PZ(r)=k \}$ and $r(n,k)\eqdef|R(n,k)|$. Here are the first
values:
\begin{equation*}
\begin{array}{r|rrrrrrrr}
n/k & 0 & 1 & 2 & 3 & 4 & 5 & 6 & 7\\\hline
0 & 1 \\
1 & 1&     1 \\
2 & 3&     2&     2 \\
3 & 13&    9&     6&     6 \\
4 & 73&    52&    36&    24&    24 \\
5 & 501&   365&   260&   180&   120&   120 \\
6 & 4051&  3006&  2190&  1560&  1080&  720&   720 \\
7 & 37633& 28357& 21042& 15330& 10920& 7560& 5040& 5040 \\
\end{array}
\end{equation*}
For example, here are the rooks of size $2$ sorted according to their first
zero:
\[
R(2,0)=\{00, 01, 02\},\qquad
R(2,1)=\{10, 20\},\qquad
R(2,2)=\{12, 21\}
\]
\begin{lemma}\label{lemma-count-fz}
  The sequence $r(n,k)$ verifies the following recurrence relation for $n>0$:
  \begin{equation}
    r(n,k) = k\,r(n-1, k-1) + (n-k-1)r(n-1,k) + \sum_{i=k}^{n} r(n-1,i)\,,
  \end{equation}
  with the convention that $r(n,k)=0$ if $k<0$ or $k>n$.
\end{lemma}
\begin{proof}
  To get the set of rooks of size $n$ from the set of rooks of size $n-1$, one
  has either to insert $n$ or to insert a $0$. To make sure to get each rook
  only once, one has to insert $0$ only before the first zero. According to
  the definition of $\PZ$, in what follows, positions are counted starting with
  $0$. Then
  \begin{itemize}
  \item $k\cdot r(n-1, k-1)$ is the number of rooks where $n$ is (and
    therefore was inserted) before position $\PZ$.
  \item $k(n-k-1)r(n-1,k)$ is the number of rooks where $n$ is after the first
    $0$.
  \item $\sum_{i=k}^{n} r(n-1,i)$ is the number of rooks where $n$ does not
    appear. They are obtained by inserting a $0$ in position $k$, in a rook
    $r$ such that $i \eqdef \PZ(r) \geq k$.  \qedhere
  \end{itemize}
\end{proof}
One recognizes the triangle A206703 of~\cite{OEIS}. It is defined as the number
$C(n,k)$ of the injective partial function on $\interv{1}{n}$ where the union the cycle supports has cardinality $k$. Recall that a rook vector $r=(r_1,\dots,r_n)$
can been seen as an injective partial function by setting $r(i)=r_i$ if $r_i\neq 0$ and
$r(i)$ is undefined otherwise. We consider the generalization of the notion of
cycle of permutations to rooks (See~\cite[Example II.21, page
132]{FlajoletSedgewick.2009}), this combinatorics was studied in details
in~\cite{GanMaz.2006}): the sequence of the iterated images $(r^n(i))_{n\in\N}$
of some integer $i$ under $r$ can have one of the two following behaviors:
\begin{itemize}
\item Either for some $n\geq 1$ one has $r^n(i)=i$ (the sequence must be
  periodic and not only ultimately periodic because of injectivity). We say
  that $i$ belongs to a \emph{cycle} of $r$.
\item Or starting from some $n\geq 1$ the iterated image $r^n(i)$ stops being
  defined; we say that $i$ belongs to a \emph{chain} of $r$.
\end{itemize}
Rooks can therefore be decomposed as two sets: the set of its cycles
(counting fixed points) and the set of its maximal chains, that is maximal
finite sequences $(c_1,\dots,c_k)$ such that $r(c_i)=c_{i+1}$ if $i<k$ and
undefined otherwise. Clearly, the \emph{supports} of the cycles and the chains
of the rook $r$ form a partition of $\interv{1}{n}$.

\begin{example}\label{expl_cycles_chains}
  Consider the rook vector $r=205109706$, it corresponds to the function
  \[\left(\begin{array}{ccccccccc}
      1&2&3&4&5&6&7&8&9\\
      2&\bot&5&1&\bot&9&7&\bot&6\\
    \end{array}\right),\]
  where $\bot$ means undefined. It has two cycles $(6,9)$ and $(7)$ and three
  maximal chains $(4,1,2)$, $(3,5)$ and $(8)$.
\end{example}
\begin{proposition}
   Let $C(n,k)$ be the set of rooks of size $n$ where the union of cycle supports 
  has cardinality $k$, and denote by $c(n,k)$ its cardinality. Then $c(n,k)=r(n,k)$ 
  for all $k$ and $n$.
\end{proposition}
We show here the rooks of size $2$ sorted according to their number of points
in a cycle:
\[
C(2,0)=\{00, 01, 20\}\,,\qquad
C(2,1)=\{10, 02\}\,,\qquad
C(2,2)=\{12, 21\}\,.
\]
\begin{proof}
  We define a bijection $\Phi$ from $C(n,k)$ to $R(n,k)$. It is an adaptation
  of Foata fundamental transformation (See~\cite[Chapter.~10]{Lothaire.2002}). For
  $r\in C(n,k)$, write its cycles starting from the smallest elements and sort
  the set of cycles according to their smallest element in decreasing
  order. By concatenating those words one obtains a first word
  $\CyW(r)$. Second, write the maximal chain backward replacing the last
  element of the chain (now the first of the word) by a $0$ and sort the
  chains according to their last element in increasing order. By concatenating
  those words one obtains a second word $\ChW(r)$. Now define
  $\Phi(r)\eqdef\CyW(r)\cdot\ChW(r)$. Then $\Phi(r)$ is a rook of size $n$
  whose first zero is in position~$k$, so that $\Phi(r)\in R(n,k)$.

  We now explain how to recover $r$ from $s\eqdef\Phi(r)$, that is the converse
  bijection: cut $s$ at the places just before the zeros replacing those
  zeros by the values missing in $s$ in increasing order. The various words
  obtained except the first one are the (reversed) chains of $r$. One recover
  the cycle of $r$ by cutting the first word before the \emph{lower records}
  (elements that are only preceded by larger ones) and interpret each part as a
  cycle. Knowing all the chains and cycles of $r$ is sufficient to recover
  $r$.
\end{proof}
\begin{example} We get back to Example~\ref{expl_cycles_chains}.  The rook
  vector $r=205109706$ has cycles $(6,9)$ and $(7)$ and chains $(4,1,2)$,
  $(3,5)$ and $(8)$. Therefore $\CyW(r)=769$ and $\ChW(r)=014030$, so that
  $\Phi(r)=769014030$.

  To demonstrate the computation of the inverse, we start with
  $769014030$. The missing numbers are $\{2,5,8\}$. Replacing the zeros by
  them and cutting gives $769|214|53|8$. So that we already got the chains
  $(4,1,2)$, $(3,5)$ and $(8)$. Now the word $769$ is cut as $7|69$ recovering
  the cycles.
\end{example}
Using the so-called symbolic method (See~\cite[Example II.21, page
132]{FlajoletSedgewick.2009}), the decomposition by cycles and chains shows
that the generating series is given by
\begin{equation}
  \sum_{n,k} r(n,k)\,\frac{x^ny^k}{n!} = \frac{\exp(x/(1-x))}{1-xy}\,.
\end{equation}

\subsection{Equivalence of the definitions of \texorpdfstring{$\Rnz$}{Rn0}}
\label{sec_action_Rn}

We now get back to the $0$-rook monoid.
Thanks to the previously defined $R$-code, we are now in position to define
the canonical reduced word $\pi_{\sfc}$ associated to a $R$-code and thus to a
rook. To define $\pi_{\sfc}$, the following notation is handy:
\begin{notation}
  For $i, n \in \mathbb{N}$ we write (with $\pi_0 \eqdef P_1$):
  $$ \col{n}{i} \eqdef \begin{cases}
    1                                     & \text{if } i>n, \\
    \pi_n\dots \pi_i                      & \text{if } 0\leq i\leq n,\\
    \pi_n\dots \pi_1\pi_0\pi_1\dots \pi_i & \text{if } i<0,
  \end{cases}
  \text{ and } \quad
  \cols{n}{i} \eqdef \begin{cases}
    1                                     & \text{if } i>n, \\
    s_n\dots s_i                          & \text{if } 0\leq i\leq n,\\
    s_n\dots s_1\pi_0 s_1\dots s_i        & \text{if } i<0.
  \end{cases}$$
\end{notation}
\textit{A priori} $\col{n}{i}\in G_n^0$ and $\cols{n}{i}\in G_n^1$. Using
$\Phi_0$ and $\Phi_1$ of Remark \ref{presFnFn0} we will sometimes see them as
elements of $F_n^0$ or $F_n^1$.

\begin{definition}\label{def_word_code}
  For any $R$-code $\sfc = \sfc_1\dots \sfc_n \in C_n$, we define $\pi_\sfc
  \in G_n^0$ and $s_\sfc \in G_n^1$ by
  \begin{equation}
    \pi_{\sfc} \eqdef
    \col{0}{\sfc_1}\cdot\col{1}{\sfc_2}\cdot\dots\cdot\col{n-1}{\sfc_n}\,,
    \qandq
    s_{\sfc} \eqdef
    \cols{0}{\sfc_1}\cdot\cols{1}{\sfc_2}\cdot\dots\cdot\cols{n-1}{\sfc_n}\,.
  \end{equation}
\end{definition}

\begin{example}
  Let $\sfc = \mathsf{11\ov{1}20}$. Then:
  $$
  \pi_\sfc = \col{0}{1}\cdot \col{1}{1}\cdot \col{2}{-1}\cdot \col{3}{2} \cdot
  \col{4}{0} = 1\cdot\pi_1 \cdot \pi_2\pi_1 \pi_0 \pi_1 \cdot \pi_3\pi_2 \cdot
  \pi_4\pi_3\pi_2\pi_1\pi_0\,.
  $$
  Going further, let us show how $\pi_\sfc$ acts on the identity rook $12345$:
  \begin{multline*}
    12345 \cdot \pi_c = 12345 \cdot
    1\cdot\pi_1\cdot\col{2}{-1}\cdot\col{3}{2}\cdot\col{4}{0} =
    21\mathbf{3}45\cdot\pi_2\pi_1 \pi_0 \pi_1\cdot\col{3}{2}\cdot\col{4}{0} = \\
    2\mathbf{3}145\cdot\pi_1 \pi_0 \pi_1\cdot\col{3}{2}\cdot\col{4}{0} =
    \mathbf{3}2145\cdot\pi_0 \pi_1\cdot\col{3}{2}\cdot\col{4}{0}=
    \mathbf{0}2145\cdot\pi_1\cdot\col{3}{2}\cdot\col{4}{0} = \\
    2\mathbf{0}145\cdot\pi_3\pi_2\cdot\col{4}{0} =
    24015\cdot\pi_4\pi_3\pi_2\pi_1\pi_0 =
    02401 = \decode(\sfc)\,.
  \end{multline*}
  We see that the $i$-th column of $\pi_\sfc$ places the letter $i$ (or the
  corresponding zero), at its place, effectively decoding $\sfc$. This is
  actually a general fact and it is also true replacing $\pi_i$ by $s_i$:
\end{example}
\begin{proposition}\label{action1n}
  If $r \in R_n$ then $1_n\cdot \pi_{\code(r)} = 1_n\cdot
  s_{\code(r)} = r.$
\end{proposition}
\begin{proof}
  We will prove it by induction on $n$. It is evident for $n = 0$.
  Assume that we have proved the result up to step $n-1$, and let $r \in R_n$.

  If $n \in r$ then $r$ writes $r=\un{b}n\un{e}$, $r' = \un{b}\un{e}$ and
  $\code(r) = \code(r')\cdot (\ell(\un{b}) +1)$.  By definition we have
  $\pi_{\code(r)} = \pi_{\code(r')}
  \col{n}{\ell(\un{b})+1}$. 
  By induction $1_{n-1}\cdot \pi_{\code(r')} = r'$. So $1_{n}\cdot
  \pi_{\code(r')} = r'n = \un{be}n$, since $\pi_{\code(r')}$ only acts on the
  first $n-1$ coordinates. Since $0 < \ell(\un{b})+1 \leq n$, a direct
  calculation gives us $\un{be}n\cdot \col{n}{\ell(\un{b})+1} = \un{b}n\un{e}
  = r$. So $1_n\cdot \pi_{\code(r)} = r.$

  Otherwise $n \notin r$. Then $r$ writes $r=\un{b}0\un{e}$ with $0\notin
  \un{b}$, $r' = \un{be}$ and $\code(r) = \code(r')\cdot -\ell(\un{b})$. We
  get in the exact same way $1_{n}\cdot \pi_{\code(r')} = r'n =
  \un{be}n$. Since $-\ell(\un{b}) \leq 0$, a simple calculation gives us
  $\un{be}n \cdot \col{n}{-\ell(\un{b})} = \un{b}0\un{e} = r$. So $1_n\cdot
  \pi_{\code(r)} = r.$

  The same proof works \emph{mutatis mutandis} for $s$.
\end{proof}

\begin{corollary}\label{plusdeRn0}
  For all $n$, $\vert G_n^0\vert \geq \vert F_n^0\vert \geq \vert R_n\vert =
  \vert C_n\vert$ et $\vert G_n^1\vert \geq \vert F_n^1\vert \geq \vert
  R_n\vert = \vert C_n\vert$.
\end{corollary}
\begin{proof}
  All the functions $\pi_{\code(r)}$ and $s_{\code(r)}$ for $r\in R_n$ are
  distinct since they have a distinct action on identity $1_n$. We conclude
  with Corollary \ref{CnRn} and Remark \ref{presFnFn0}.
\end{proof}

The next step is to transfer on $R$-codes the action on rooks:
\begin{definition}\label{defactioncode}
  For $\sfc= \sfc_1\dots \sfc_n \in C_n$ and
  $t\in\lbrace\pi_0,\pi_1,\dots,\pi_{n-1}\rbrace\subset G_n^0$ we define
  $\sfc\cdot t$ recursively the following way:
  \medskip

  \noindent $\bullet$\ If $n= 1$ and $t = \pi_0$ then $\sfc\cdot t \eqdef \mathsf{0}$.
  \medskip

  \noindent Otherwise we proceed by induction depending on the sign
  of ${c_n}$ and the value of $t$:
  \begin{itemize}
  \item[Pos.] If $\sfc_n=i\geq 1$:
    \begin{enumerate}[a.]
    \item If $t = \pi_i$ then $\sfc\cdot t \eqdef \sfc$.
    \item If $t = \pi_{i-1}$ then $\sfc\cdot t \eqdef \sfc_1 \dots \sfc_{n-1}
      (\sfc_n-1)$.
    \item If $t = \pi_j$ with $j<i-1$ then $\sfc\cdot t \eqdef \left[(\sfc_1\dots
        \sfc_{n-1})\cdot \pi_j\right] \sfc_n$.
    \item If $t = \pi_j$ with $j>i$ then $\sfc\cdot t \eqdef
      \left[(\sfc_1\dots\sfc_{n-1})\cdot \pi_{j-1}\right] \sfc_n$.
    \end{enumerate}
  \item[Neg.] If $\sfc_n = -i \leq 0$:
    \begin{enumerate}[a.]
    \item If $t = \pi_i$ then $\sfc\cdot t \eqdef \sfc$.
    \item If $t = \pi_j$ with $0<j<i$ then $\sfc\cdot t \eqdef
      \left[(\sfc_1\dots \sfc_{n-1})\cdot \pi_j\right] \sfc_n$.
    \item If $t = \pi_j$ with $j > i+1$ then $\sfc\cdot t \eqdef
      \left[(\sfc_1\dots \sfc_{n-1})\cdot \pi_{j-1}\right] \sfc_n$.
    \item If $t = \pi_0$ then $\sfc\cdot t \eqdef
      \left[(\sfc_1\dots \sfc_{n-1})\cdot \pi_0\dots \pi_{i-1}\right]
      \mathsf{0}$. (In particular $\sfc\cdot t = \sfc$ if $i = 0$.)
    \item If $t = \pi_{i+1}$ (thus $i\neq n$) we have two possibilities:
      \begin{itemize}
      \item[$\alpha$.] If $m(\sfc_1\dots \sfc_{n-1}) = i$ then
        $\sfc\cdot t \eqdef\sfc$.
      \item[$\beta$.] Otherwise $\sfc\cdot t \eqdef \sfc_1\dots \sfc_{n-1}\
        \ov{i+1}$.
      \end{itemize}
    \end{enumerate}
  \end{itemize}
\end{definition}

\begin{lemma}\label{codebiendef}
  For any code $\sfc= \sfc_1\dots \sfc_n \in C_n$ and generator
  $t\in\lbrace\pi_0,\pi_1,\dots,\pi_{n-1}\rbrace\subset G_n^0$, then
  $\sfc\cdot t$ is a code of size $n$.
\end{lemma}
\begin{proof}
  We will prove the result by induction on $n$, and we will prove along the
  way that $m(\sfc\cdot t) \geq m(\sfc)$ if $t\neq \pi_0$. It is evident if $n=1$.

  For all subcases of case Pos.\ of Definition~\ref{defactioncode} it is evident that we get a code by induction
  since the last value is positive which do not lead to difficulties (we add
  to $\sfc_1\dots \sfc_{n-1}$ either $c_n$ or $c_{n}-1$). The property of
  function $m$ is clear for subcase a.\ In b.\ if $i-1\neq 0$ then ${c_n-1} > 0$
  so $m(\sfc_1 \dots \sfc_{n-1} (\sfc_n-1)) \geq m(\sfc)$. In c. the induction
  gives us $m((\sfc_1\dots \sfc_{n-1})\cdot \pi_j) \geq m(\sfc_1\dots
  \sfc_{n-1})$ and we conclude with the definition of $m$ to get
  $m(\left[(\sfc_1\dots \sfc_{n-1})\cdot \pi_j\right] \sfc_n) \geq
  m(\sfc_1\dots \sfc_{n-1}\sfc_n)$ (we do the same for d.).

  The subcase Neg.a. is clear. We prove subcases Neg.b. and Neg.c. using the
  induction on the condition of $m$ and the fact that in these two subcases
  $m(\sfc\cdot t) = \sfc_n = m(\sfc)$. The subcase Neg.d. is clear by induction
  (we do not have to prove the condition of $m$ here), as subcase Neg.e.$\alpha$.
  The subcase Neg.e.$\beta$ remains, whose condition gives us
  $m(\sfc_1\dots \sfc_{n-1}) > i$ (since $\sfc \in C_n$) so $\sfc\cdot t \in
  C_n$ and $m(\sfc\cdot t) = i+1 > m(\sfc) = i$.
\end{proof}

It therefore makes sense to apply the $\decode$ algorithm to $\sfc\cdot
t$. The crucial fact that motivated the definition of the action on a code is
that, forall $R$-code $\sfc$
\begin{equation}
  \decode(\sfc\cdot t) = \decode(\sfc)\cdot t\,.
\end{equation}
We could prove this fact right away, by a tedious explicit calculation,
distinguishing all cases. We urge the reader who want to understand the
motivation of Definition~\ref{defactioncode} to do so. For example, in case
Neg.e.$\alpha$, the assumption that $m(\sfc_1\dots \sfc_{n-1}) = i = -\sfc_n$
shows that, using Proposition~\ref{PZdecode}, $\PZ(\decode(\sfc_1\dots
\sfc_{n-1})) = i$. Therefore $\decode(\sfc_1\dots \sfc_{n-1})$ is of the
form
$$
\decode(\sfc_1\dots \sfc_{n-1}) = r_1\dots r_{i}0r_{i+2}\dots r_{n-1}n\,,
$$
where none of the $r_j$ for $j\leq i$ vanish. Decoding further, since
$\sfc_n=-i$, on finds that
$$
\decode(\sfc_1\dots \sfc_{n})=r_1\dots r_{i}00r_{i+2}\dots r_{n-1}\,.
$$
So that, $\decode(\sfc)\cdot\pi_{i+1}=\decode(\sfc)$. That's why, in case Neg.e.$\alpha$, we defined $\sfc\cdot\pi_{i+1}\eqdef\sfc$. Instead of doing the
proof in all other cases, we will get the properties as a corollary of the
much stronger fact that $\pi_{\sfc\cdot t} \eqz \pi_\sfc t$ using the
morphism $\Phi_0 : G_n^0 \twoheadrightarrow F_n^0$.
\bigskip

We turn now to the proof of that later statement. It will use intensively
the following technical lemma:
\begin{lemma}\label{lemtechnmultcolonne}
  If $i>0$, $k<0$ and $j<i-1$ we have the following identities:
  \begin{equation}
    \pi_j \col{i}{k} = \col{i}{k}\pi_j\ \text{ if }\  0<j<\vert k\vert
    \qandq
    \pi_j \col{i}{k} = \col{i}{k}\pi_{j+1}\ \text{ if }\ j>\vert k\vert .
  \end{equation}
  In particular, by immediate induction:
  \begin{equation}
    \col{j}{l}\cdot \col{i}{k} = \col{i}{k} \cdot \col{j}{l}
    \quad\text{ if }  0<l\leq j<\min(i, \vert k\vert).
  \end{equation}
\end{lemma}
\begin{proof}
  We will only use relations (\ref{RB1} to \ref{RB4}) of
  Remark~\ref{presFnFn0} written according to Corollary~\ref{relRn0pi0}.  For
  the first equality we just apply successively in this order \ref{RB4},
  \ref{RB2}, \ref{RB4}, \ref{RB2} and \ref{RB4}. For the second we only apply
  \ref{RB4}, \ref{RB2} and \ref{RB4}.
\end{proof}
We may now proceed to the main theorem of this section:
\begin{theorem}\label{actionsurcode}
  For a code $\sfc= \sfc_1\dots \sfc_n \in C_n$ and a generator $t\in \lbrace
  \pi_0, \pi_1,\dots , \pi_{n-1}\rbrace\subset G_n^0$, the congruence
  $\pi_{\sfc\cdot t} \eqz \pi_\sfc t$ holds. Furthermore $\ell(\pi_{\sfc\cdot
    t})\leq \ell(\pi_\sfc)+1$.
\end{theorem}
\begin{proof}
  We will only use the relations of the proof of Lemma
  \ref{lemtechnmultcolonne}. We then prove the theorem by induction on $n$
  depending on ${c_n}$ and $t$. The remark on the length can be checked
  systematically in all the cases, we left it to the reader.

  If $n= 1$ and $t = \pi_0$ then $\sfc\cdot t = 0$. Then $\pi_{\sfc\cdot t} =
  \pi_0 = \pi_c t$ by~\ref{RB1}.

  Otherwise we write $\sfc' \eqdef \sfc_1\dots \sfc_{n-1}$ and we recall that $\pi_{\sfc} =
  \pi_{\sfc'}\col{n-1}{\sfc_n}$.
\begin{itemize}
\item[Pos.] $\sfc_n = i \geq 1$
\begin{itemize}
\item[a.] If $t = \pi_i$ then $\sfc\cdot t = \sfc$. Then
  $\pi_{\sfc'}\col{n-1}{\sfc_n}t = \pi_{\sfc'}\pi_{n-1}\dots \pi_i \pi_i \eqz
  \pi_{\sfc'}\col{n-1}{\sfc_n}$ by~\ref{RB1}.
\item[b.] If $t = \pi_{i-1}$ then $\sfc\cdot t = \sfc'(\sfc_n-1)$. The
  relation is just $\pi_{\sfc'}\col{n-1}{\sfc_n}\pi_{i-1} =
  \pi_{\sfc'}\col{n-1}{\sfc_{n}-1}$.
\item[c.] If $t = \pi_j$ with $j<i-1$ then $\sfc\cdot t = (\sfc'\cdot \pi_j)
  \sfc_n$. Then
  \begin{equation*}
    \pi_\sfc \, t = \pi_{c'}\col{n-1}{\sfc_n}\pi_j
    \eqz \pi_{\sfc'}\pi_j \col{n-1}{\sfc_n}
    \eqz \pi_{\sfc' \cdot \pi_j} \col{n-1}{\sfc_n}
    = \pi_{(\sfc'\cdot \pi_j)\sfc_n} = \pi_{\sfc\cdot t}.
  \end{equation*}
  Indeed, the first congruency is Lemma~\ref{lemtechnmultcolonne}, and the
  second holds by induction.
\item[d.] If $t = \pi_j$ with $j>i$ then $\sfc\cdot t = (\sfc'\cdot \pi_{j-1})
  \sfc_n$. We do the same than in Pos.c. using this time Relation~\ref{RB2} and
  Relation~\ref{RB4}.
\end{itemize}
\item[Neg.] $\sfc_n = -i \leq 0$
\begin{itemize}
\item[a.] If $t = \pi_i$ we do the same than in Pos.a. with~\ref{RB1}.
\item[b.] If $t = \pi_j$ with $0<j<i$ we do the same than in Pos.c. with~\ref{RB4}.
\item[c.] If $t = \pi_j$ with $j > i+1$ we do the same than in
  Pos.d. with~\ref{RB2} and \ref{RB4}.
\item[d.] If $t = \pi_0$ ($i\neq 0$) then $\sfc\cdot t = \left[(\sfc_1\dots \sfc_{n-1})\cdot \pi_0\dots \pi_{i-1}\right] 0$. Furthermore
\begin{alignat*}{2}
  \col{n-1}{\sfc_n}\pi_0 &
  =_{\phantom{0}}
  \pi_{n-1}\dots \pi_2\pi_1\pi_0\pi_1\pi_2\dots \pi_i \pi_0 \\
  & \eqz \pi_{n-1}\dots \pi_2\pi_1\pi_0 \pi_1\pi_0 \pi_2\dots \pi_i
  &\qquad&\text{ by~\ref{RB4}} \\
  & \eqz \pi_{n-1}\dots \pi_2 \pi_0\pi_1\pi_0 \pi_2\dots \pi_i
  &&\text{ by~\ref{RB3}} \\
  & \eqz \pi_0 \pi_{n-1}\dots \pi_2\pi_1\pi_0 \pi_2\dots \pi_i
  = \pi_0 \col{n-1}{0} \pi_2\dots \pi_i
  &&\text{ by~\ref{RB4}}.
\end{alignat*}
Now using iteratively Lemma~\ref{lemtechnmultcolonne}, one gets
\begin{equation}
  \pi_0 \col{n-1}{0} \pi_2\dots \pi_i
  \eqz \pi_0\pi_1\col{n-1}{0}\pi_3\dots \pi_i \eqz\cdots
  \eqz \pi_0\dots \pi_{i-1} \col{n-1}{0}.
\end{equation}
Thus $\pi_\sfc\pi_0\ \eqz\ \pi_{\sfc'}\,
(\pi_0\dots\pi_{i-1}) \col{n-1}{0}\ \eqz\ \pi_{\sfc'\cdot (\pi_0\dots
  \pi_{i-1})0}\ =\ \pi_{\sfc\cdot \pi_0}$.
\item[e.] If $t = \pi_{i+1}$ (so $i\neq n$) we have two possibilities:
\begin{itemize}
\item[$\alpha$.] Either $m(\sfc_1\dots \sfc_{n-1}) = i$;
\item[$\beta$.] Or $m(\sfc_1\dots \sfc_{n-1}) \neq i$.
  In this second case $\sfc\cdot t = \sfc'\,\ov{i+1}$, and we
  proceeds as in case Pos.b.
\end{itemize}
\end{itemize}
\end{itemize}
The last remaining case is then $\sfc_n = -i \leq 0$ with $t = \pi_{i+1}$ and
$m(\sfc_1\dots\sfc_{n-1})=i$. In this case we have $\sfc\cdot t=\sfc$.

Let $k$ be the index of the last non-positive $c_k\leq0$.  Since, by hypothesis,
$m(\sfc_1\dots \sfc_{n-1}) = i$, there are $i-|\sfc_k|=i+\sfc_k$ further indexes where
the value of $m$ increase, we write them as $k < j_1 < \dots < j_{i+\sfc_k} <
n$. In other words, these are the steps of the inductive construction of
$\decode(\sfc)$ where the value of $\PZ$ change. For each such index $j_u$,
we split the columns of the corresponding decoded word into two parts as
\begin{equation}\label{eq_split_col}
  \col{j_u-1}{\sfc_{j_u}} = \col{j_u-1}{|\sfc_k|+u+1}\col{|\sfc_k|+u}{\sfc_{j_u}}\,.
\end{equation}
For the other indexes not belonging to the $j_u$, we consider them as first
parts, leaving their second parts empty. Thanks to
Lemma~\ref{lemtechnmultcolonne}, all the second parts commute with the first
parts on their right so that:
\begin{align*}
  \pi_{\sfc}t
  & =_{\phantom{0}}
  \pi_{\sfc_1\dots \sfc_{k-1}}\col[red]{k-1}{\sfc_k}
  \dots\col[blue]{j_1-1}{\sfc_{j_1}}
  \dots\col[cyanp]{j_2-1}{\sfc_{j_2}}
  \dots \col[cyan]{j_{i+\sfc_k}-1}{\sfc_{j_{i+\sfc_k}}}
  \dots \col[olive]{n-1}{-i} \textcolor{olive}{\pi_{i+1}} \\
  & =_{\phantom{0}}
  \pi_{\sfc_1\dots \sfc_{k-1}}\col[red]{k-1}{\sfc_k}
  \dots\col[blue]{j_1-1}{\vert\sfc_k\vert+2}
  \col[blue]{\vert\sfc_k\vert + 1}{\sfc_{j_1}}
  \dots\col[cyanp]{j_2-1}{\vert\sfc_k\vert+3}
  \col[cyanp]{\vert\sfc_k\vert+2}{\sfc_{j_2}}
  \dots\col[cyan]{j_{i+\sfc_k}-1}{|\sfc_k|+i+\sfc_k+1}
  \col[cyan]{|\sfc_k|+i+\sfc_k}{\sfc_{j_{i+c_k}}}
  \dots \col[olive]{n-1}{-i-1} \\
  & \eqz \pi_{\sfc_1\dots \sfc_{k-1}} \col[red]{k-1}{\sfc_k}
  \dots\col[blue]{j_1-1}{\vert\sfc_k\vert+2}
  \dots\col[cyanp]{j_2-1}{\vert\sfc_k\vert+3}
  \dots\col[cyan]{j_{i+\sfc_k}-1}{i+1}
  \dots\col[olive]{n-1}{-i-1}
  \cdot\col[blue]{\vert\sfc_k\vert + 1}{\sfc_{j_1}}
  \col[cyanp]{\vert\sfc_k\vert+2}{\sfc_{j_2}}
  \dots\col[cyan]{i}{\sfc_{j_{i+c_k}}}\,.
\end{align*}
We similarly further split the column $\col[red]{k-1}{\sfc_k}$ into its
negative and positive part, and commute the negative part as
\begin{equation*}
  \eqz \pi_{\sfc_1\dots \sfc_{k-1}}\col[red]{k-1}{1}
  \dots\col[blue]{j_1-1}{\vert\sfc_k\vert+2}
  \dots\col[cyan]{j_{i+\sfc_k}-1}{i+1}
  \ \dots\
  \textcolor{red}{\pi_0\pi_1\dots \pi_{\vert \sfc_k\vert }}
  \col[olive]{n-1}{-i-1}
  \col[blue]{\vert \sfc_k\vert + 1 }{\sfc_{j_1}}
  \dots \col[cyan]{i}{\sfc_{j_{i+ c_k}}}\,.
\end{equation*}
We now focus on the product of the the second parts which we call
$S$. Using~\ref{RB4}, and striping the second parts from their topmost
element, we get:
\begin{align*}
  S & \eqdef
  \textcolor{red}{\pi_0\pi_1\dots \pi_{\vert \sfc_k\vert }}
  \col[olive]{n-1}{-i-1}
  \col[blue]{\vert \sfc_k\vert + 1 }{\sfc_{j_1}}
  \dots \col[cyan]{i}{\sfc_{j_{i+ c_k}}} \\
  & \eqz
  \textcolor{red}{\pi_0}\col[olive]{n-1}{-i-1}
  \textcolor{red}{\pi_1\dots \pi_{\vert \sfc_k\vert }}
  \textcolor{blue}{\pi_{\vert \sfc_k\vert +1}}
  \dots\textcolor{cyan}{\pi_i}
  \col[blue]{\vert \sfc_k\vert }{\sfc_{j_1}}
  \dots\col[cyan]{i-1}{\sfc_{j_{i+c_k}}}\\
  & \eqz
  \textcolor{red}{\pi_0}
  \col[olive]{n-1}{2}
  \textcolor{olive}{\pi_1 \pi_0 \pi_1 \dots \pi_i \pi_{i+1}}
  \textcolor{red}{\pi_1\dots \pi_{\vert \sfc_k\vert }}
  \textcolor{blue}{\pi_{\vert \sfc_k\vert +1}}
  \dots \textcolor{cyan}{\pi_i}
  \col[blue]{\vert \sfc_k\vert }{\sfc_{j_1}}
  \dots\col[cyan]{i-1}{\sfc_{j_{i+c_k}}} \\
  & \eqz
  \col[olive]{n-1}{2}
  \textcolor{red}{\pi_0}\textcolor{olive}{\pi_1\pi_0\pi_1\dots\pi_i\pi_{i+1}}
  {\pi_1\dots\pi_i}
  \col[blue]{\vert\sfc_k\vert}{\sfc_{j_1}}
  \dots\col[cyan]{i-1}{\sfc_{j_{i+c_k}}}\,.
\intertext{We can now use~\ref{RB3} and redistribute the colors:}
  & \eqz
  \col[olive]{n-1}{2}
  \textcolor{red}{\pi_0} \textcolor{olive}{\pi_1 \pi_0}
  \textcolor{red}{\pi_2\dots \pi_{\vert \sfc_k\vert+1}}
  \textcolor{blue}{\pi_{\vert \sfc_k\vert +2}}
  \dots \textcolor{cyan}{\pi_{i+1}}
  \textcolor{olive}{\pi_1\dots \pi_i}
  \col[blue]{\vert \sfc_k\vert }{\sfc_{j_1}}
  \dots\col[cyan]{i-1}{\sfc_{j_{i+c_k}}}\,. \\
\intertext{Now thanks to Lemma~\ref{lemtechnmultcolonne}:}
  & \eqz
  \textcolor{red}{\pi_0} \col[olive]{n-1}{1}
  \textcolor{red}{\pi_2\dots \pi_{\vert \sfc_k\vert+1}}
  \textcolor{blue}{\pi_{\vert \sfc_k\vert +2}}
  \dots \textcolor{cyan}{\pi_{i+1}}
  \textcolor{olive}{\pi_0\pi_1\dots \pi_i}
  \col[blue]{\vert \sfc_k\vert }{\sfc_{j_1}}
  \dots\col[cyan]{i-1}{\sfc_{j_{i+c_k}}} \\
  & \eqz
  \textcolor{red}{\pi_0 \pi_1 \dots\pi_{\vert \sfc_k\vert}}
  \textcolor{blue}{\pi_{\vert \sfc_k\vert+1}}
  \dots \textcolor{cyan}{\pi_{i}}
  \col[olive]{n-1}{1}
  \textcolor{olive}{\pi_0 \pi_1\dots \pi_i}
  \col[blue]{\vert \sfc_k\vert }{\sfc_{j_1}}
  \dots\col[cyan]{i-1}{\sfc_{j_{i+c_k}}}\\
  & =_{\phantom{0}}
  \textcolor{red}{\pi_0 \pi_1 \dots\pi_{\vert \sfc_k\vert}}
  \textcolor{blue}{\pi_{\vert \sfc_k\vert+1}}
  \dots \textcolor{cyan}{\pi_{i}}
  \col[olive]{n-1}{-i-1}
  \col[blue]{\vert \sfc_k\vert }{\sfc_{j_1}}
  \dots\col[cyan]{i-1}{\sfc_{j_{i+c_k}}}\,.
\end{align*}
Going back to the main computation we can undo the splitting of
Equation~\ref{eq_split_col}:
\begin{align*}
  \pi_{\sfc}t
  & \eqz
  \pi_{\sfc_1\dots \sfc_{k-1}} \col[red]{k-1}{1}
  \dots\col[blue]{j_1-1}{\vert \sfc_k\vert+2}
  \dots\col[cyan]{j_{i+\sfc_k}-1}{i+1}
  \textcolor{red}{\pi_0 \pi_1 \dots\pi_{\vert \sfc_k\vert}}
  \textcolor{blue}{\pi_{\vert \sfc_k\vert+1}}
  \dots \textcolor{cyan}{\pi_{i}}
  \col[olive]{n-1}{-i-1}
  \col[blue]{\vert \sfc_k\vert }{\sfc_{j_1}}
  \dots\col[cyan]{i-1}{\sfc_{j_{i+c_k}}} \\
  & \eqz \pi_{\sfc_1\dots \sfc_{k-1}}\col[red]{k-1}{\sfc_k}
  \dots\col[blue]{j_1-1}{\vert \sfc_k\vert+1} \
  \dots\col[cyan]{j_{i+\sfc_k}-1}{i}
  \dots\col[olive]{n-1}{-i}
  \cdot\col[blue]{\vert \sfc_k\vert }{\sfc_{j_1}}
  \dots\col[cyan]{i-1}{\sfc_{j_{i+c_k}}} \text{ by~\ref{RB4}}\\
  & \eqz \pi_{\sfc_1\dots \sfc_{k-1}}
  \col[red]{k-1}{\sfc_k}
  \dots\col[blue]{j_1-1}{\sfc_{j_1}}
  \dots\col[cyan]{j_{i+c_k}-1}{\sfc_{j_{i+c_k}}}
  \dots\col[olive]{n-1}{-i} \text{ by Lemma \ref{lemtechnmultcolonne}}.
\end{align*}
So that we have proved that $\pi_{\sfc}t=\pi_{\sfc}$ in the last remaining case.

As told at the beginning of the proof, the remark on the length has been
checked through all cases.
\end{proof}
\begin{example}
  Since this last calculation is huge using specific notations, we now give an
  explicit example of calculation in case Neg.e.$\alpha$. We take
  $\sfc=\mathsf{1234\textcolor{red}{\ov{2}}\textcolor{blue}{1}\textcolor{cyan}{2}6\textcolor{olive}{\ov{4}}}$. Then, with $t = \pi_5$:
  \begin{align*}
    \pi_{\sfc}t
    & =_{\phantom{0}}
    \col[red]{4}{-2} \cdot \col[blue]{5}{1}\cdot \col[cyan]{6}{2}\cdot \col{7}{6} \cdot \col[olive]{8}{-4} \textcolor{olive}{\pi_5} \\
    & \eqz
    \col[red]{4}{-2} \cdot \col[blue]{5}{4}\cdot \col[cyan]{6}{5}\cdot \col{7}{6} \cdot \col[olive]{8}{-5} \cdot \col[blue]{3}{1}\cdot \col[cyan]{4}{2} \text{ by~\ref{RB4} and Lemma~\ref{lemtechnmultcolonne}}\\
    & \eqz
    \col[red]{4}{0} \cdot \col[blue]{5}{4}\cdot \col[cyan]{6}{5}\cdot \col{7}{6} \cdot \col[olive]{8}{-5} \textcolor{red}{\pi_1\pi_2} \textcolor{blue}{\pi_3} \textcolor{cyan}{\pi_4} \cdot \col[blue]{2}{1}\cdot \col[cyan]{3}{2} \text{ by~\ref{RB4}}\\
    & \eqz
    \col[red]{4}{1} \cdot \col[blue]{5}{4}\cdot \col[cyan]{6}{5}\cdot \col{7}{6} \cdot \col[olive]{8}{2}\textcolor{red}{\pi_0}\textcolor{olive}{\pi_1\pi_0\pi_1\pi_2\pi_3\pi_4\pi_5} \textcolor{red}{\pi_1\pi_2} \textcolor{blue}{\pi_3} \textcolor{cyan}{\pi_4} \cdot \col[blue]{2}{1}\cdot \col[cyan]{3}{2} \text{ by~\ref{RB4}}\\
    & \eqz
    \col[red]{4}{1} \cdot \col[blue]{5}{4}\cdot \col[cyan]{6}{5}\cdot \col{7}{6} \cdot \col[olive]{8}{2}\textcolor{red}{\pi_0}\textcolor{olive}{\pi_1\pi_0}\textcolor{red}{\pi_2\pi_3}\textcolor{blue}{\pi_4}\textcolor{cyan}{\pi_5} \textcolor{olive}{\pi_1\pi_2 \pi_3 \pi_4 }\cdot \col[blue]{2}{1}\cdot \col[cyan]{3}{2} \text{ by~\ref{RB3} and redistributing.}\\
    & \eqz
    \col[red]{4}{1} \cdot \col[blue]{5}{4}\cdot \col[cyan]{6}{5}\cdot \col{7}{6} \cdot \textcolor{red}{\pi_0\pi_1\pi_2}\textcolor{blue}{\pi_3}\textcolor{cyan}{\pi_4} \col[olive]{8}{2}\textcolor{olive}{\pi_1\pi_0\pi_1\pi_2 \pi_3 \pi_4 } \cdot \col[blue]{2}{1}\cdot \col[cyan]{3}{2} \text{ by~\ref{RB2} and~\ref{RB4}} \\
    & \eqz
    \col[red]{4}{-2} \cdot \col[blue]{5}{3}\cdot \col[cyan]{6}{4}\cdot \col{7}{6} \cdot \col[olive]{8}{-4}\cdot \col[blue]{2}{1}\cdot \col[cyan]{3}{2}
    \eqz \col[red]{4}{-2} \cdot \col[blue]{5}{1}\cdot \col[cyan]{6}{2}\cdot \col{7}{6} \cdot \col[olive]{8}{-4} = \pi_\mathsf{c } \text{ by Lemma~\ref{lemtechnmultcolonne}}.
  \end{align*}
\end{example}

\begin{remark}\label{casq=1}
  The Definition \ref{defactioncode}, the Lemma \ref{codebiendef} and the
  Theorem \ref{actionsurcode} can be also adapted to the case of $G_n^1$,
  using the transformation $\pi_i\mapsto s_i$ for $i\neq 0$ and $\pi_0 \mapsto
  \pi_0$. There are only few cases which differ; they are precisely those
  where relation \ref{RB1} is used (with $i\neq 0$), that is case Pos.a. and Neg.a. The
  modifications in the definition are thus the followings:
\begin{itemize}\setlength{\itemindent}{1cm}
\item[Pos.a.] ${c_n} = i > 0$ and $t = s_i$ then $\sfc\cdot s_i =
  \mathsf{c_1\dots c_{n-1} (c_n+1)}$.
\item[Neg.a.] ${c_n} = -i \leq 0$ and $t = s_i$ then $\sfc\cdot s_i =
  \mathsf{c_1\dots c_{n-1} (c_n+1)}$.
\end{itemize}
The equivalent of Lemma \ref{codebiendef} can be proved the same way. Finally
the proof of Theorem \ref{actionsurcode} only use the relation $s_i^2 = 1$ in
these two cases.
\end{remark}

\begin{corollary}\label{picodesurj}
  Let $\mathsf{1^c_n}$ denote the code of the identity rook of size $n$.  For
  any $\pi \in G_n^0$ and $s\in G_n^1$, the congruencies
  $\pi\eqz\pi_{\mathsf{1^c_n}\cdot\pi}$ et $s\equiv_1s_{\mathsf{1^c_n}\cdot
    s}$ hold.
\end{corollary}
\begin{proof}
  We use Theorem \ref{actionsurcode} and Remark \ref{casq=1} at $\sfc =
  \mathsf{1^c_n}$ and proceed by induction on the length of the words $\pi$ or $s$.
\end{proof}
We now have an easy proof of the identities that motivated
Definition~\ref{defactioncode}:
\begin{corollary}\label{code_action_commute}
  For any generator $t$ the following diagram is commutative:
  \begin{equation*}
  \begin{tikzcd}[column sep=2cm]
    R_n \arrow[r, bend left=10, "\code"] \arrow[d, "\cdot t"] &
    C_n \arrow[l, bend left=10, "\decode"] \arrow[d, "\cdot t"] \\
    R_n \arrow[r, bend left=10, "\code"] &
    C_n \arrow[l, bend left=10, "\decode"]
  \end{tikzcd}
\end{equation*}
\end{corollary}
\begin{proof}
  We start by Theorem \ref{actionsurcode}, $\pi_{\sfc\cdot t} \eqz \pi_\sfc
  t$. Now since $\Phi_0: G_n^0 \rightarrow F_n^0$ is a morphism, we can apply 
  this relation to the rook
  $1_n$. We obtain: $1_n\cdot \pi_{\sfc\cdot t} = 1_n\cdot (\pi_\sfc \, t) =
  (1_n\cdot \pi_\sfc)\,t.$
  We conclude thanks to Proposition \ref{action1n} and Theorem \ref{codedecode}.
\end{proof}

\begin{corollary}\label{egalitetouscardinaux}
  The maps
  $\left\{\begin{array}{@{}c@{\,}c@{\,}c@{}}
    C_n &\twoheadrightarrow& G_n^0 \\
    c&\mapsto&\pi_c
  \end{array}\right.$
  and
  $\left\{\begin{array}{@{}c@{\,}c@{\,}c@{}}
    C_n &\twoheadrightarrow& G_n^1 \\
    c&\mapsto& s_c
  \end{array}\right.$
  are surjective; the following cardinalities coincide:
  \begin{equation*}
    \vert C_n\vert = \vert R_n\vert =
    \vert F_n^0\vert = \vert G_n^0\vert = \vert F_n^1\vert = \vert G_n^1\vert\,.
  \end{equation*}
  Moreover, $F_n^0 \simeq G_n^0$, $F_n^1 \simeq G_n^1$ as monoids.
\end{corollary}
\begin{proof}
  Using both Remark~\ref{presFnFn0} and Corollary~\ref{code_action_commute},
  we get the following sequence of surjective maps: $C_n \twoheadrightarrow G_n^0
  \twoheadrightarrow F_n^0. $ Furthermore $\vert F_n^0\vert \geq \vert
  C_n\vert$ by Corollary \ref{plusdeRn0}. Consequently $\vert C_n\vert = \vert
  F_n^0\vert = \vert G_n^0\vert$ and $F_n^0 \simeq G_n^0$ as monoids.
\end{proof}

\begin{example}\label{exmpl_decode}
Let $r = 240503$ and $t = \pi_0$. Then $r\cdot t = 040503$. Let us check our algorithm.

Firstly $\code(r) = \mathsf{01323\ov{2}}$. Our algorithm gives us the following serie of operations:
\[\begin{aligned}
\mathsf{01323\ov{2}} \cdot \pi_0 	& = \left[(\mathsf{01323} )\cdot \pi_0\pi_1 \right] 0\\
							& = \left[\left((\mathsf{0132})\cdot \pi_0\right)\mathsf{3} \cdot \pi_1 \right] \mathsf{0} = \left[\left((\mathsf{013})\cdot \pi_0\right) \mathsf{23} \cdot \pi_1 \right] \mathsf{0} = \left[\left((\mathsf{01})\cdot \pi_0\right) \mathsf{323} \cdot \pi_1 \right] \mathsf{0}\\
							& = \left[\mathsf{00323} \cdot \pi_1 \right] \mathsf{0} = \left[\mathsf{0032} \cdot \pi_1 \right] \mathsf{30}\\
							& = \mathsf{003130}
\end{aligned}\]
Finally we really have $\decode(\mathsf{003130}) = 040503$.
\end{example}

Now, there is no need to distinguish between the monoids of functions from the presented monoids, since we have the proof that they are isomorphic. 
\begin{notation}
  We denote $\Rnz \eqdef F_n^0 \simeq G_n^0$ the $0$-rook monoid.

  For any rook $r$ we also denote $\pi_r \eqdef \pi_{\code(r)}$.
\end{notation}

\begin{corollary}\label{action_mulr_R0}
  $\pi_r$ is the unique element of $\Rnz$ such that $1_n\cdot \pi_r = r$.
  With the identification $r\leftrightarrow \pi_r$, the action of $\Rnz$ on
  $R_n$ is nothing but the right multiplication in $\Rnz$:
  $\pi_r\pi_s =\pi_{r\cdot \pi_s}$.
\end{corollary}
\begin{proof}
  The identity $1_n\cdot \pi_r = r$ is Proposition~\ref{action1n}, and $\pi_r$
  is unique thanks to cardinalities. Finally,
  $1_n\cdot\pi_r\pi_s=(1_n\cdot\pi_r)\cdot\pi_s=r\cdot\pi_s$ and we conclude
  by unicity.
\end{proof}
We have, by the way, re-proven the presentation for the classical rook monoid:
\begin{corollary}
  For all $n$, We have the following isomorphisms of monoids: $F_n^1 \simeq R_n
  \simeq G_n^1$.
\end{corollary}
\begin{proof}
The monoid morphism $\left\{\begin{array}{@{}c@{\ }c@{\ }c@{}}
\langle s_1, \dots, s_{n-1}, \pi_0\rangle \subseteq R_n & \longrightarrow & F_n^1 \subseteq \mathcal{F}(R_n, R_n)\\
r & \longmapsto & (r' \mapsto r'\cdot r)
\end{array}\right.$ is well-defined, and surjective. By Corollary \ref{egalitetouscardinaux} we can deduce that ${\langle s_1, \dots, s_{n-1}, \pi_0\rangle \simeq R_n \simeq F_n^1}$.
\end{proof}
\bigskip

Here is a further immediate consequence of the presentation:
\begin{corollary}\label{autoopp}
  The monoid $\Rnz$ is isomorphic to its opposite.
\end{corollary}
\begin{proof}
  It comes from the fact that the relations of the presentation of $\Rnz$ are
  symmetrical.
\end{proof}

\subsection{A Matsumoto theorem for rook monoids}
We now turn to the specific study of reduced words.
\begin{proposition}
  The words $s_{\code(r)}$ and $\pi_{\code(r)}$ are reduced expressions
  (i.e. of minimal length) respectively for $r\in R_n$ and $\pi_r\in\Rnz$.
\end{proposition}
\begin{proof}
  Corollary~\ref{picodesurj} tells us that every element of $R_n$ and $\Rnz$
  can be written as $\pi_\sfc$ and $s_\sfc$ for some code $\sfc$.  Moreover,
  according to Theorem~\ref{actionsurcode} the rewriting of any word to
  $\pi_\sfc$ and $s_\sfc$ only decrease the length. To conclude, we still have
  to argue that $\pi_\sfc$ and $s_\sfc$ cannot be obtained with a different
  shorter code, which is clear from Proposition~\ref{action1n}.
\end{proof}

\begin{remark}
  The Corollary \ref{picodesurj} gives us a standard expression for every
  element of $\Rnz$. We can now look back at Lemma \ref{Pn} and realize that
  $P_n$ corresponds to the $R$-code $00\dots 0$ ($n$ times), and thus to the
  action of replacing all the entries by $0$.
\end{remark}

A final important consequence of our construction is a proof of the analogue
of Matsumoto's theorem, answering a question of Solomon~\cite[p.~209, bottom
of the middle paragraph]{Solomon.2004}:
\begin{theorem}[Matsumoto theorem for Rook monoids]
  \label{theoreom-matsumoto}
  If $\un{u}$ and $\un{v}$ are two reduced words over
  $\{\pi_0,s_1\dots,s_{n-1}\}$ (resp. $\{\pi_0, \pi_1,\dots,\pi_{n-1}\}$) for the
  same element $r$ of $R_n^1$ (resp.~$R_n^0$), then they are congruent using
  only the two Relations~\ref{RB2} and \ref{RB4}, namely the braid relations:
  \begin{alignat}{2}
  s_is_{i+1}s_i &= s_{i+1}s_is_{i+1} &&  1\leq i \leq n-2,
  \tag{Rs2}\label{RBraids2}\\
  s_is_j &= s_js_i                &&\vert i-j\vert \geq 2.
  \tag{Rs3}\label{RBraids3}\\
  \pi_0 s_j &= s_j \pi_0           && j \neq 1.
  \tag{Rs5.1}\label{RBraids5.1}\\
\intertext{Respectively:}
  \pi_i\pi_{i+1}\pi_i &= \pi_{i+1}\pi_i\pi_{i+1} &\qquad&  1\leq i \leq n-2,
  \tag{RB2}\label{RBraid2}\\
  \pi_i\pi_j &= \pi_j\pi_i &&
  0\leq i,j \leq n-1,\qquad\vert i-j\vert \geq 2.
  \tag{RB4}\label{RBraid4}
  \end{alignat}
\end{theorem}
\begin{proof}
  First of all, we only do the proof at $q=0$, the $q=1$ case is done
  similarly. Moreover, by transitivity, it is sufficient to work in the case
  where $\un{v} = \pi_\sfc$ whith $\sfc = \code(1_n\cdot r)$.  We proceed by
  induction on the common length $\ell$ of $\un{u}$ and $\un{v}$. It is
  obvious when $\ell=0$. We now consider a reduced word $\un{v}=\un{v'}t$ for
  an element $r$. Then $\un{v'}$ is also reduced for an element $r'$, so
  that $r't=r$. We assume by induction that $\un{v'}$ is congruent to
  $\pi_{\sfc'}$ where $\sfc'\eqdef\code(1_n\cdot r')$ using only
  Relations~\ref{RB2} and \ref{RB4}. Therefore~$\un{v'}t$ and $\pi_{\sfc'}t$
  are congruent too. In the proof of Theorem~\ref{actionsurcode}, we
  explicitely gave how to go from $\pi_{\sfc'}t$ to $\pi_{\sfc'\cdot
    t}$. Hence we only need to check that Relations~\ref{RB1} and \ref{RB3}
  are only used in the case where $\un{v'}t$ is not reduced that is when the
  length of $\un{v'}t$ is larger that the length of $\pi_{\sfc'\cdot t}$. This
  indeed holds, namely, in cases Pos.a., Neg.a which use~\ref{RB1} on one
  hand, and cases Neg.d, Neg.e.$\alpha$ which use~\ref{RB3} on the other hand.
\end{proof}
As a consequence reduced words for $R_n^1$ and $\Rnz$ are the same:
\begin{corollary}\label{action_reduced}
  Let $\un{w}^1\in G_n^1$ a word for a rook $r$ and $\un{w}^0$ its
  corresponding word in $G_n^0$ obtained by replacing $s_i$ by $\pi_i$ and
  leaving $P_1$. Then $\un{w}^1$ is reduced if and only if $\un{w}^0$ is
  reduced. Moreover, when they are, for any $k=0,\dots,|w|$, one has $1_n\cdot
  w^1_1\cdots w^1_k=1_n\cdot w^0_1\cdots w^0_k$ and the elements $(1_n\cdot
  w^0_1\cdots w^0_k)_{k=0\dots|w|}$ are all distinct.
\end{corollary}
\begin{proof}
  Any reduced word is congruent by braid relations to a canonical one:
  $s_\sfc$ and $\pi_c$. Moreover, the canonical words corresponds by the
  exchange $s\leftrightarrow\pi$ and the braid relations keep this
  correspondence, so that the first statement holds. Now assume that a word
  $\un{w}^i$ is reduced. Thanks to Corollary~\ref{action_mulr_R0}, we know
  that the sequence of elements are distinct, otherwise it would imply that
  some products $w^i_1\cdots w^i_k$ are equal for two different values of $k$
  leading to a shorter word. Now Equation~\ref{act_rook_pik}, prove the
  equality.
\end{proof}

As explained by Solomon~\cite{Solomon.2004}, this is sufficient to give a presentation of
the $q$-rook algebra. Here is a quick sketch on how to do that: fix a
parameter $q$ in a ring $\mathbf{R}$ and define an endomorphism $T_i$ of $\mathbf{R} R_n$
interpolating between $q=1$ and $q=0$ by
\begin{equation}\label{equ_def_Hecke_generic}
r\cdot T_i \eqdef q (r\cdot s_i) + (1 - q)(r \cdot (\pi_i - 1))\,,
\end{equation}
for $i=1, \dots, n-1$ (where $s_i$ and $\pi_i$ acts according to
Equations~\ref{act_rook_sk} and \ref{act_rook_pik}). It is well
known~\cite{Lascoux.2003.book,LascouxSchutzenberger.1987} that these
operators generate the Hecke algebra. We now consider the algebra generated by
those generators plus $P_1$ defined as in Equation~\ref{act_rook_pik}. Since
$P_1$ commutes with $s_i$ and $\pi_i$ for $i\geq2$, it commutes with
$T_i$. Therefore for any rook $r$, it makes sense to define $T_r \eqdef T_{i_1}
T_{i_2} \dots P_1 \dots T_{i_k}$ for any reduced word $s_{i_1} P s_{i_2} \dots
P_1 \dots s_{i_k}$. Due to the braid relations the result is independent from the
chosen reduced word. Moreover for each of those words
\begin{equation}
  1\cdot T_r = r + \text{shorter terms},
\end{equation}
so that these $(T_r)_{r\in R_n}$ are linearly independent. It finally suffices
to add four more relations which explain how to simplify non reduced words. Namely:
\begin{align}
  (T_i + 1) (T_i - q) &= 0, \\
  P_1^2 &= P_i,\\
  (P_1-1)T_1(P_1-1)T_1 &= T_1(P_1-1)T_1(P_1-1),\\
  P_1(T_1-q)P_1(T_1(1-P_1)T_1-q)&=0.
\end{align}
We remark that this presentation is true over $\Z$ and therefore over any ring,
and not only on fields. As far as we know, this was unknown before.

\subsection{More actions of \texorpdfstring{$\Rnz$}{Rn0}}\label{sec_action_Rnz}

In Definition~\ref{def_Ro_fun}, we have given a right action of $\Rnz$ on
$R_n$. It is now clear from Corollary~\ref{action_mulr_R0} that this action is
nothing but the right multiplication in $\Rnz$. Under this action, $P_j$ acts
by killing the first $j$ entries:
\begin{equation}
(r_1\dots r_n)\cdot P_j = 0\dots 0\,r_{j+1}\dots r_n\,.
\end{equation}

The inverse of a permutation matrix is its transpose. Transposing a rook
matrix still gives a rook matrix, so that one can transfer the notion to rook
vectors. It is computed as follows: for a rook $r$, the $i$-th coordinate of
$r^t$ is the position of $i$ in $r$ if $i\in r$, and $0$ otherwise. For
instance $(105203)^t = 146030$.

Transposing the natural right action, we naturally get a left action of the
opposite monoid on rooks. However $\Rnz$ is isomorphic to its oppose. It is
therefore possible to define a left natural action:
\begin{definition}\label{def_left_action}
  For $0\leq i\leq n$ and $r = r_1\dots r_n\in R_n$, define
  \begin{equation}
    \pi_i \cdot r\eqdef (r^t \cdot \pi_i)^t
    \qquad \text{ so that } \qquad
    r\cdot \pi_i = (\pi_i\cdot r^t)^t\,.
  \end{equation}
  More explicitely, for $0\leq j \leq n$, we write $j \in r$ if $j \in \lbrace
  r_1,\dots, r_n\rbrace$.  Then for any rook $r$:
  \begin{itemize}
  \item $\pi_0$ replaces $1$ by $0$ in $r$ if $1\in r$, and fixes $r$
    otherwise.
  \item For $i>0$, the action of $\pi_i$ on $r$ is
    \begin{itemize}
    \item if $i, i+1\in r$, call $k$ and $l$ their respective positions. Then
      $\pi_i$ fixes $r$ if $l<k$, otherwise it exchanges $i$ and~$i+1$.
    \item if $i \notin r$ and $i+1\in r$, then $\pi_i$ replaces $i+1$ by $i$.
    \item if $i+1 \notin r$ then $\pi_i$ fixes $r$.
    \end{itemize}
  \end{itemize}
\end{definition}
\begin{lemma}\label{left_action}
  The previous definition is a left monoid action of $\Rnz$ on $R_n$ called
  the \emph{left natural action}.  Under this action, $P_j$ acts by replacing
  the entries smaller than $j$ by $0$.
\end{lemma}

\begin{example}
$\pi_0 \cdot 0342= 0342, \quad \pi_1 \cdot 0342 = 0341, \quad \pi_2 \cdot 0342 = 0342, \quad {\pi_3 \cdot 0342 = 0432}$, \quad ${\pi_0 \cdot 132 = 032}.$
\end{example}

This sheds some light on the link with the type $B_n$: it is well known that
type $B_n$ can be realized using signed permutations. The quotient giving the
$0$-rook monoid can be realized by replacing the negative numbers by zeros.

\begin{proposition}\label{action_mull_R0}
  $\pi_r$ is the unique element of $\Rnz$ such that $\pi_r \cdot 1_n = r$.
  With the identification $r\leftrightarrow \pi_r$, the left action of $\Rnz$
  on $R_n$ is nothing but the left multiplication in $\Rnz$: $\pi_r\pi_s
  =\pi_{\pi_r\cdot s}$.
\end{proposition}
\begin{proof}
  For a rook $r$, let us call temporarily ${}_r\pi$ the reverse of the word
  $\pi_{r^t}$. Transposing Corollary~\ref{action_mulr_R0} we get that
  ${}_r\pi$ is characterized by ${}_r\pi \cdot 1_n = r$ and ${}_r\pi{}_s\pi =
  {}_{\pi_r\cdot s}\pi$.  However, at this stage it's not clear that
  ${}_r\pi=\pi_{r}$ (as element of $R_n^0$). Nevertheless, for generators that
  is words of length $1$, the equality ${}_r\pi=\pi_{r}$ holds. Now given any
  reduced word $\un{w}=w_1\dots w_l$ for an element $x\in\Rnz$, set $r \eqdef 1_n
  \cdot \un{w} = 1_n\cdot w_1 \cdot w_2 \cdots w_l$ so that $x=\pi_r$ in
  $\Rnz$. Since $\un{w}$ is reduced, using Corollary~\ref{action_reduced}, one
  gets that $r = \un{w}^1$ (the product of the corresponding word in $R_n^1$
  which is nothing but a matrix product). But this gives that $r = \un{w}^1
  \cdot 1_n$ so that using the transpose of Corollary~\ref{action_reduced},
  $r=\un{w} \cdot 1_n$. By unicity, one concludes that ${}_r\pi=\pi_{r}$.
\end{proof}
\begin{corollary}\label{commutdroitegauche}
  The natural left and right actions of $\Rnz$ on $R_n$ commute.
\end{corollary}
\begin{proof}
  Thanks to~\ref{action_mulr_R0} and \ref{action_mull_R0}, this is just
  associativity in $\Rnz$.
\end{proof}
\bigskip

One can also extend the action of $\Hnz$ by isobaric divided differences on
polynomials: the monoid $\Rnz$ acts also on the polynomials in $n$
indeterminates over any ring $k$, $k[X_1, \dots, X_n]$ in the following way.
\begin{lemma}
 Let $f \in k[X_1,\dots, X_n]$. Define
\begin{equation}
  f \cdot \pi_0 \eqdef f_{\vert X_1 = 0} = f(0, X_2, \dots X_n),
  \qquad\text{and}\qquad
  f \cdot \pi_i \eqdef \frac{X_i f - (X_i f) \cdot s_i}{X_i-X_{i+1}}\,.
\end{equation}
This definition is a right monoid action of $\Rnz$ over $k[X_1, \dots, X_n]$.
Under this action,
\begin{equation}
f\cdot P_j = f(0, \dots0, X_j, \dots X_n)\,.
\end{equation}
\end{lemma}
\begin{proof}
  It is a well-known fact~\cite{LascouxSchutzenberger.1987} that isobaric
  divided differences give an action of the Hecke algebra at $q=0$. It remains
  only to show the relation
  $\pi_1P_1\pi_1P_1=P_1\pi_1P_1=P_1\pi_1P_1\pi_1$. We easily check by an
  explicit computation that the three members are equals to the operator $P_2$
  defined by $f\cdot P_2 = f(0, 0, X_2,\dots, X_n)$.  The action of $P_n$ can
  be easily obtained by induction with $P_{i+1} = P_i \pi_i P_i$.
\end{proof}
Actually, there is an extra relation, which can be checked by a explicit
computation:
\begin{equation}
  f\cdot \pi_1\pi_0\pi_1 = f\cdot \pi_0\pi_1\pi_0\,.
\end{equation}
This shows that the monoid which is actually acting is $H^0(A_{n+1})$ (Cartan
type $A_{n+1}$) thanks to the following sequence of surjective morphisms:
\begin{equation}
  H^0(B_{n}) \twoheadrightarrow \Rnz \twoheadrightarrow H^0(A_{n+1}).
\end{equation}
Finally, we note that it is actually possible to get an action of the full
generic $q$-rook algebra by taking the same definition as
Relation~\ref{equ_def_Hecke_generic}.

\section{The \texorpdfstring{$\RR$}{R}-order on rooks}
\label{sec-r-order}

In this section, we seek for combinatorial, order theoretic and geometric
analogs of the permutohedron for rooks. Recall that the right Cayley graph of
the symmetric group $\SG{n}$ has several interpretations, namely:
\begin{itemize}
\item the Hasse diagram of the right weak order of $\SG{n}$ seen as a Coxeter
  group, which is naturally a lattice~\cite{GuilbaudRosenstiehl};
\item the Hasse diagram of Green's $\RR$-order of the $0$-Hecke monoid
  $H_n^0$~\cite{DHST};
\item the skeleton of the polytope obtained as the convex hull of the set of
  points whose coordinates are permutations~\cite[Example 0.10]{ZieglerGeom}.
\end{itemize}
As we will see, some of these properties have an analog for rooks. 
\begin{figure}[ht]
  $$
  \vcenter{\hbox{\includegraphics[scale=0.6]{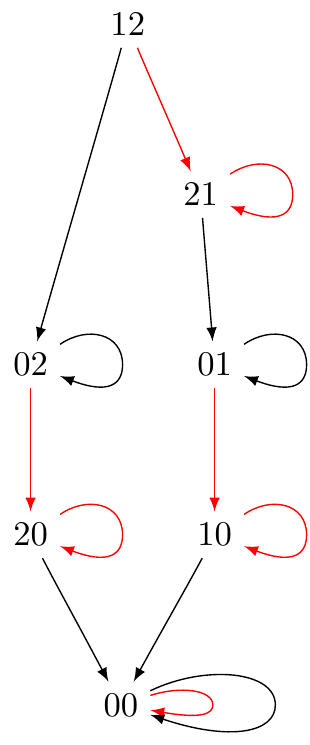}}}
  \hspace{-2cm}
  \raisebox{3cm}{\includegraphics[scale=1]{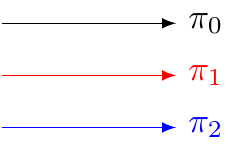}}
  \qquad
  \vcenter{\hbox{\includegraphics[scale=0.6]{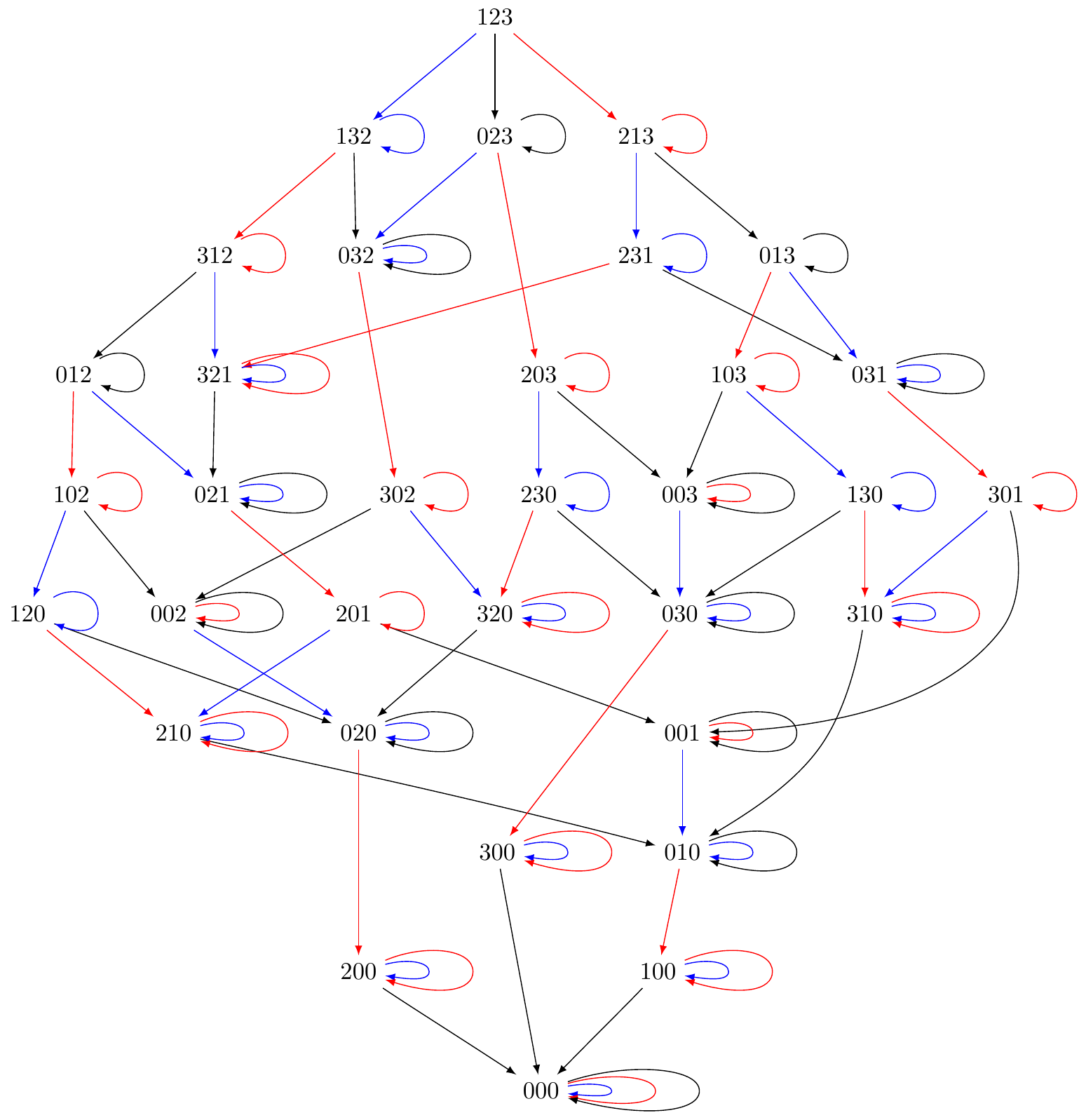}}}
  $$
  \caption{\label{CayleyR02} The right Cayley graph of $R_2^0$ and $R_3^0$.}
\end{figure}

We first notice an important difference: on the contrary to $\SG{n}$ the right
order is not graded. This has been already noted for $R_2^0$. Indeed in the left part of
Figure~\ref{CayleyR02} we see two paths from $12$ to $00$ namely
$\pi_0\pi_1\pi_0$ on the left and $\pi_1\pi_0\pi_1\pi_0$ on the
right. Starting with $n=3$ the right order is moreover not isomorphic to its
dual order.

\subsection{\texorpdfstring{$\RR$}{R}-triviality of \texorpdfstring{$\Rnz$}{Rn0}}

In this section we study the right Cayley graph of $\Rnz$ showing that except
for loops (edge from a vertex to itself) it is acyclic. In monoid theoretic
terminology, one says that $\Rnz$ is $\RR$-trivial. From Coxeter group point
of view, this is the analogue on rook of the (dual) right weak order. Note
that the order considered here is different to the (strong) Bruhat order. Its
analogue for rook is the subject of~\cite{CanRenner.2012}.

Having shown this acyclicity, we will deduce from the symmetry of the
relations of $\Rnz$ that the left sided Cayley graph is also acyclic. By a
standard semigroup theory argument, this will imply that the two-sided Cayley
graph is acyclic too, that is that $\Rnz$ is actually $\JJ$-trivial.

We first recall a combinatorial description of the $\RR$-order of the
$0$-Hecke monoid (or equivalently the dual right-weak order of the symmetric
group seen as a Coxeter group)~\cite{BjornerBrenti.2005}. Recall that for two
permutations $\sigma$ and $\tau$ one has $\sigma\leq_\RR\tau$ if there exists
a sequence $(i_1, \dots, i_k)$ with $0<i_j<n$ such that
$\sigma=\tau\cdot\pi_{i_1}\dots\pi_{i_k}$. Note that, in accord with the
monoid convention and contrary to the Coxeter group convention, the identity
is the largest element for this order. An algorithmic way to compare two
permutations is to use values inversions (sometimes called co-inversions). We
give here a definition which is also valid for rooks:
\begin{definition}\label{def_inversions}
  For a rook $r$, the \emph{set of inversions} of $r$ is defined by
  \begin{equation}
    \Inv(r) \eqdef\left\{ (r_i, r_j) \mid i<j\text{ and }r_i>r_j>0\right\}.
  \end{equation}
\end{definition}
It is a subset of $\Delta\eqdef\{(b, a)\mid n\geq b>a>0\}$, but
not all subsets are inversions sets of permutations and of rooks as we will see.
\begin{definition}
  A subset $I\subseteq \Delta$ is \emph{transitive} if $(c,b)\in I$ and $(b,a)\in I$
  implies $(c,a)\in I$.
\end{definition}
Here is a characterization of inversions sets:
\begin{lemma}\label{perm_lin_ext}
  Given a set $I\subseteq \Delta$, there exists a permutation $\sigma$ such that
  $\Inv(\sigma)=I$ if and only if $I$ and $\Delta \setminus I$ are both
  transitive. When this holds the permutation $\sigma$ is unique.
\end{lemma}
\begin{proof}
  This is a folklore result. To reconstruct $\sigma$ from its inversion set,
  one shows that the relation
  $I \cup \{(i,j) \mid (j, i)\in \Delta \setminus I \}$
  is a total order, that is a permutation.
\end{proof}
Inversion sets allow to characterize the right order:
\begin{lemma}[\cite{BjornerBrenti.2005}]
  Let $\sigma, \tau \in \SG n$, then $\sigma\leq_\RR\tau$ if
  and only if $\Inv(\tau)\subseteq\Inv(\sigma)$.
\end{lemma}
\begin{proposition}[\cite{BjornerBrenti.2005}]\label{meet-permuto}
  The right $\RR$-order on permutations is a lattice. The meet $\sigma\meR\mu$
  of $\sigma$ and $\tau$ is characterized by: $\Inv(\sigma\meR\mu)$ is the
  transitive closure of~$\Inv(\sigma)\cup\Inv(\mu)$. The join of $\sigma$ and
  $\tau$ is characterized by: $\Delta\setminus\Inv(\sigma\jnR\mu)$ is the
  transitive closure of $(\Delta\setminus\Inv(\sigma)) \cup
  (\Delta\setminus\Inv(\mu))$.
\end{proposition}
\medskip

We now present how to adapt inversion sets to rooks. The idea is to record
usual inversions as well as inversion with a $0$ letter. Here is a way to do
it:
\begin{definition}\label{def-rook-triple}
  We call the \emph{support} of a rook $r$ denoted $\supp(r)$ the set of
  non-zero letters appearing in its rook vector.  For each letter
  $\ell\in\supp(r)$, we denote $Z_r(\ell)$ the number of $0$ which appear
  after $\ell$ in the rook vector of $r$.

  We finally say that $(\supp(r), \Inv(r), Z_r)$ is the \emph{rook triple
    associated to $r$}.
\end{definition}
\begin{example}\label{rook-triple}
  For example for $r=2054001$, one gets $\supp(r) = \{1,2,4,5\}$,
  together with $\Inv(r)=\{(2,1),(4,1),(5,4),(5,1)\}$, $Z_r(1)=0$, $Z_r(2)=3$ and
  $Z_r(4)=Z_r(5)=2$.
\end{example}
Here is a characterization of the rook triples:
\begin{proposition}\label{rook-triple-characterization}
  A triple $(S, I, Z)$ where $S\subseteq\{1,\dots n\}$, $I\subseteq\Delta$ and
  $Z:S \mapsto \N$ is the rook triple of a rook $r$ if and only if the three
  following properties hold:
  \begin{itemize}
  \item the sets $I\subset \Delta\cap S^2$ and $I$ and $(\Delta\cap S^2) \setminus I$
    are both transitive.
  \item for $\ell\in S$, one has $0\leq Z(\ell) \leq n - |S|$;
  \item if $(b,a)\in I$ then $Z(b) \geq Z(a)$ else $Z(b) \leq Z(a)$.
  \end{itemize}
  Moreover, when these properties hold the corresponding rook $r$ is unique.
\end{proposition}
\begin{proof}
  We first prove the direct implication. The first statement says that if one
  erases the zeros from a rook, one gets a permutation of its support. The
  second statement says that there are $n-|\supp(r)|$ zeros. The third
  statement says that if $a$ is after $b$ in $r$, then there are less $0$ to
  the right of $a$ than to the right of $b$.

  Conversely, given such a triple, we can reconstruct a rook~$r$ in two steps:
  the first condition ensures that there is a unique permutation~$\sigma$ of
  the support~$S$ with inversions set~$I$. The third statement says that the
  function~$Z$ is decreasing along the word~$\sigma$. As a consequence,
  writing~$\sigma^Z_i$ the subword of~$\sigma$ composed by the letters $\ell$
  such that $Z(\ell)=i$, one has
  \begin{equation}
    \sigma = \sigma^Z_{n-|\supp(r)|}\,\dots\,\sigma^Z_2\,\sigma^Z_1\,\sigma^Z_0\,.
  \end{equation}
  Note that some of the $\sigma^Z_i$ may be empty. Then the rook
  \begin{equation}
    r = \sigma^Z_{n-|\supp(r)|}\,0\,\dots\,0\,\sigma^Z_2\,0\,\sigma^Z_1\,0\,\sigma^Z_0\,.
  \end{equation}
  is indeed associated with the triple $(S, I, Z)$ and is by construction
  unique.
\end{proof}
\begin{example}
  Going back to Example~\ref{rook-triple}, consider the following triple
  with $n=7$:
  $$
  (S, I, Z) = \left(\{1,2,4,5\},\{(2,1),(4,1),(5,4),(5,1)\},
    \begin{psmallmatrix}
    1&2&4&5\\
    0&3&2&2
    \end{psmallmatrix}\right)\,.
  $$
  There is a unique permutation $\sigma$ of $S$ with inversion set $I$,
  namely $2541$. Writing $Z(i)$ below $i$ for each letter of $\sigma$, we get
  $\begin{psmallmatrix}
    2&5&4&1\\
    3&2&2&0\\
  \end{psmallmatrix}$
  and see that $Z$ is indeed decreasing. We then get that
  $\sigma^Z_3=(2)$, $\sigma^Z_2=(54)$, $\sigma^Z_1=()$, $\sigma^Z_0=(1)$,
  so that we recover $r=2054001$.
\end{example}
\bigskip

Our aim is now to show that the $\RR$-order is actually an order. To do so, we
start by defining combinatorially an order $r\leq_I u$, and then show that
$\leq_I$ and $\leq_\RR$ are actually equivalent.
\begin{definition}\label{def-rook-order}
  Let $r$ and $u\in R_n$.  We write $r\leq_I u$ if and only if the three
  following properties hold:
  \begin{itemize}
  \item $\supp(r)\subseteq \supp(u)$,
  \item $\left\{ (b, a) \in \Inv(u) \mid b \in \supp(r)\right\}
    \subseteq \Inv(r)$,
  \item $Z_{u}(\ell) \leq Z_{r}(\ell)$ for $\ell\in\supp(r)$.
  \end{itemize}
\end{definition}
\begin{remark}\label{rem-compat}
  If $r$ and $u$ are permutations, then $\supp(r)=\supp(u)=\{1,\dots,n\}$, so
  that $r\leq_I u$ if and only if $\Inv(u)\subset\Inv(r)$.
\end{remark}
Moreover, as a consequence of the second condition, if $(b,a)\in\Inv(u)$ and
$b\in\supp(r)$ then $a\in\supp(r)$. We abstract this fact with the following
definition and lemma:
\begin{definition}\label{def-compat-set}
  Let $I\subseteq \Delta$ and $S\subset \interv{1}{n}$. We say that $S$ is
  \emph{$I$-compatible} if $(b,a)\in I$ and $b\in S$ implies $a\in S$, for all
  $b,a$.
\end{definition}
The previous remark now rephrases as:
\begin{lemma}
  If $r\leq_\RR u$ then $\supp(r)$ is $\Inv(u)$-compatible.
\end{lemma}
We will further need the following basic facts about compatibility:
\begin{lemma}
  The union $S_1\cup S_2$ of two $I$-compatibles sets $S_1$ and $S_2$ is
  $I$-compatible.

  If $S$ is $I_1$ and $I_2$-compatible, then it is $I_1\cup I_2$-compatible.

  If $S$ is $I$-compatible then it is compatible with the transitive closure
  of $I$.
\end{lemma}

We get back to the study of~$\leq_I$.
\begin{proposition}\label{Rn_poset}
  The set $R_n$ endowed with the relation~$\leq_I$ is a poset with maximal
  element $1_n$ and minimal element $0_n = 0\dots 0$.
\end{proposition}
\begin{proof}
The relation $\leq_I$ is reflexive, by definition.

If $r,u\in R_n$ are such that $r\leq_I u$ and $u\leq_I r$ then $\supp(r) =
\supp(u)$ and therefore $\Inv(r) = \Inv(u)$ and $Z_r=Z_{u}$. As a
consequence, the non-zero letters appear in the same order in $r$ and $u$ and
the zeros are in the same places. Thus $\leq_I$ is antisymmetric.

Let $r\leq_I u \leq_I v$. Then $\supp(r) \subseteq \supp(v)$. Let
$(b,a)\in \Inv(v)$ with $b\in\supp(r)$. Necessarily $b\in \supp(u)$ so that
$(b, a)\in \Inv(u)$ and consequently $(b, a)\in \Inv(r)$. Finally if
$\ell\in\supp(r)$ then $Z_{v}(\ell) \leq Z_{u}(\ell) \leq
Z_{r}(\ell)$. Thus $\leq_I$ is transitive.
\end{proof}

\begin{theorem}\label{ordreR0}
  Let $r, u \in R_n.$ Then $\pi_r\leq_\RR \pi_u$ if and only if $r\leq_Iu$.
\end{theorem}
\begin{proof}
  By definition, $\pi_r\leq_\RR \pi_u$ if there exists $\pi \in \Rnz$ such
  that $\pi_r=\pi_u\pi$. Using the identification $r\leftrightarrow
  \pi_r$ of Corollary~\ref{action_mulr_R0}, this is equivalent to
  $r=u\cdot\pi$. By abuse of notation in this proof we will therefore write
  $r\leq_\RR u$ if there exists $\pi \in \Rnz$ such that $r=u\cdot\pi$.
  \bigskip

  For the direct implication, by induction and transitivity, it is sufficient
  to assume that $r=u\cdot\pi_i$ with $r\neq u$ and show $r<_I u$.
  \begin{itemize}
  \item If $i \neq 0$. Then $\supp(u)=\supp(r)$. Since $r\neq u$ we must
    have $u_i < u_{i+1}$ and also $r=u_1\dots u_{i+1}u_i\dots u_n$. If $u_i\neq0$
    then $\Inv(r)=\Inv(u)\sqcup\lbrace (u_{i+1},u_i)\rbrace$ and
    $Z_r=Z_u$. On the contrary, if $r_i=0$, then $\Inv(r)=\Inv(u)$ and
    $Z_r(\ell)=Z_u(\ell)$ for $\ell\neq u_{i+1}$ and $Z_r(u_{i+1})=Z_u(u_{i+1})+1$.

  \item If $i = 0$. Since $r\neq u$ we have $r_1\neq 0$ and $u=0r_2\dots
    r_n$. We can deduce that $\supp(r) = \supp(u) \cup \{ r_1\}$. Furthermore,
    \begin{equation}
      \Inv(r)
      = \lbrace (u_i, u_j) \in \Inv(u) \mid i\neq 1 \rbrace
      = \lbrace (u_i, u_j) \in \Inv(u) \mid u_i \in r\rbrace\,.
    \end{equation}
    Finally for $\ell\in\supp(r)$, $Z_{u}(\ell)=Z_{r}(\ell)$.
  \end{itemize}

  For the converse implication, assume that $r<_I u$. By induction and
  transitivity it is sufficient to show that there exists $i$ such that
  $r\leq_I u\cdot \pi_i$ and $u\cdot \pi_i\neq u$. We proceed by a case
  analysis. First since $\supp(u)\subseteq\supp(u)$, we can distinguish
  whether $\supp(u)=\supp(u)$ or $\supp(u)\subsetneq\supp(u)$. In the
  equality case, we further distinguish whether $Z_u= Z_r$ or not.
  \begin{itemize}
  \item If $\supp(u)=\supp(r)$, and $Z_u\neq Z_r$, then there must exist
    $\ell\in\supp(r)$ such that $Z_{u}(\ell) < Z_{r}(\ell)$. Pick the leftmost
    $\ell$ in $u$ which verifies this condition. First, there must be some $0$
    on the left of $\ell$ in $u$ because there are $Z_{u}(\ell)$ on the right
    and at least $Z_{r}(\ell)$ in the word. Thus $\ell$ is not the first
    letter of $u$.

    Let $k$ be the letter immediately preceding $\ell$ in $u$. We claim that
    either $k=0$ or $k$ is after $\ell$ in $r$. Indeed if $k\neq0$ and $k$ is
    before $\ell$ in $r$ then we have $Z_{r}(k)\geq Z_{r}(\ell)$. Moreover
    $Z_{u}(\ell)=Z_{u}(k)$ because there is no zero in $u$ between $\ell$ and
    $k$. Therefore $Z_{r}(k)\geq Z_{r}(\ell)>Z_{u}(\ell)=Z_{u}(k)$ which
    contradicts our choice of $\ell$ as being the leftmost.

    Now, call $i$ the position of this $k$ in $u$. If $k=0$, the only
    difference between the rook triples of $u$ and $u\cdot\pi_i$ is that
    $Z_{u\cdot\pi_i}(\ell) = Z_{u}(\ell)+1$ so that $r\leq_I u\cdot\pi_i$. On
    the contrary, if $k\neq 0$, then the only difference between the rook
    triples of $u$ and $u\cdot\pi_i$ is that $\Inv(u\cdot\pi_i) = \Inv(u)
    \sqcup \{(l, k)\}$ so that again $r\leq_I u\cdot\pi_i$.

  \item If $\supp(u)=\supp(r)$, and $Z_u=Z_r$, then necessarily
    $\Inv(u)\subsetneq\Inv(r)$. Write $\tilde{r}$ and $\tilde{u}$ the words
    obtained by removing the zeros in $r$ and $u$. The inclusion of inversions
    shows that $\tilde{u}\leq_S\tilde{r}$ where $\leq_S$ is the right order
    for permutations of $S=\supp(u)$. As a consequence, we know that it is
    possible to exchange two consecutive letters $a<b$ in $\tilde{u}$ to get a
    permutation $\tilde{v}$ of $\supp(u)$ such that
    \begin{equation}
      \Inv(\tilde{v})=\Inv(\tilde{u})\sqcup\{(b,a)\}\subset\Inv(\tilde{r})\,.
    \end{equation}
    From the equality of $Z$, there cannot be any $0$ between $a$ and $b$ in
    $u$, thus $a$ and $b$ are consecutive in $u$ as well. Writing $i$ for the
    position of $a$ in $u$, we have $r\leq_I u\cdot\pi_i$.

  \item The remaining case is $\supp(r)\subsetneq\supp(u)$. Let
    $\ell\eqdef\max(\supp(u)\setminus\supp(r))$. If $\ell$ is in position $1$ in
    $u$ then $r\leq_I u\cdot\pi_0$ and we are done in this case.

    Otherwise if $\ell$ is not in position $1$, we claim that the letter $k$
    immediately preceding $\ell$ in $u$ is smaller than $l$. If not, then
    there is an inversion $(k,\ell)$ in $u$. Since $\supp(r)$ is
    $\Inv(u)$-compatible, then $k\notin\supp(r)$. This contradicts our choice
    of $\ell$ as being the maximum.

    Writing $i$ for the position of $k$ in $u$, we proceed as in the end of
    the first case: the only difference between the rook triples of $u$ and
    $u\cdot\pi_i$ is that $\Inv(u\cdot\pi_i) = \Inv(u) \sqcup \{(\ell,k)\}$ so
    that again $r\leq_I u\cdot\pi_i$. \qedhere
\end{itemize}
\end{proof}
\bigskip

We are now in position to prove the main result of this section:
\begin{corollary}\label{Jtrivial}
  The monoid $\Rnz$ is $\RR$-trivial, $\LL$-trivial and thus $\JJ$-trivial.
\end{corollary}
\begin{proof}
  A consequence Theorem~\ref{ordreR0} is that the $\RR$-preorder is an order
  so that $\Rnz$ is $\RR$-trivial. Moreover, it is isomorphic to its opposite
  by Corollary \ref{autoopp} and thus it is $\LL$-trivial. We conclude with
  Lemma~\ref{R_and_L_donc_J}.
\end{proof}

\subsection{The lattice of the \texorpdfstring{$\RR$}{R}-order}

Our goal here is to show that, similarly to the weak order of permutations,
the $\RR$-order for the rooks is a lattice. We start with an algorithm which
computes the meet.
\begin{theorem}\label{meet-r-order}
  Let $u$ and $v$ be two rooks of size $n$. Define a new rook $r$ by the
  following algorithm:
  \begin{itemize}
  \item Let $I_0$ be the transitive closure of $\Inv(u)\cup\Inv(v)$.
  \item Let $S$ be the largest (for inclusion) $I_0$-compatible set contained
    in $\supp(u)\cap\supp(v)$.
  \item Let $I \eqdef I_0 \cap S^2$.
  \item Finally, for $x\in s$ let
      $Z(x) \eqdef \max\{Z_u(i), Z_v(i) \mid i = x \text{ or } (x, i)\in I \}$
      with the convention that $Z_s(i)=0$ if $i\notin\supp(s)$.
  \end{itemize}
  Then $(S, I, Z)$ is a rook triple whose associated rook $r$ is the meet
  $u\meR v$ of $u$ and $v$ for the $\RR$-order.
\end{theorem}
\begin{proof}
  We first prove that $(S, I, Z)$ is indeed a rook triple.
  \begin{itemize}
  \item By definition, $I\subset \Delta\cap S^2$, let us show that $I$
    and $(\Delta\cap S^2) \setminus I$ are transitive. We claim that $I$ is
    the transitive closure of $(\Inv(u)\cap S^2) \cup (\Inv(v)\cap S^2)$.
    Indeed, for any $(b, a)\in I$, then $(b, a)\in I_0$. By definition of the
    transitive closure, there exists a decreasing sequence of integer $b=c_1>
    c_2>\dots>c_k = a$ such that $(c_i, c_{i+1})\in \Inv(u) \cup \Inv(v)$ for
    $i=1,\dots,k-1$. By induction, since $b\in S$, compatibility ensures that
    all of the $c_i$ belong to $S$. Hence the claim.

    As a consequence, using Proposition~\ref{meet-permuto}, $I$ is the
    inversion set of the meet in the permutohedron of the restriction of $u$
    and $v$ to $S$ so that $I$ and $(\Delta\cap S^2) \setminus I$ are
    transitive.
  \item On has $|S| \leq \max(|\supp(u)|, |\supp(v)|)$. So that the condition
    $0\leq Z(x) \leq n - |S|$ holds.
  \item Write $\mathcal{Z}(x) \eqdef\{Z_u(i), Z_v(i) \mid i = x \text{ or } (x,
    i)\in I \}$ so that $Z(x) \eqdef \max\mathcal{Z}(x)$. If $(b,a)\in I$, the
    transitivity of $I$ ensures that as sets
    $\mathcal{Z}(b)\supseteq\mathcal{Z}(a)$ so that $Z(b)\geq Z(a)$.
    Conversely write $\overline{I}\eqdef(\Delta\cap S^2) \setminus I$.  If
    $(b,a)\in\overline{I}$, the transitivity of $\overline{I}$ shows that
    $(a,i)\in\overline{I}$ implies $(b,i)\in\overline{I}$. By contraposition,
    $(b,i)\in I$ implies $(a,i)\in I$ so that
    $\mathcal{Z}(b)\subseteq\mathcal{Z}(a)$ and therefore $Z(b)\leq Z(a)$.
  \end{itemize}
Hence, we have proved that $(S,I,Z)$ is a rook triple. It remains to prove
that its associated rook is the meet $u\meR v$. By construction, $r\leq_I u$
and $r\leq_I v$. So that we only need to prove that for any rook $s$ such that
$s\leq_I u$ and $s\leq_I v$ then $s\leq_I r$.
\begin{itemize}
\item Using the rephrasing of Remark~\ref{rem-compat} we know that then
  $\supp(s)$ is $\Inv(u)$ and $\Inv(v)$-compatible and therefore compatible
  with the transitive closure of their union $I_0$. Since $S=\supp(r)$ is
  defined as the largest such set, $\supp(s)\subseteq\supp(r)$.
\item Suppose $(b, a)\in\Inv(r)$, with $b\in\supp(s)$. Then by construction of $r$,
  there is a decreasing sequence $b=c_1>c_2>\dots>c_k=a$ such that
  $(c_i,c_{i+1})\in\Inv(u)\cup\Inv(v)$ for $i=1,\dots,k-1$. By induction,
  having $s\leq_I u$ and $s\leq_I v$, one prove $c_i\in\supp(s)$ and
  $(c_i,c_{i+1})\in\Inv(s)$. One concludes by transitivity that
  $(b,a)=(c_1,c_k)\in\Inv(s)$.
\item Finally, assume $x\in\supp(s)$. Then $Z_s(x)\geq Z_u(x)$ and $Z_s(x)\geq
  Z_v(x)$. Moreover for any $i$ such that $(x, i)\in\Inv(r)$, by the preceding
  item, $i\in\supp(s)$ and $(x, i)\in\Inv(s)$. One deduces that $Z_s(x)\geq
  Z_s(i)\geq Z_u(i)$ and $Z_s(x)\geq Z_s(i)\geq Z_v(i)$. We just showed that
  $Z_s(x)\geq\max\mathcal{Z}(x)$. \qedhere
\end{itemize}
\end{proof}
\begin{corollary}\label{cor-order-lattice}
  The $\RR$-order of $\Rnz$ is a lattice.
\end{corollary}
\begin{proof}
  From the previous theorem, we know that $\Rnz$ is a meet semi-lattice. Now it
  is well known that a meet semi-lattice with a maximum element is a lattice.
\end{proof}
From the proof, we have a more explicit algorithm to compute the meet:
\begin{itemize}
\item Start with $S\eqdef\supp(u)\cap \supp(v)$. Then while one can find a $(b,a)
  \in \Inv(u) \cup \Inv(v)$ with $b\in S$ and $a\notin S$, remove $b$ from
  $S$. When no more such $(b,a)$ can be found, $S$ is the support of $u\meR v$.
\item Using the usual algorithm for permutations of the set $S$ (see the
  sketch of the proof of Lemma~\ref{perm_lin_ext}), compute the meet of the
  restriction $u_{|S}$ and $v_{|S}$.
\item Compute the $Z$ function using $\max$ as in the statement of
  Theorem~\ref{meet-r-order}.
\item Finish inserting the zeros using $Z(x)$ as in the proof of
  Proposition~\ref{rook-triple-characterization}.
\end{itemize}
\begin{example}
  Let $u = 25104$ and $v = 12453$.  So $\supp(u)\cap
  \supp(v) = \{1,2,4,5\}$. But $(4,3)$ and $(5,3) \in \Inv(v)$
  and $3\notin S$. So $S=\{1,2\}$. We then get $I=\{(2,1)\}$, So that $(u\meR
  v)_{|S}=21$. It remains to insert the zeros. One compute $Z(2)=1$ and $Z(1)=1$
  so that $u\meR v = 00210$.
  Here is a bigger example: Let us compute $r=31086502\,\meR\,02178534$. One
  finds that $S=\{1, 2, 3\}$, and $I=\{(3, 2), (3, 1), (2, 1)\}$ and
  $Z=\begin{psmallmatrix}
    1&2&3\\
    2&2&2
    \end{psmallmatrix}$, so that $r=00032100$. Similarly
  $$
  30175082 \meR 02154738 = 00308210\qandq
  43017582 \meR 02154738 = 75430821.
  $$
\end{example}
In the case of permutations, the involution
$\sigma\rightarrow\tilde\sigma=\sigma\omega$ where $\omega$ is the maximal
permutation (otherwise said, $\tilde\sigma$ is the mirror image of $\sigma$) is
an isomorphism from the $\RR$-order to its dual. A a consequence, one can
compute the join using the meet: $\sigma\vee_\RR\mu =
\widetilde{\tilde\sigma\meR\tilde\mu}$. However, as seen for example on
Figure~\ref{CayleyR02} this trick does not work anymore for rooks. This ask
for an algorithm to compute the join of two rooks. To describe this algorithm,
we need a notion of non-inversion and a dual notion of compatibility:
\begin{definition}\label{def-dual-compat-set}
  For any rook $r$, call \emph{set of version} of $r$ the set:
  \begin{equation}
    \InvB(r)\eqdef(\Delta\setminus\Inv(r))\ \cup\
    \{(b,a)\in\Delta\mid a\notin r\text{ and } b\in r\}\,.
  \end{equation}
  Let $I\subseteq \Delta$ and $S\subset \interv{1}{n}$. We say that $S$ is
  \emph{dual $I$-compatible} if $(b,a)\in \Delta\setminus I$ and $a\in S$
  implies $b\in S$.
\end{definition}
\begin{theorem}\label{join-r-order}
  Let $u$ and $v$ be two rooks of size $n$. Define a new rook $r$ by the
  following algorithm:
  \begin{itemize}
  \item Let $I_0\eqdef\Delta\setminus T$ where $T$ is the transitive closure of
    $\InvB(u)\cap\InvB(v)$.
  \item Let $S$ be the smallest dual $I_0$-compatible set containing
    $\supp(u)\cup\supp(v)$.
  \item Let $I \eqdef I_0 \cap S^2$.
  \item Finally, for $x\in s$ let $Z(x) \eqdef \min\{Z_u(i), Z_v(i) \mid i = x
    \text{ or } (x, i)\in \Delta\setminus I\}$,
    with the convention that $Z_s(i)=+\infty$ if $i\notin\supp(s)$.
  \end{itemize}
  Then $(S, I, Z)$ is a rook triple whose associated rook $r$ is the join
  $u\jnR v$.
\end{theorem}
The proof is very similar to the one we did for the meet and is left to the reader.
\begin{example}
  Let us compute $r=30175082\,\jnR\,72185043$. One finds
  $S=\{1, 2, 3, 4, 5, 7, 8\}$, $I=\{(7, 5), (4, 3)\}$ and $Z
  =\begin{psmallmatrix}
      1&2&3&4&5&7&8\\
      1&0&0&0&0&0&0
  \end{psmallmatrix}$, so that $r=10243758$.
\end{example}
\bigskip

We want to enumerate the join-irreducible elements. As in the classical
permutohedron, they are related to descents, however, it the case of rooks,
they are two different notions of descents.

\begin{definition}[Weak and strict descents]

  Let $r\in R_n$ be a rook. For any $0\leq i<n$, we say that $i$ is a
  \emph{weak (right) descent} of $r$ if $r\cdot\pi_i=r$. We say that $i$ is a
  \emph{strict (right) descent} if there exists a rook $s\neq r$ such that
  $s\cdot\pi_i=r$. Moreover, in the particular case $i=0$, we say that $0$ is
  a strict descent with multiplicity $k$, if there are exactly $k$ rooks
  $s\neq r$ such that $s\cdot\pi_0=r$.
\end{definition}
Any strict descent is a weak descent. Indeed if $s\cdot\pi_i=r$ then
$r\cdot\pi=s\cdot\pi_i^2=s\cdot\pi_i=r$. Weak descent and strict descents are
equivalent when restricted to permutations, but they differ on rooks. For
example, the rook $04003$, has 3 weak descent namely $0,2,3$, but only $0,2$
are strict ($04003=24003\cdot\pi_0$ and $04003=00403\cdot\pi_2$) and $0$ has
multiplicity $3$: $04003=14003\cdot\pi_0=24003\cdot\pi_0=54003\cdot\pi_0$.
\begin{lemma}\label{des0mult}
  The multiplicity of $0$ as a strict descent in a rook $r$ is $0$ if $r$ does
  not start with $0$ and is the number of $0$ in $r$ otherwise.
\end{lemma}

\begin{definition}
  An element $z$ of a lattice $L$ is called \emph{meet irreducible} if it can
  not be obtained as a non trivial meet that is $z = z_1 \wedge z_2$ implies
  $z_1=z$ or $z_2=z$.
\end{definition}
An equivalent definition is that $z$ has only one successor in the Hasse
diagram of $L$. By definition, in a finite lattice, any element can be
written as the meet of some meet irreducible elements. As a consequence, they
form the minimal generating set of the meet semi-lattice.

For permutations, the number of meet irreducible for the $\RR$-order (that is
permutation with only one descent) is $a(n) = 2^n - n - 1$. It is a particular
case of Eulerian numbers and is recorded as OEIS A000295. Here are the first
values
\begin{equation}
  0, 0, 1, 4, 11, 26, 57, 120, 247, 502, 1013, 2036, 4083, 8178, 16369, 32752\,.
\end{equation}
For rooks, the number of meet irreducibles has a very simple expression too:
\begin{proposition}
  The number of meet irreducibles for $\leq_\RR$ is $3^n - 2^n$.
\end{proposition}
This sequence is recorded as OEIS A001047. Here are the first values
\begin{equation}
  0,1,5,19,65,211,665,2059,6305,19171,58025,175099,527345,1586131\,.
\end{equation}
We will actually prove a stronger statement, the previous one will follow
thanks to the identity:
\begin{equation}
  3^n - 2^n = \sum_{i=1}^{n}\ 3^{n-i}\,2^{i-1}\,.
\end{equation}
\begin{proposition}\label{prop-meet-irred}
  For any rook vector $r$ denote $p(r)$ the first value $r_0$ if its non zero,
  and $1$ if its zero. The number of meet-irreducibles $r$ of $R_n$ such that
  $p(r)=i$ is $3^{n-i}\,2^{i-1}$.
\end{proposition}
\begin{proof}
  A rook is meet irreducible if and only if it has a unique strict descent
  (counting multiplicities). Consider a meet irreducible rook $r$ with
  $p(r)=i$. There are two cases:
  \begin{itemize}
  \item if $i>1$, then the rook is composed by two nondecreasing sequences,
    the first one starts with $i$. So each number smaller than $i$, either
    appears in the second subsequence or, do not appear at all so that the
    second sequence starts with some $0$. Similarly each number larger than
    $i$, may appear in any of those two subsequences or not at all. So the
    number of choices is $2^{i-1}3^{n-i}$.
  \item if $i=1$, then $r$ start either with $0$ or $1$. We want to show
    that the number of such rooks is $3^{n-1}$. We show that the set of
    those rooks is in bijection with the set of maps
    $f:\interv{2}{n}\rightarrow\{0,1,2\}$.

    In the following, for any set $S$ of integers we write $W(S)$ the word
    obtained by writing the letter of $S$ in increasing order. Given a map
    $f:\interv{2}{n}\rightarrow\{0,1,2\}$, one build a sequence starting with
    $1$, then ordering the preimage of $0$, putting as many zero as the
    preimage of $1$, and then ordering the preimage of $2$:
    \begin{equation}\label{def_first_r}
      r(f) \eqdef 1\cdot W(f^{-1}(0))\cdot 0^{|f^{-1}(1)|}\cdot W(f^{-1}(2))\,.
    \end{equation}
    By definition, the result is a rook of size $n$ with at most one
    descent. Moreover, each rook with only one descent is obtained exactly
    once as the image of some $f$.

    It remains to show that the maps which give rooks with no descent by the
    preceding construction are in bijection with rooks having $0$ as unique
    descent with multiplicity $1$. The point is the following: $r(f)$ has zero descents,
    that is $r(f)$ is nondecreasing, if and only if there exists a $1\leq
    k\leq n$ such that
    \begin{equation}
      f(i) =
      \begin{cases}
        0 & \text{if $i\leq k$,} \\
        2 & \text{otherwise.}
      \end{cases}
    \end{equation}
    If it is the case, we redefine $r(f)$ as
    \begin{equation}\label{def_second_r}
      r_1(f) \eqdef 0\cdot W(\{i-1\mid i\in f^{-1}(0)\})\cdot W(\{ f^{-1}(2)\}).
    \end{equation}
    The set of the rooks obtained this way is the set of increasing rooks
    which start with a $0$. According to Lemma~\ref{des0mult}, those are
    exactly the rooks having $0$ as unique descent with multiplicity $1$.

    On conclude that there are exactly $3^{n-1}$ rooks starting either by $0$
    or $1$. \qedhere
  \end{itemize}
\end{proof}
\begin{example}
Consider the function $f=\begin{psmallmatrix}
  2&3&4&5&6&7&8&9\\
  2&0&1&0&1&0&2&1\\
\end{psmallmatrix}$. Then $r(f)=1\cdot357\cdot000\cdot28$ which has only one
strict descent (the dots are only here to visualize the different part of the
right hand side of Equation~\ref{def_first_r}).

Now with $f=\begin{psmallmatrix}
  2&3&4&5&6&7&8&9\\
  0&0&0&0&0&2&2&2\\
\end{psmallmatrix}$, Equation~\ref{def_first_r} gives
$r(f)=1\cdot23456\cdot\cdot789$ which has no descent at all. So we take the
second definition (Equation~\ref{def_second_r}) and get the new value
$r_1(f)=0\cdot12345\cdot789$ which has $0$ as unique strict descent.
\end{example}

As a concluding remark on irreducible elements, we note that, on the contrary
to permutations, the poset is not self dual. So there is no reason why the
number of meet irreducible elements should be equal to the number of join
irreducible elements. They indeed differ and we do not have a formula for the
number of join irreducibles. We give here the first values:
\begin{equation}
  0, 1, 5, 16, 43, 106, 249\,.
\end{equation}

\subsection{Chains in the rook lattice}\label{subsec:chains}

We recall that a \emph{chain} in a lattice $(L, \preceq)$ is a sequence of
elements $(x_1, \dots, x_r)$ such that
$x_1 \preceq x_2 \preceq\dots \preceq x_r$. A \emph{maximal chain} is a chain
which is not strictly included in another one. Denoting $m$ and $M$ the
minimal and maximal elements of $L$, this is equivalent to 
$x_1 = m$, $x_r = M$ and for every $i< r$ there is no element between $x_i$
and $x_{i+1}$ for the order $\preceq$. For the weak order on permutations,
maximal chains corresponds to reduced expression of the maximal permutation.

We now consider maximal chains of $R_n$ (thus also $R_n^0$ by
Corollary~\ref{action_mulr_R0}). We see in Figure~\ref{CayleyR02} that all
the maximal chains are not of equal length. Experimental computation of the
numbers of maximal chains give the following sequence: $1, 2, 23, 3625,
16489243$. We did not find any nice property: it is not refered in OEIS and the
numbers contain big prime factor. A more interesting question is to only
consider maximal chains of minimal length, that is reduced expressions of the
maximal rook $P_n = 0\dots 0 \in R_n$. Note by Lemma~\ref{Pn} that $\ell(P_n) =
\binom{n+1}{2}$. We find the following numbers of such chains:
\begin{equation}
1, 2, 12, 286, 33592, 23178480\,.
\end{equation}
This sequence is refered as OEIS A003121. It counts, among many other things,
the number of maximal chains of length $\binom{n+1}{2}$ (hence maximal) in the
\emph{Tamari lattice} $\mathcal{T}_{n+1}$. This suggests that there is a
bijection between the chains. It turns out that the coincidence is much
stronger: the two posets restricted to the elements appearing in their
respective chains of maximal length are isomorphic.

We first need to describe the elements appearing in a reduced expression of
$P_n$. We need the following combinatorial definition for this:
\begin{definition}\label{def:shuffle}
Let $\mathcal{A}$ be an alphabet and $a, b\in \mathcal{A}$, $\un{u}, \un{v}$ be two words over $\mathcal{A}$. The \emph{shuffle product} is defined inductively by:
  \begin{equation}
  a\un{u}\shuffle b\un{v} = a(\un{u} + b\un{v}) + b(a\un{u} + \un{v})\,,
  \end{equation}
the initial condition being that the empty word is the unit element.
\end{definition}
\begin{proposition}
  The rook vectors appearing as a left factor of a reduced expression of $P_n$
  are the rooks:
  \begin{equation}\label{def_MCRn}
    \MCR{n} \eqdef \{ 0\dots 0  \shuffle (k+1)\dots n \mid 0\leq k \leq n \}. 
  \end{equation}
\end{proposition}
\begin{proof}
  Let $r \in\MCR{n}$ as defined by Equation~\ref{def_MCRn}. We assume that $r$
  has $k$ zeros, so that the nonzero letters appearing in $r$ are $k+1, \dots,
  n$. Take the reduced expression for $r$ given by the $R$-code
  (Definition~\ref{def_word_code}). Since the nonzero letters are in
  order, this expression if of length $\ell(r) = 1 +2+ \dots + k + \sum_{i =
    k+1}^n Z_r(i)$. In order to bring $r$ to $P_n$ by right action we repeat
  the following steps until we reach $P_n$: let $i$ be the first nonzero letter and
  $p=i-Z_r(i)$ its position. Then multiplying $r$ on the right by
  $s_{p-1}\dots s_1\pi_0$ brings $i$ to the front and kills it.  The length of
  the word for $P_n$ obtained this way is equal to
  \begin{equation}
    \ell(r) + \sum_{i=k+1}^n \left(i-Z_r(i)\right) = \sum_{i=1}^n i = \binom{n+1}{2}.
  \end{equation}
  This is the length of $P_n$, hence the expression is reduced, and $r$
  appears in a maximal chain of minimal length.

  Now we prove the converse inclusion by contradiction. Let $r\in R_n
  \setminus \MCR n$, with $k$ zeros. We want to show that there is no reduced
  word for $P_n$ of the form $\un{r}\,\un{m}$ where $\un{r}$ is a word for
  $r$. Assume that we have such a word. Since $r\notin \MCR n$, then either
  there is a nonzero letter~$k$ before a nonzero letter $k'$ with $k'<k$, or
  there is a nonzero letter $k'$ while a letter $k>k'$ is missing. The
  algorithm computing the canonical reduced word
  (Definition~\ref{def_word_code}) shows that:
\begin{equation}
\ell(r) > 1 +2+ \dots + k + \sum_{i \in r,\,i\neq 0} Z_r(i).
\end{equation}
We call $\tilde{r}\in \MCR n$ the rook vector obtained from $r$ by replacing
the nonzero letters by $k+1, \dots, n$ in this order, so that $\sum_{i \in
  r,\, i\neq 0} Z_r(i) = \sum_{i = k+1}^n Z_{\tilde{r}}(i)$. Then
$\tilde{r}\,\un{m}$ gives $P_n$ as well. Thus $\ell(\un{m}) \geq \sum_{i =
  k+1}^n (i-Z_{\tilde{r}}(i))$. So that $\ell(P_n) = \vert \un{r}\,\un{m}\vert >
\binom{n+1}{2} = \ell(P_n)$, which is absurd.
\end{proof}

In particular note that:
$
\displaystyle | \MCR n | = \sum_{i=0}^n \binom{n}{i} = 2^n.
$
\begin{example}
$\MCR 2 
= \{ 12 \} \cup \{ 0 \shuffle 2 \} \cup \{00\}
= \{ 12 \} \cup \{ 02, 20 \} \cup \{00\}$
\begin{alignat*}{2}
\MCR 3 =\ & \{ 123 \} \cup \{ 0 \shuffle 23\} \cup \{ 00 \shuffle 3\} \cup \{000 \}\\
=\ & \{ 123\}\cup \{023, 203, 230\}\cup \{ 003, 030, 300\}\cup \{ 000\}, \\[3mm]
\MCR 4 =\ & \{ 1234 \} \cup \{ 0 \shuffle 234\} \cup \{ 00 \shuffle 34\} \cup \{000 \shuffle 4 \} \cup \{ 0000\} \\
=\ & \{ 1234\}\cup \{ 0234, 2034, 2304, 2340\}\cup \{ 0034, 0304, 3004, 0340, 3040, 3400\} \\ &  \cup \{0004, 0040, 0400, 4000\} \cup \{ 0000\}.
\end{alignat*}
\end{example}
We now introduce a sequence of bijections from $\MCR n$ to some special \emph{Dyck
  paths}, that is vertices of the Tamari lattice. For now, recall that a
\emph{Dyck path} of length $n$ is a path in the plane starting from $(0,0)$,
ending in $(2n,0)$ made with north-east (NE) $(1,1)$ and south-east (SE) $(1,
-1)$ such that the path is always above the line $y=0$. We represent a Dyck
path by a word of size $2n$ with $n$ letters $0$ and $n$ letters $1$, where
$0$ is a SE step, and $1$ a NE step, and such that in every prefix of the word
the number of $0$ is less or equal to the number of $1$. For instance
$101100110$ is a Dyck path. We will also represent it $1^1 0^1 1^2 0^2 1^2 0^1$.
\medskip

The first bijection sends an element of $\MCR n$ to a subset of
$[n+1]$ the following way:
\begin{equation}
  \eta : \left\{ \begin{array}{rcl}
      \MCR n & \longrightarrow & [n+1]\\
      r = r_1\dots r_n & \longmapsto & \{i \mid r_i \neq 0\} .
    \end{array}\right.
\end{equation}
This application is clearly a bijection since the nonzero letters of $r\in
\MCR n$ are $k+1, \dots, n$ in this order, where $k$ is the number of zeros of
$r$. Now that we have a subset of $[n]$ we can use the bijection $\cset$ to
compositions of $n+1$ introduced in Equation~\ref{eq.set.to.comp}. If $I =
(i_1, \dots, i_m) \compof n+1$ the actions of the generators of $\Rnz$ through
the bijection $\cset \circ \eta$ are as follows:
\begin{eqnarray}
  I \cdot \pi_0 & = & (i_1 + i_2, i_3, \dots, i_m), \\
  I \cdot \pi_j & = & (i_1, \dots,  i_{j-1}, i_j -1, i_{j+1} + 1, i_{j+2}, \dots, i_m) \text{ for } 0<j<m.
\end{eqnarray}

We finally send a composition of $n+1$ to a Dyck path as follows:
\begin{equation}
  \delta : (i_1, \dots, i_m) \compof n+1 \longmapsto 1^{n-m}\, 0^{i_1} \,1 \,0^{i_2} \,1\, 0^{i_3} \dots 0^{i_{m-1}} \,1\, 0^{i_m}.
\end{equation}
It is easy to check that the Dyck paths we obtain this way are exactly those
for whose the pattern $011$ is forbidden. Note that the action of the
generators of $\Rnz$ is thus to replace a $01$ by $10$ which pictorially
inserts a diamond in a ``valley''. See Figure~\ref{fig:flip-valley}. We say
that a Dyck path $D$ contains another Dyck path $D'$, and we denote it $D'
\subseteq D$, if the path $D$ is above the path $D'$. Then the $\RR$-order on
$\Rnz$ is mapped to the order $\subseteq$ on Dyck paths avoiding the pattern
$011$ by the bijection $\delta\circ \cset \circ \eta$. The maximal element of the
poset on Dyck path is $1^{n+1} 0^{n+1}$ and its minimal $(10)^{n+1}$. See the
first line of Figure~\ref{fig:chain_max} to see all these
isomorphisms. We finally remark that all these posets are actually lattices.

\begin{figure}[ht]
\centering
\includegraphics[scale=1.5]{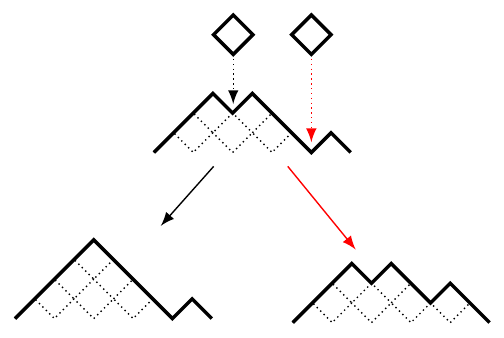}
\caption{The flip of a valley in our special Dyck paths. The generator $\pi_i$ adds a diamond in the $i+1$th valley, counting from the left. Thus $\pi_0$ reduces the number of valley. }\label{fig:flip-valley}
\end{figure}

Now we briefly present the Tamari order, the reader should ref
to~\cite{Tamari.1962} for more details. A Dyck path is called \emph{primitive}
if it is not empty and has no other contact with the line $y=0$ except at the
starting and ending point. If $u$ is a Dyck path such that $u$ has a SE step
$d$ followed by a primitive path $p$. Then the rotation on $u$ is to exchange
the SE step $d$ with the primitive path $p$. See
Figure~\ref{fig:rot_dyck}. These rotations are the cover relations of the
Tamari order $\preceq_{\mathcal{T}}$.

\begin{figure}[ht]
\centering
\includegraphics[scale=1]{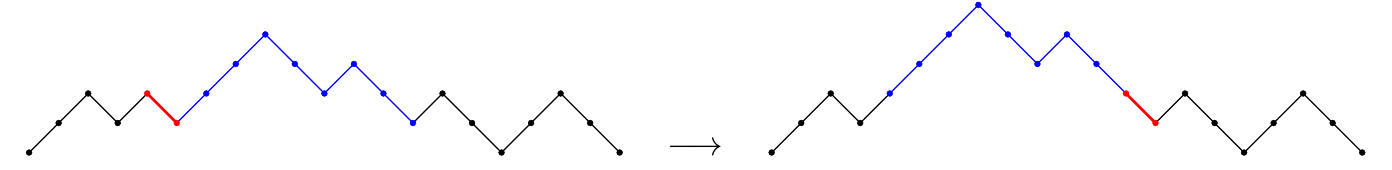}
\caption{\label{fig:rot_dyck} The rotation of Dyck words.}
\end{figure}

We are interested in Dyck paths in a maximal chain of length $\binom{n}{2}$ in
the Tamari lattice of size $n$. We denote by $\MCT n$ their set.
\begin{proposition}
  The set $\MCT n$ is exactly the set of Dyck paths avoiding
  $011$. Furthermore the order $\preceq_{\mathcal{T}}$ restricted to $\MCT n$
  is equal to the order of inclusion $\subseteq$.
\end{proposition}
\begin{proof}
  The difference of diamonds between the minimal element $(10)^n$ and the
  maximal element $1^n 0^n$ is exactly $\binom{n}{2}$, so that each rotation
  must add only one diamond. But a rotation on a SE step $0$ followed by two
  NE steps $11$ adds at least two diamonds, so that we can not rotate in such
  a SE step. Moreover the rotations on another licit SE step preserve the $011$
  pattern, so that an element with pattern $011$ can not be in $\MCT n$. On
  the contrary if $D$ is a Dyck path avoiding $011$, then a rotation is
  exactly to add a diamond in a valley, and the resulting Dyck path also avoids	
  $011$.

  Now that we have the description of elements of $\MCT n$, doing a rotation
  corresponds to adding a diamond on a valley, so that the order
  $\preceq_\mathcal{T}$ implies the order $\subseteq$. Furthermore, by
  definition of the order $\preceq_{\mathcal{T}}$, the converse also holds.
\end{proof}
As a consequence we have proven that the order on $\MCT n$ obtained through
the bijection $\delta \circ \cset \circ \eta$ is exactly the Tamari order, so that
the posets of $\MCR n$ and $\MCT{n+1}$ are isomorphic.

The elements appearing in $\MCT n$ appears in many different
contexts, see \cite{HohlwegLange.2007, HohlwegLangeThomas.2011, LabbeLange.2018} and the references in the latters. They correspond to binary trees which are chains, that
is also binary trees with exactly one linear extension. For this reason they
are called \emph{singletons}. Equivalently they are permutations avoiding the
patterns $132$ and $312$, or permutations with exactly one element in their
sylvester class, that is common vertices between the associahedron and the
permutahedron.  Furthemore the historic definition of the associahedron is to
keep only the faces of the permutahedron which contains such a singleton. See
Figure~\ref{fig:assoc-perm}, and \cite{HohlwegLange.2007} for more details. See also Figure~\ref{fig:chain_max} for all the bijections seen in this section.

\begin{figure}[ht]
\centering
\includegraphics[scale=1]{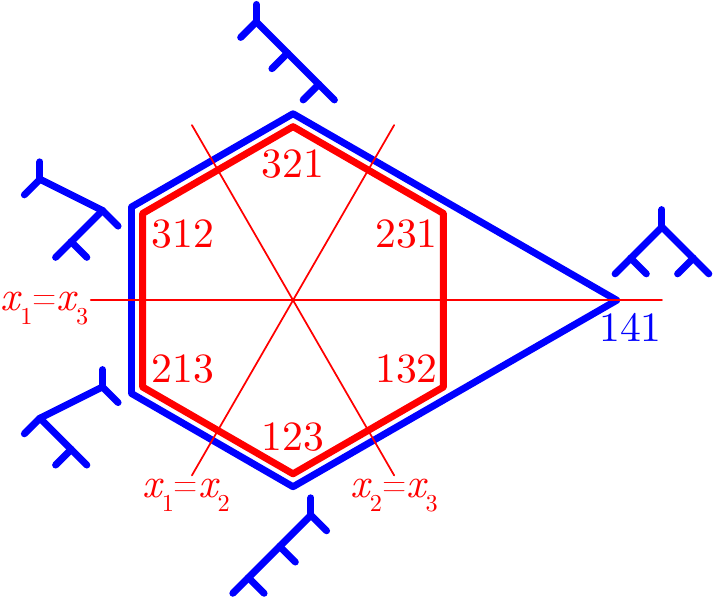}\qquad
\includegraphics[scale=0.6]{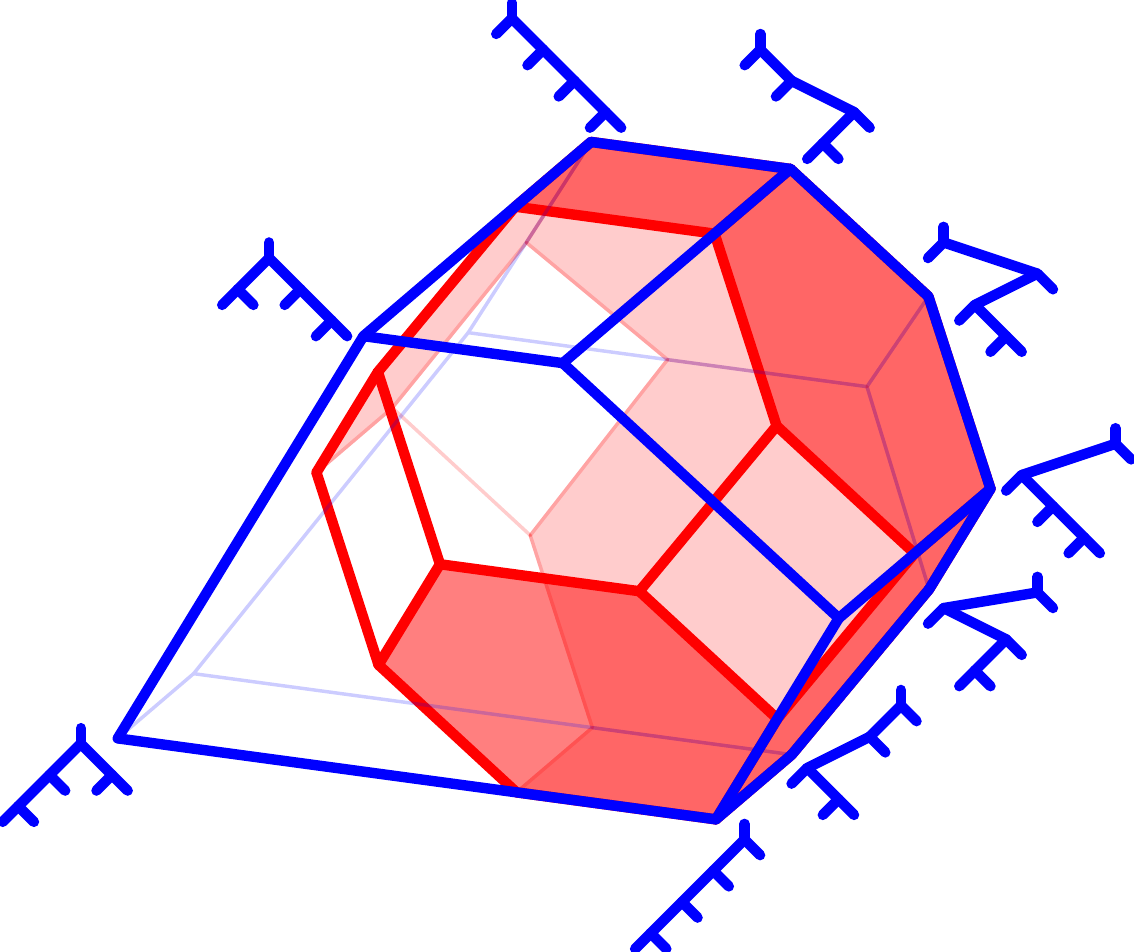}
\caption{\label{fig:assoc-perm} The Associahedron is obtained from the Permutahedron by keeping only faces containing a singleton.}
\end{figure}

\begin{figure}[p]
\centering
\includegraphics[scale=0.5]{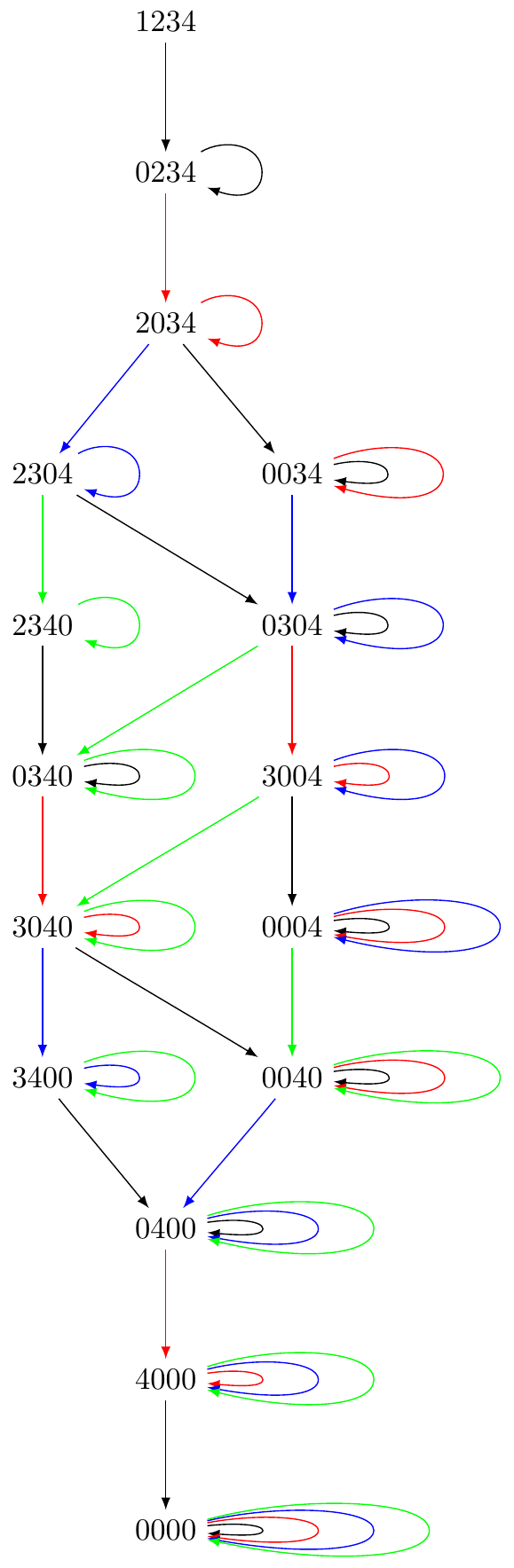} \raisebox{4cm}{$\xrightarrow{\;\;\;\eta\;\;\;}$}
\includegraphics[scale=0.5]{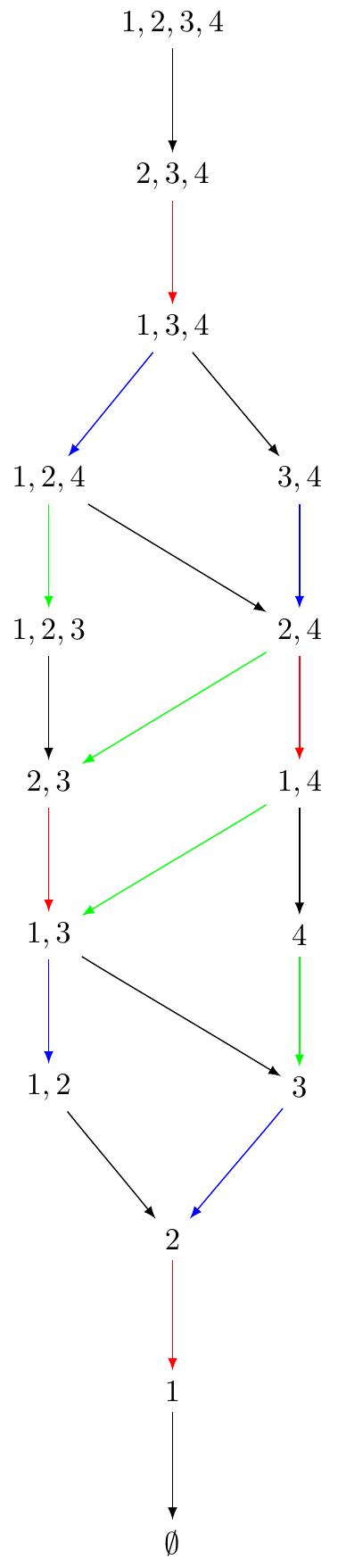}\raisebox{4cm}{$\xrightarrow{\;\;\; \cset\;\;\;}$}
\includegraphics[scale=0.5]{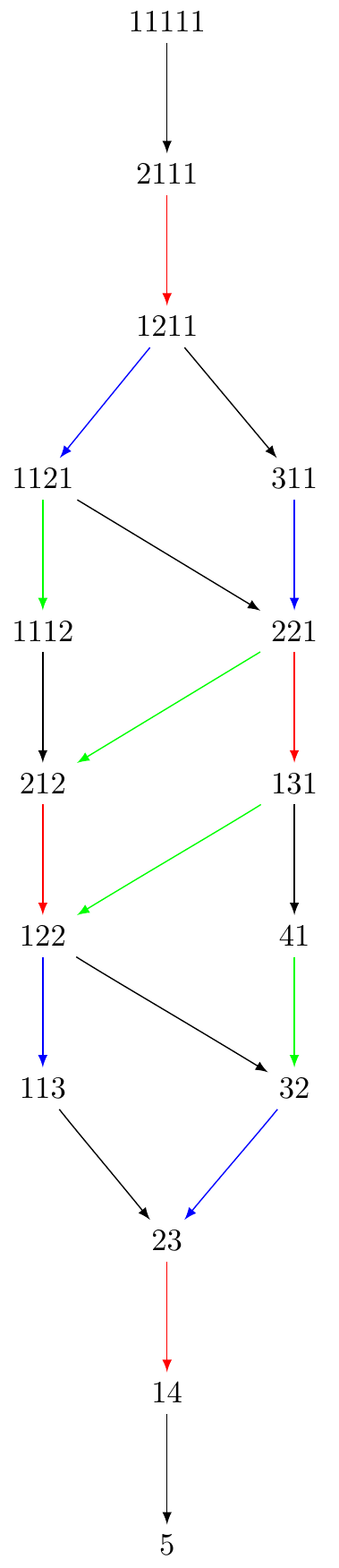}\raisebox{4cm}{$\xrightarrow{\;\;\; \delta\;\;\;}$}
\includegraphics[scale=0.5]{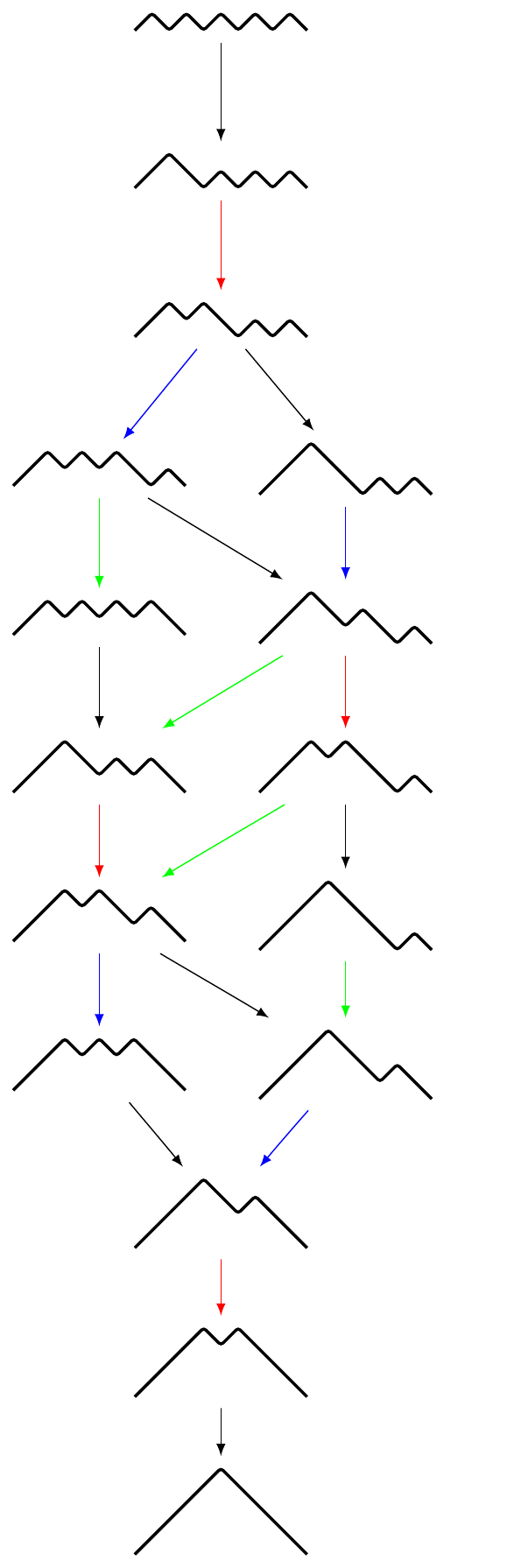}\\[5mm]
\includegraphics[scale=0.5]{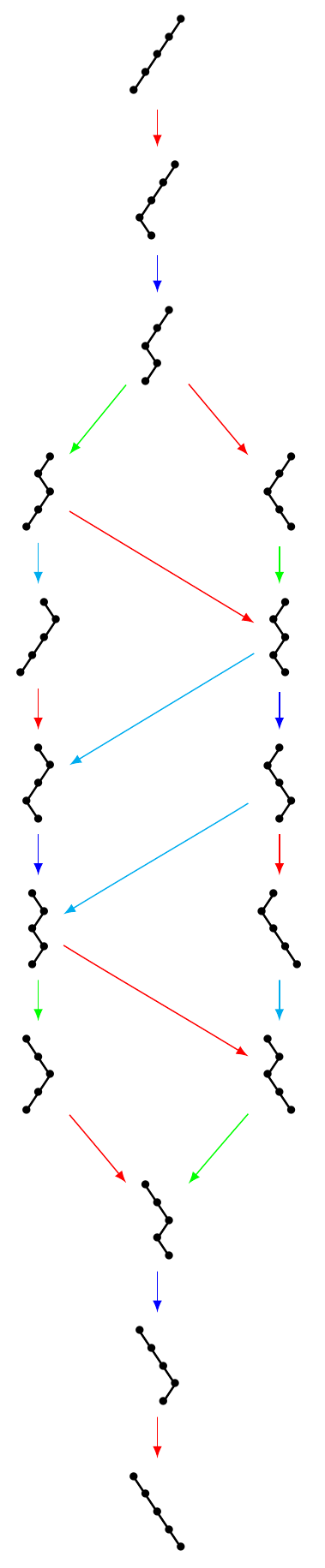}\qquad 
\includegraphics[scale=0.52]{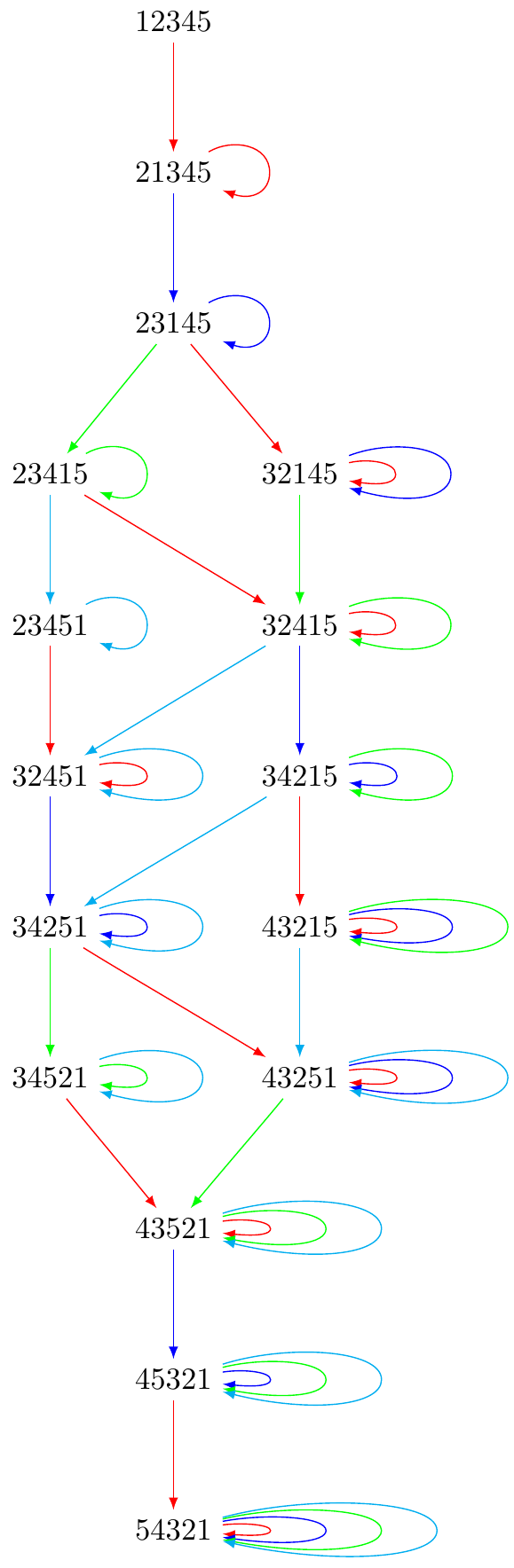}
\hspace{2cm}\raisebox{3cm}{\includegraphics[scale=1]{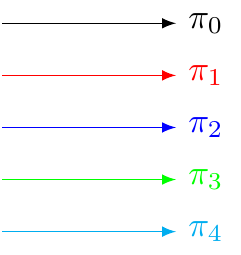}}
\caption{\label{fig:chain_max} The lattice of $\MCR 4$, send to subsets of
  $[4]$, compositions of $5$ and $\MCT 5$. On the second row we represent the
  poset $\MCT 5$ seen on binary trees which are chains, and permutations alone
  in their sylvester class or avoiding $132$ and $312$. We only represent
  loops on the rook vectors and the permutations, the other can be deduced by
  bijection. On the second line we apply generators of $H_5^0$ rather than
  $R_4^0$. Note that the bijection on the generators is only $\pi_i\mapsto
  \pi_{i+1}$.}
\end{figure}

\subsection{Geometrical remarks}
Recall that the right Cayley graph of the symmetric group $\SG{n}$ is the
$1$-skeleton of a polytope namely the permutohedron~\cite[Example
0.10]{ZieglerGeom}. It is defined as the convex hull of the set of points
whose coordinates are permutations. It therefore lives in the hyperplane $\sum
x_i = \frac{n(n+1)}{2}$, so that it is a polytope of dimension $n-1$.

Starting with $n=3$, we can not hope that the right Cayley graph of $R_n$
could be the $1$-skeleton of a polytope. Indeed in $R_n$ the element
$1000\dots$ is always of degree $2$, being linked only to $0000\dots$ and
$0100\dots$, whereas the identity $123\dots$ is of degre $n$. Thus it is
impossible to get a polytope.

\begin{figure}[ht]
  \centering{\includegraphics[width=6cm]{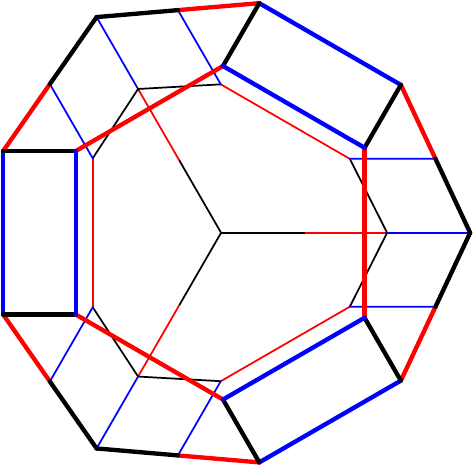}}
  \caption{\label{rookohedron}The Cayley graph of $R_3$ embedded in a
    $3$-dimensional space.}
\end{figure}

Nevertheless, one can consider in a $n$-dimensional space the set of points
whose coordinates are rook vectors (see Figure~\ref{rookohedron}). The
extremal points of its convex hull are the points in
\begin{equation}
  \Stell_n \eqdef
  \{\SG{n}(0\dots 0 k\dots n) \suchthat k\in \interv{1}{n}\}\,.
\end{equation}
This polyedron appeared under the name of stellohedron in~\cite[Figure
18]{MannevillePilaud.2017} where it was defined as the graph associahedron of
a star graph. It is also the secondary polytope of $\Delta_n \cup 2\Delta_n$ (see \cite{GKZ08}),
two concentric copies of a $n$-dimensional simplex, which can also be defined
as
\begin{equation}
  \{e_i \suchthat i\in [n+1]\} \cup
  \{ (n+2) e_i - \mathbf{1} \suchthat i\in [n+1]\}\,.
\end{equation}
So we can see the Cayley graph of $R_n$ as being drawn on the face of the
stellohedron. One can recover this graph from the permutohedron by taking all
its projections on coordinate planes. Indeed, it is just saying that a rook can
be obtained from a permutation replacing some entries by zeros and that edges
are mapped to an identical edge or contracted.

\subsection{A monoid associated to the stellohedron}

The geometrical considerations raise the question whether there is a monoid
structure giving the skeleton of the Stellohedron as Cayley graph. It turns
out that the answer is true. Will we show moreover that there are monoids and
lattice structures on graphs interpolating between the rook case and the
stellar case.

\begin{definition}\label{def-st}
  For any rook $r\in R_n$ denote $M(r):=\max\{i\in\interv{1}{n}\suchthat
  i\notin r\}$ define $\St(r)$ to be the rook obtained by replacing by $0$
  all the letter smaller that $M(r)$ in $r$.
\end{definition}
\begin{example}
  $M(104625)=3$ and thus $\St(104625)=004605$. Similarly $M(10806270)=5$ and
  thus $\St(10806270)=00806070$.
\end{example}
Clearly for any rook $r\in R_n$ then $\St(r)\in\Stell_n$, and for any
$s\in\Stell_n$ one has $\St(s)=s$. We have then proved the following lemma:
\begin{lemma}
  The map $\St$ is a projection (i.e. $\St\circ\St=\St$) on $\Stell_n$.
\end{lemma}
\begin{proposition}\label{prop-prod-compat}
  Denote $\St^0:\Rnz\to\Rnz$ the map corresponding to $\St$ in $\Rnz$, that
  is $\St^0(\pi_r) := \pi_{\St(r)}$.  Then $\St^0$ is compatible with the
  product of $\Rnz$, namely for any $r,s\in R_n$
  \begin{equation}
    \St^0(\St^0(\pi_r)\St^0(\pi_s)) =  \St^0(\pi_r\pi_s).
  \end{equation}
  As a consequence, there is a unique monoid structure on
  $\Snz:=\St^0(\Rnz)=\{\pi_s\suchthat s\in\Stell_n\}$ such that
  such that $\St^0:\Rnz\to\Snz$ is a surjective monoid morphism.
\end{proposition}
\begin{proof}
  It is sufficient to show that for any $i\in\interv{0}{n-1}$ and any $r\in
  R_n$ one has
  \begin{equation}
    \St^0(r\cdot\pi_i)=\St^0(\St^0(r)\cdot\pi_i)
    \qandq
    \St^0(\pi_i\cdot r)=\St^0(\pi_i\cdot \St^0(r))\,.
  \end{equation}
  Indeed these equalities means that the relation $\equiv$ defined by $r\equiv
  s$ if and only if $\St^0(r)=\St^0(s)$ is a monoid congruence. They are
  easily checked on the definition of the left and right
  action~(Definitions~\ref{def_Ro_fun} and~\ref{def_left_action}).
\end{proof}
We now explicit the left and right multiplication of the generator in $\Snz$:
\begin{proposition}
  We denote $\ov\pi_i:=\St^0(\pi_i)$. Then $\Snz$ is generated by
  $\{\ov\pi_i\suchthat 0\leq i<n\}$. And for $i\in\interv{1}{n-1}$ and $s=(s_1
  \dots s_n)\in\Stell_n$, one has $\pi_s\ov\pi_i=\pi_s\pi_i$ and $\pi_{(s_1 \dots
    s_n)}\cdot\ov\pi_0=\pi_u$ where $u$ is the vector obtained by replacing
  all the element less or equal to $s_1$ by $0$ in $s$.

  On the left, the product is given by $\ov\pi_i\pi_{(s_1\dots i+1\dots s_n)}
  = \pi_{(s_1\dots 0\dots s_n)}$ if $i\notin s$ and $i+1\in s$ and
  $\ov\pi_i\pi_r=\pi_i\pi_r$ in all the other cases.
\end{proposition}
We can moreover give a presentation for this new monoid:
\begin{theorem}\label{pres_stellar}
  The stellar monoid $\Snz$ is the quotient of the rook monoid by the
  relations
  \begin{equation}
    \pi_i\pi_{i-1}\dots\pi_1\pi_0\pi_i\ \equiv\ \pi_i\pi_{i-1}\dots\pi_1\pi_0
    \tag{ST}\label{ST}
  \end{equation}
  for $i < n-1$.
\end{theorem}
In order to prove the theorem, we need two preliminary lemmas.
\begin{lemma}
  Relation~\ref{ST} holds in $\Snz$.
\end{lemma}
\begin{proof}
  If we apply both side of Relation~\ref{ST} on the left on the identity
  rook, then $\pi_i$ exchange $i$ and $i+1$ and $\pi_i\pi_{i-1}\dots\pi_1\pi_0$
  kills all letters from $0$ to $i+1$. So both side are equal.
\end{proof}
\begin{lemma}
  In the rook monoid Relations~\ref{ST} implies the following relations:
  \begin{equation}
    \pi_j\ \pi_i\pi_{i-1}\dots\pi_1\pi_0\ \equiv\ \pi_i\pi_{i-1}\dots\pi_1\pi_0
    \tag{ST'}\label{ST'}
  \end{equation}
  for any $0\leq j\leq i < n$.
\end{lemma}
\begin{proof}
  We distinguish three cases:
  \begin{itemize}
  \item $j = 0$. In this case, we have
    \begin{alignat*}{2}
      \pi_0\pi_i\pi_{i-1}\dots\pi_1\pi_0
      & = \pi_i\pi_{i-1}\dots\pi_2\pi_0\pi_1\pi_0&\quad&\text{(by~\ref{R4})}\\
      & = \pi_i\pi_{i-1}\dots\pi_2\pi_1\pi_0\pi_1\pi_0&&\text{(by~\ref{R3})}\\
      & \equiv \pi_i\pi_{i-1}\dots\pi_2\pi_1\pi_0\pi_0&&\text{(mod~\ref{ST} with $i=1$)}\\
      & = \pi_i\pi_{i-1}\dots\pi_2\pi_1\pi_0&&\text{(by~\ref{R1})}.
  \end{alignat*}
  \item $0 < j < i$. In this case, we have
    \begin{alignat*}{2}
      \pi_j\pi_i\pi_{i-1}\dots\pi_1\pi_0
      & = \pi_i\pi_{i-1}\dots\pi_j\pi_{j+1}\pi_j\dots\pi_1\pi_0&\quad&\text{(by~\ref{R4})}\\
      & = \pi_i\pi_{i-1}\dots\pi_{j+1}\pi_j\pi_{j+1}\dots\pi_1\pi_0&\quad&\text{(by~\ref{R2})}\\
      & = \pi_i\pi_{i-1}\dots\pi_{j+1}\pi_j\dots\pi_1\pi_0\pi_{j+1}&\quad&\text{(by~\ref{R4})}\\
      & \equiv \pi_i\pi_{i-1}\dots\pi_2\pi_1\pi_0&&\text{(mod~\ref{ST} with $i=j+1$)}.
  \end{alignat*}
  \item $j = i$. In this case, we just have to apply~\ref{R1}. \qedhere
  \end{itemize}
\end{proof}
\begin{proof}[Proof of Theorem~\ref{pres_stellar}]
  From Corollary~\ref{action_mulr_R0}, for any rook $r\in R_n$, its $R$-code
  $\sfc=(\sfc_1, \dots, \sfc_n)$ verifies (with the notation of
  Definition~\ref{def_word_code}):
  \begin{equation}
    \pi_r = \col{0}{\sfc_1}\cdot\col{1}{\sfc_2}\cdot\dots\cdot\col{n-1}{\sfc_n}\,,
  \end{equation}
  We denote $m=\max\{i\suchthat \sfc_i\leq 0\}$ with the convention that $m=0$
  if all the $\sfc_i$ are positive. Then thanks to relation~\ref{ST'} we
  know that
  \begin{equation}
    \pi_r \equiv
    \col{m-1}{0}\col{m}{\sfc_{m+1}}\cdot\dots\cdot\col{n-1}{\sfc_n}
    \qquad\text{(mod~\ref{ST})},
  \end{equation}
  where the first column is empty if $m=0$. We call \emph{stellar canonical
    word} any word appearing on the right hand side of this equation. In this
  case, the $(\sfc_i)_{i>m}$ verify $0<\sfc_i\leq i$, so that there are
  \begin{equation}
    \sum_{m=0}^n (m+1)\dots(n-1)n
    = \sum_{m=0}^n\frac{n!}{m!} = \left|\Snz\right|
  \end{equation}
  stellar canonical words. We have shown that
  each rook is congruent to a stellar canonical words modulo~\ref{ST} and that
  they are as many stellar canonical words as element of $\Snz$. As a
  consequence Relation~\ref{ST} is the only relations needed to get $\Snz$
  from $\Rnz$.
\end{proof}

\subsubsection{The stelloid lattice}

By analogy to the Rook lattice, we might wonder if the $\RR$-order or $\LL$-order
of $\Snz$ are lattices (on the contrary to rooks, they are not isomorphic). It
turns out that the $\LL$-order is a lattice. See Figure~\ref{stellohedron} for a
picture.
\begin{figure}[ht]
  \centering
    \includegraphics[width=8cm]{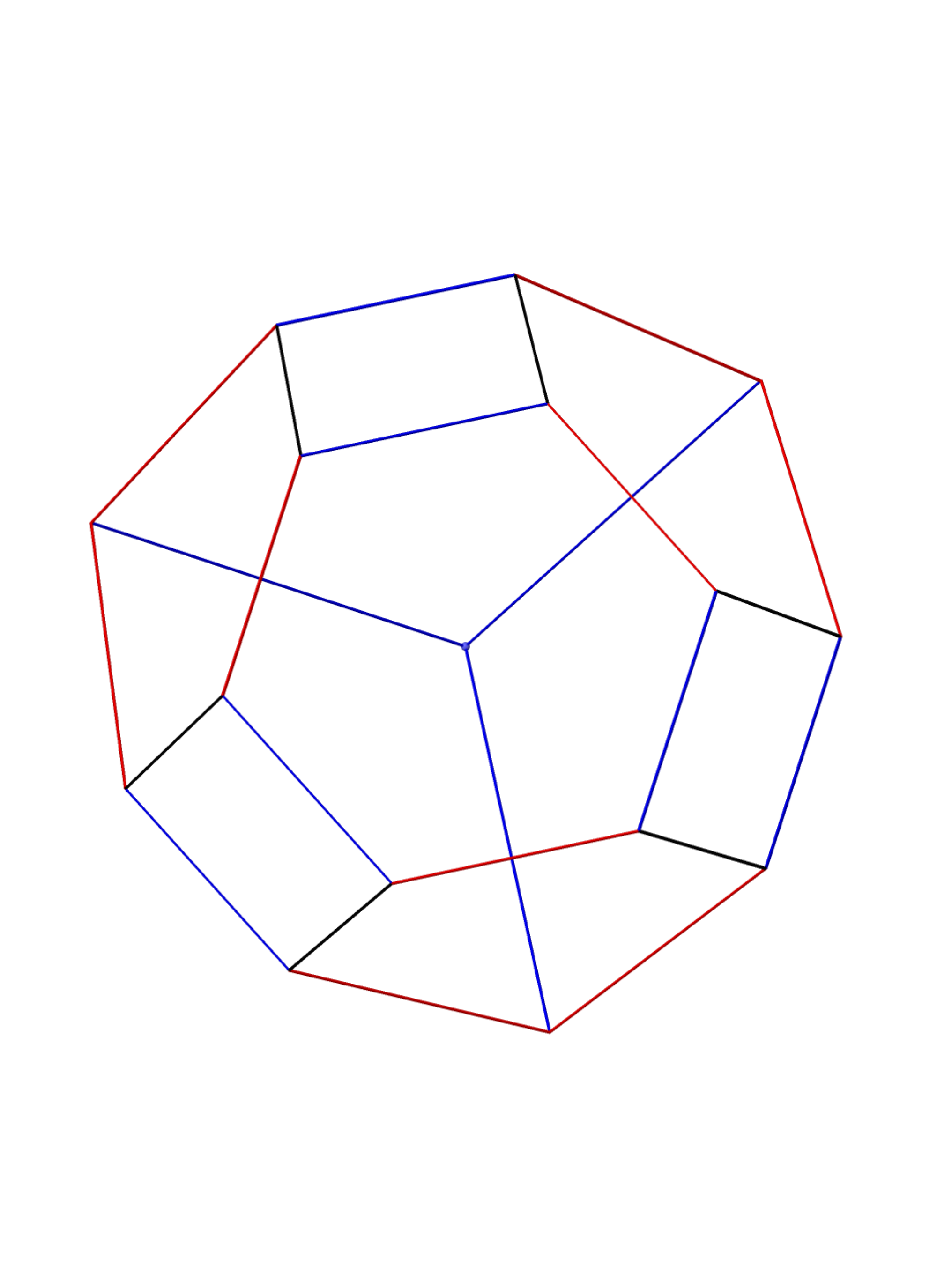}
    \includegraphics[width=6cm]{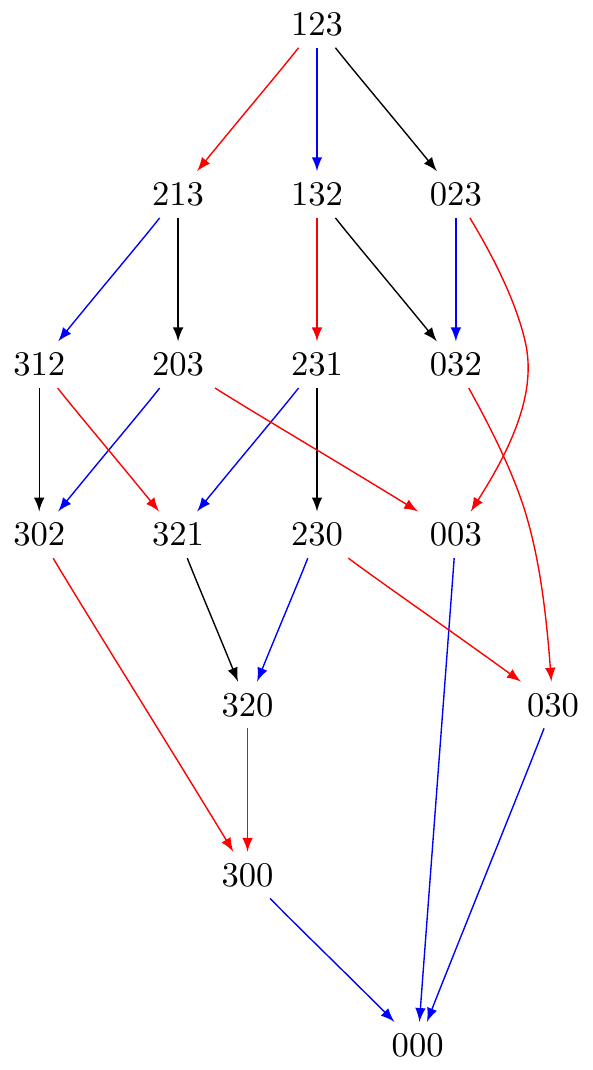}
  \caption{\label{stellohedron}The left order of $\Stell_3^0$}
\end{figure}
We will show actually a stronger result:
\begin{theorem}\label{theo-stell-sublattice}
  The $\LL$-order on $\Snz$ is a sublattice of the $\LL$-order of $\Rnz$.
\end{theorem}
\begin{proof}
  We need to conjugate all the rooks to pass to the left order. The conjugate
  of a stellar rook $r$ is a rook such that all the zeroes are at the
  beginning. Equivalently this means that in its rook triple $(S_r, I_r,
  Z_r)$, the $Z_r$ function is the zero function. Now looking at the algorithm
  for computing the meet and join of two rooks, we have that
  \begin{alignat}{1}
    Z_{u\meR v}(x) &:=
    \max\{Z_u(i), Z_v(i) \mid i = x \text{ or } (x, i)\in I_{u\meR v} \}\,, \\
    Z_{u\jnR v}(x) &:=
    \min\{Z_u(i), Z_v(i) \mid i = x \text{ or } (x, i)\in\Delta\setminus I_{u\meR v}\}\,.
  \end{alignat}
  As a consequence both $Z_{u\meR v}$ and $Z_{u\jnR v}$ are zero functions so
  that $u\meR v$ and $u\jnR v$ are conjugate stellar rooks too.
\end{proof}
\begin{remark}
  The preimage of the stellar rook~$300$ is~$\{300, 301, 310\}$ which is not a
  interval of the $\LL$-order. As a consequence, $\St^0$ can't be a lattice
  morphism and the $\LL$-order of $\Snz$ is a not a lattice quotient of
  the $\LL$-order of $\Rnz$.
\end{remark}

\subsubsection{Higher order Stelloid monoid and lattices}

The proofs of the two previous theorem makes it clear that the $\Snz$ monoid
together with its $\LL$-lattice is a particular case of a more general
construction: for $k\geq0$, define $\St_k$ the map from rooks to rooks which
replace by $0$ all the letter $i$ such that there is $k$ or more missing
letter larger that $i$ (the usual $\St$ map is the case $k=1$). For example
$\St_2(3057016)=(3057006)$ and $\St_2(3407016)=(3407006)$. Also,
$\St_i\circ\St_j=\St_{\max(i,j)}$. So that for all $n$, we have the inclusion
of sets:
\begin{equation}\label{seq-set-st-k}
  \{0^n\}=\St_0(R_n)\subset\St_1(R_n)\subset\St_2(R_n)
  \subset\dots\subset\St_n(R_n)=R_n\,.
\end{equation}
\begin{figure}[ht]
  \centering
    \includegraphics[height=8cm]{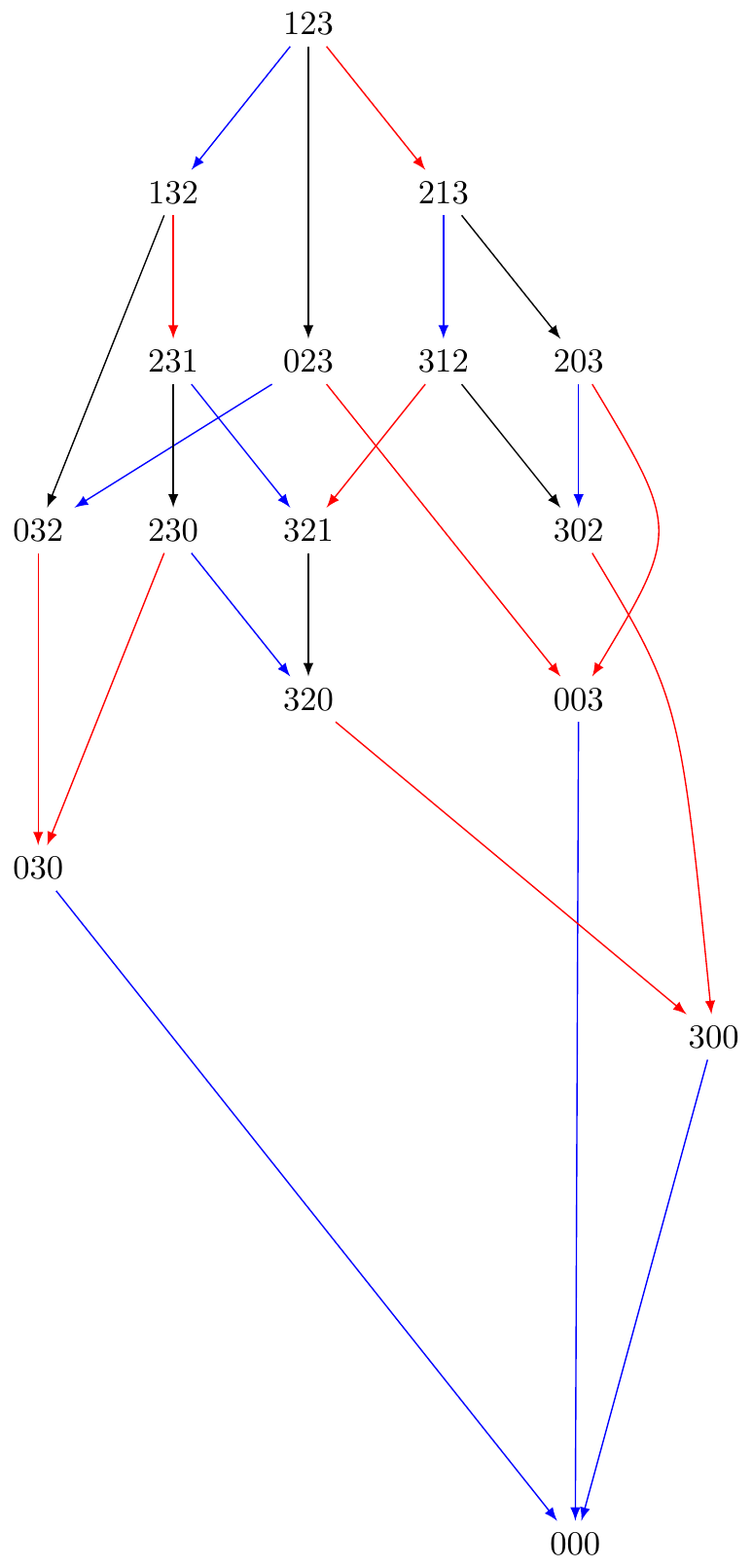}\qquad\quad\
    \includegraphics[height=8cm]{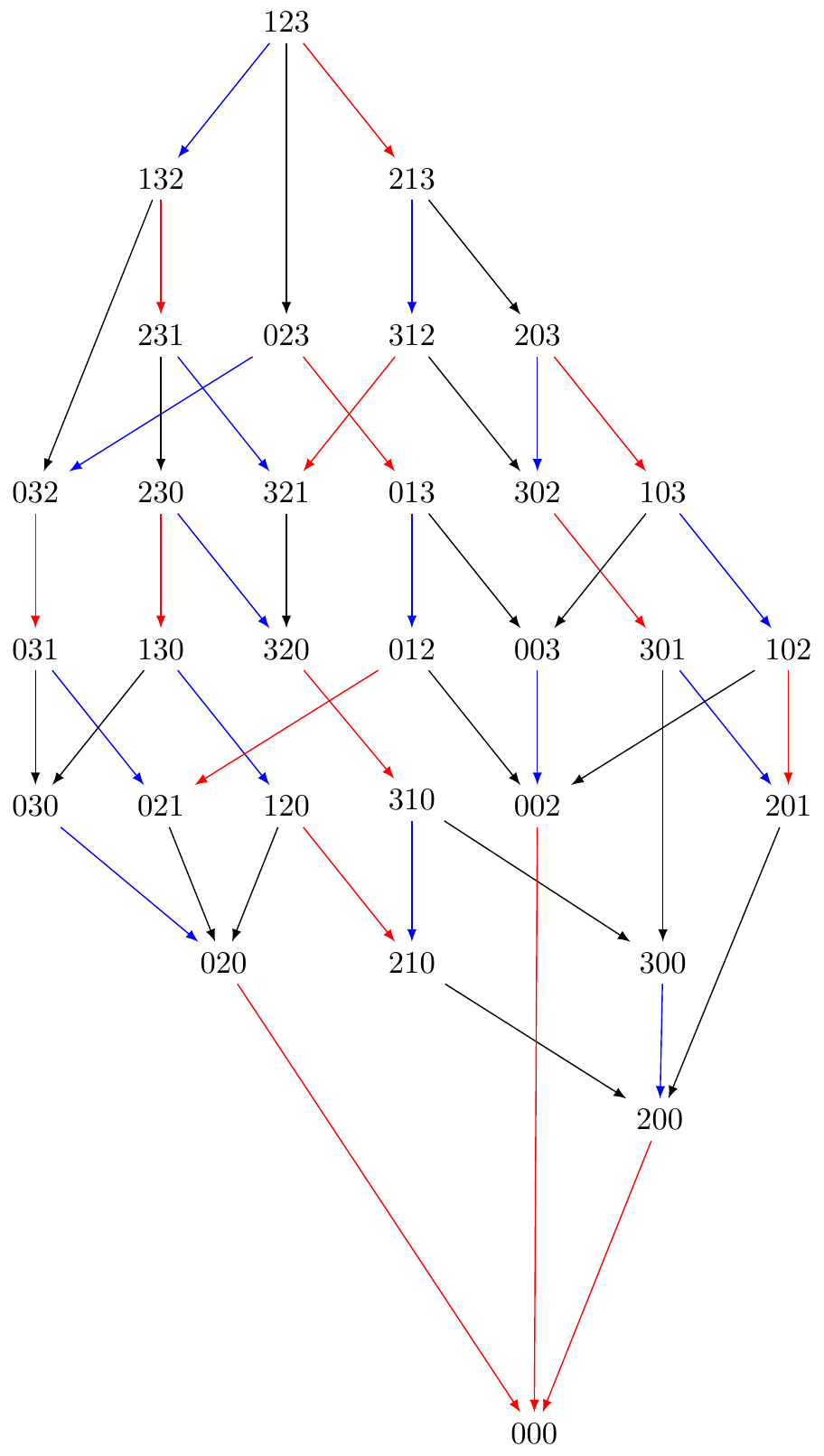}\qquad
    \includegraphics[height=8cm]{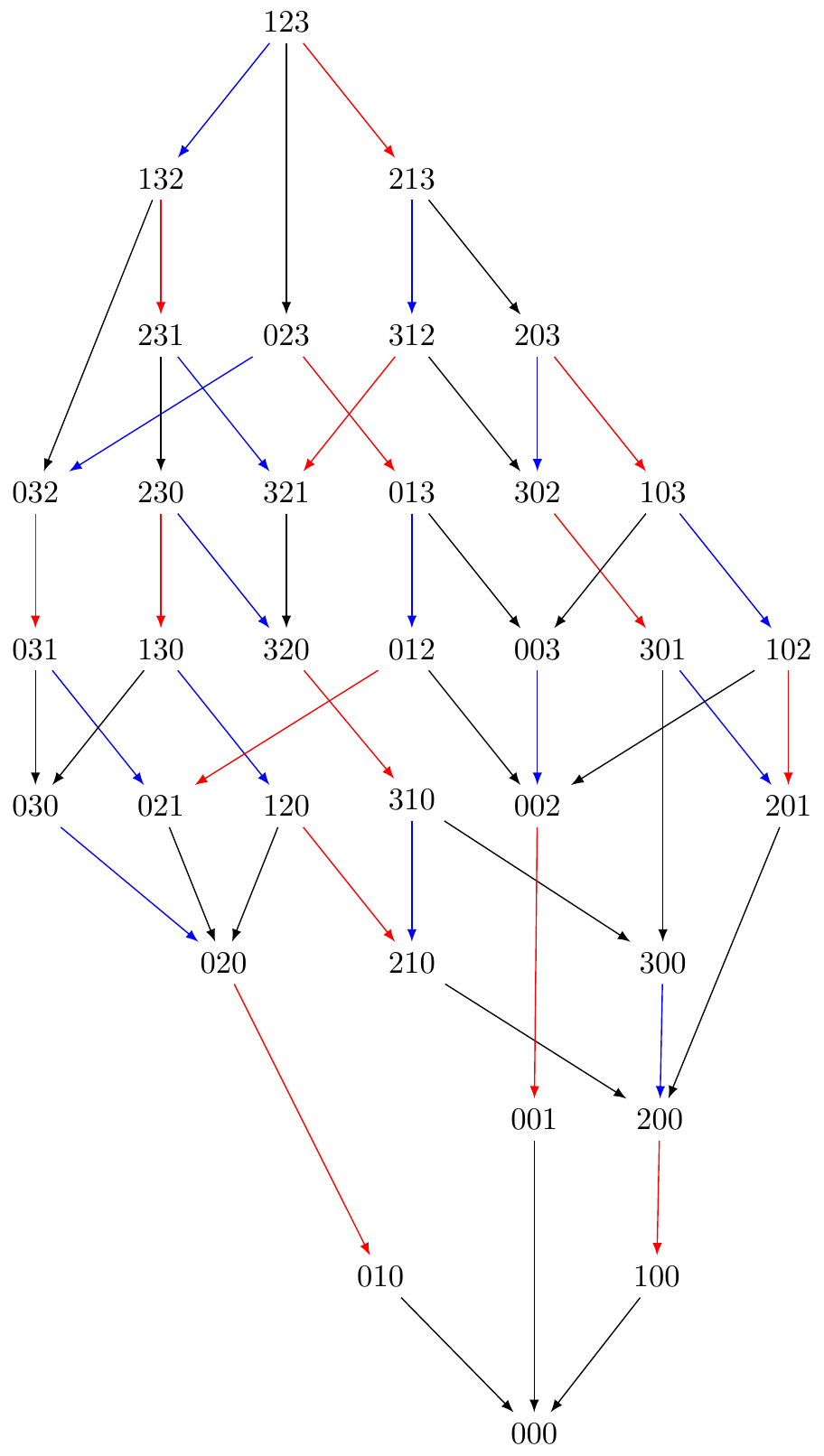}\\
    \includegraphics[width=5cm, trim={3cm 6cm 3cm 6cm}, clip]{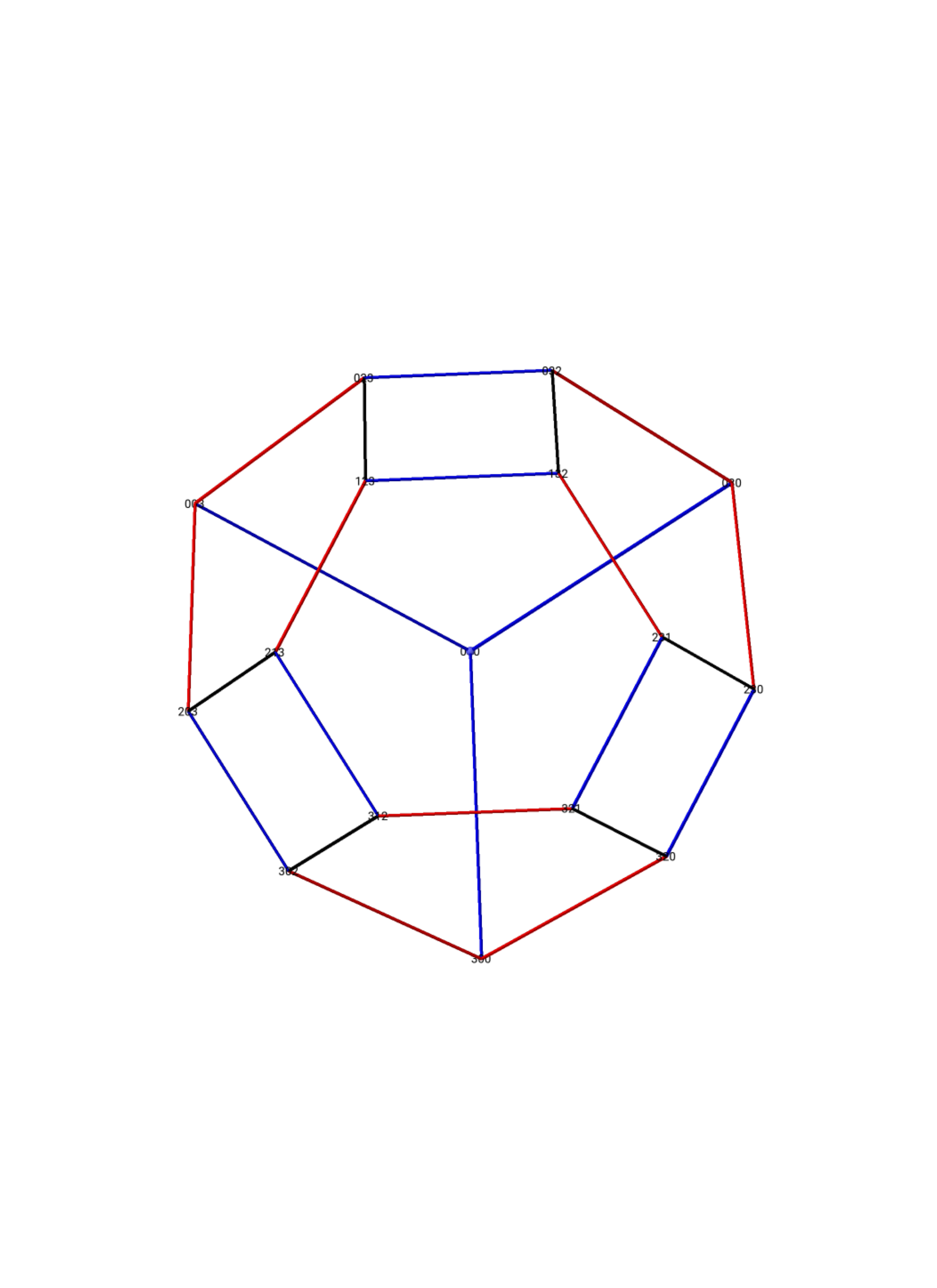}
    \includegraphics[width=5cm, trim={3cm 6cm 3cm 6cm}, clip]{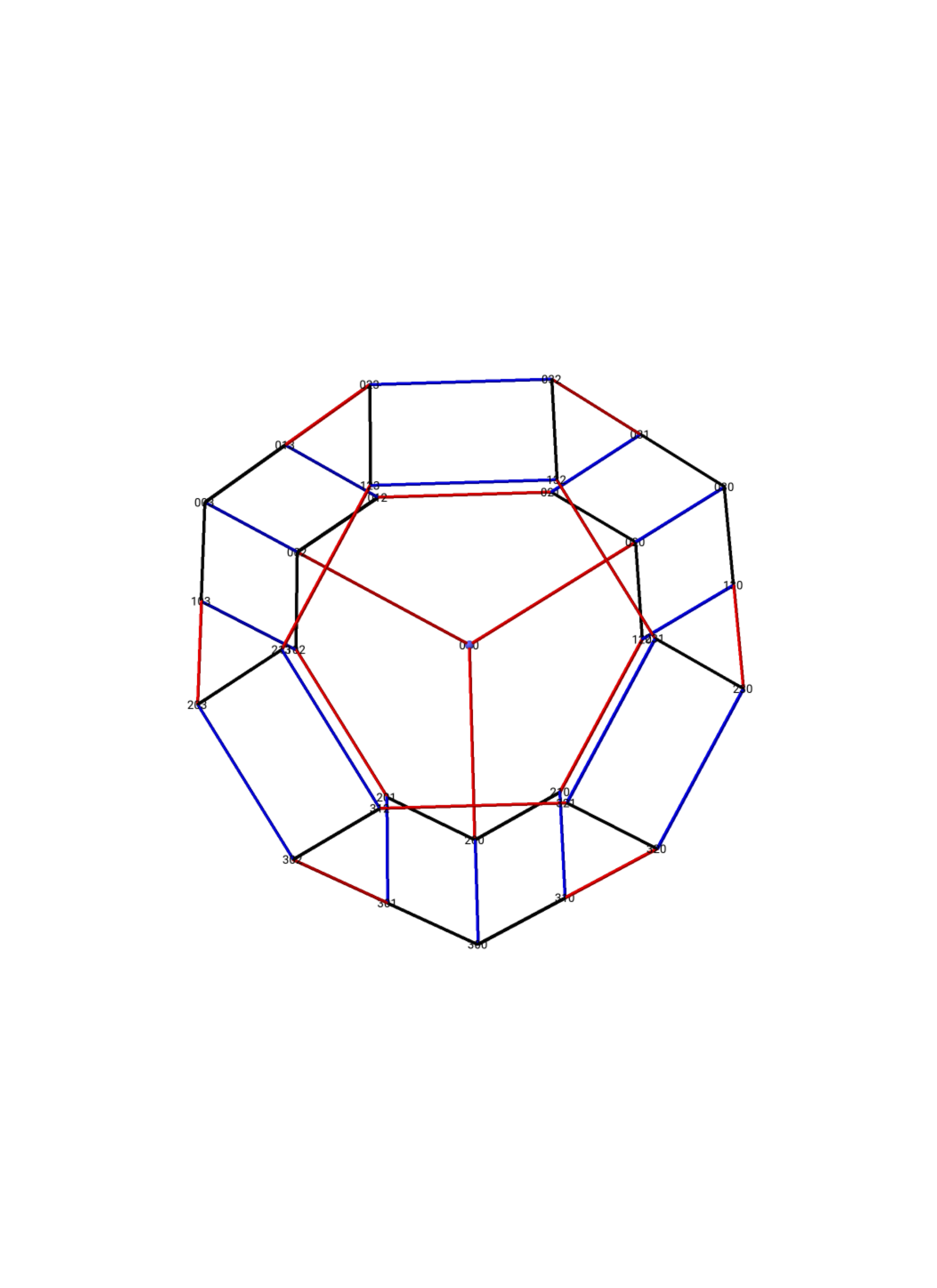}
    \includegraphics[width=5cm, trim={3cm 6cm 3cm 6cm}, clip]{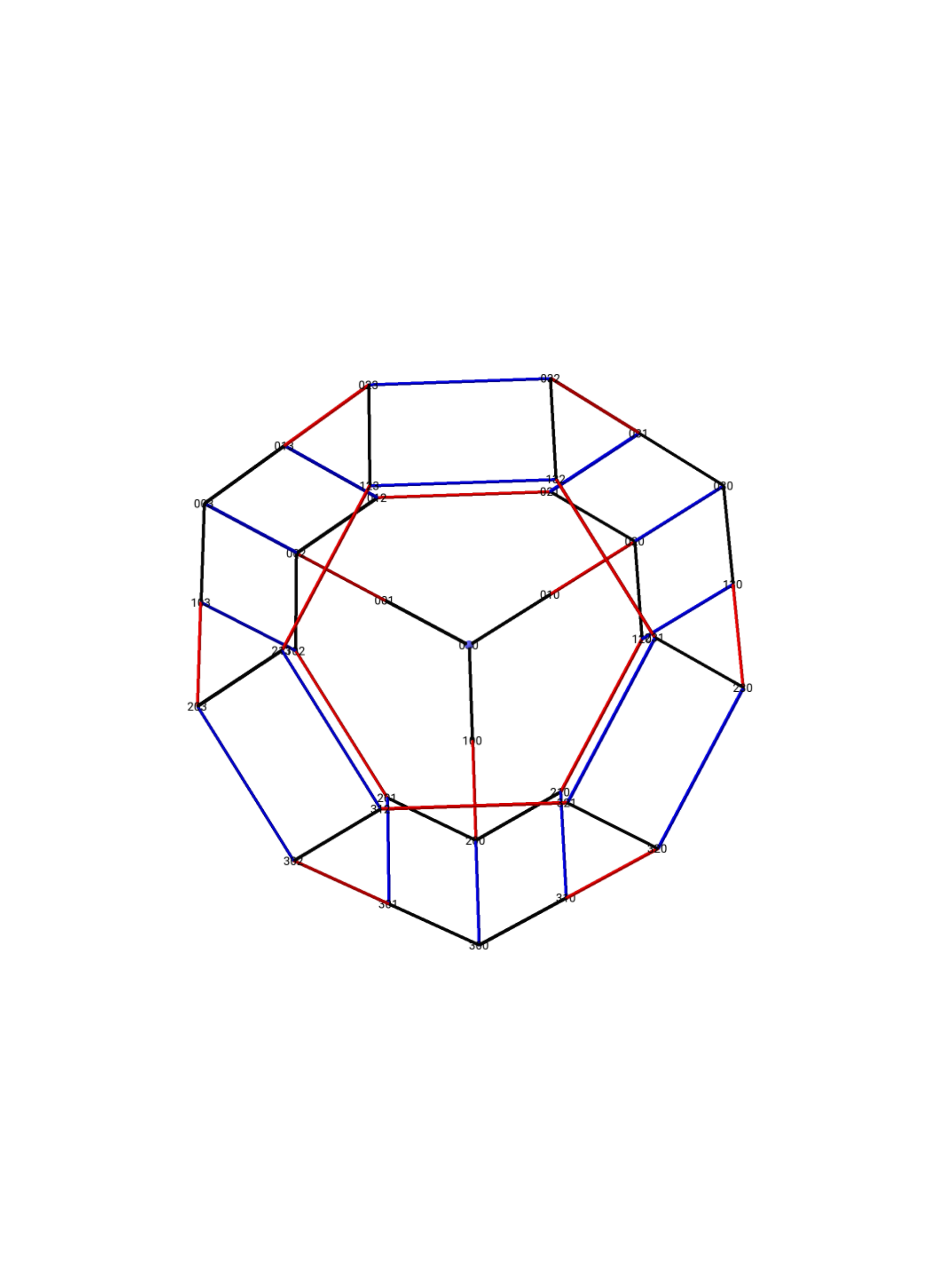}
    \caption{\label{stellohedron-lattgeom}The left order of $\St_k^0(R_3)$ for
      $k=1,2,3$}
\end{figure}

The following array give the cardinality of $\St_k(R_n)$.
\begin{equation*}
  \begin{array}{c|rrrrrrrr}
    k\backslash n& 0 & 1 & 2 &  3 &   4 &    5 &     6 &      7\\
    \hline
    0 & 1 & 1 & 1 &  1 &   1 &    1 &     1 &      1\\
    1 &   & 2 & 5 & 16 &  65 &  326 &  1957 &  13700\\
    2 &   &   & 7 & 31 & 165 & 1031 &  7423 &  60621\\
    3 &   &   &   & 34 & 205 & 1456 & 11839 & 108214\\
    4 &   &   &   &    & 209 & 1541 & 13165 & 127289\\
    5 &   &   &   &    &     & 1546 & 13321 & 130656\\
    6 &   &   &   &    &     &      & 13327 & 130915\\
    7 &   &   &   &    &     &      &       & 130922\\
  \end{array}
\end{equation*}
Then the proof of Proposition~\ref{prop-prod-compat} and
Theorem~\ref{theo-stell-sublattice} generalize to this new cases:
\begin{theorem}
  Denote $\St_k^0:\Rnz\to\Rnz$ the map corresponding to $\St_k$ in
  $\Rnz$. Then this map is a surjective monoid morphism to $\St_k^0(\Rnz)$.
  Moreover, the $\LL$-order of $\St_k^0(\Rnz)$ is a sublattice of the
  $\LL$-order of $\Rnz$.
\end{theorem}
Hence Equation~\ref{seq-set-st-k} is actually a sequence of inclusion of
lattices. It also give rise to the following sequence of monoid morphisms:
\begin{equation}
  \{0^n\}=\St_0(R_n^0)\leftarrow\St_1(R_n^0)\leftarrow\St_2(R_n^0)
  \leftarrow\dots\leftarrow\St_n(R_n^0)=R_n^0\,.
\end{equation}
Figure~\ref{stellohedron-lattgeom} shows the three consecutive quotients
of~$\St_k^0(R_3)$ together with their geometric counterpart.

We finally remark that the quotient morphism are in the opposite direction of
the inclusion of lattices of Equation~\ref{seq-set-st-k}. This suggest some
kind of duality, but we haven't been able to give a formulation.
\newpage

\section{Representation theory of the \texorpdfstring{$0$}{0}-Rook monoid \texorpdfstring{$\Rnz$}{Rn0}}
\label{sec-rep-theo}

The goal of this section is to investigate the representation theory of
$\Rnz$. We write $\C[\Rnz]$ the monoid algebra of $\Rnz$. In the sequel of the
article $P_1$ will rather be denoted by $\pi_0$. Moreover, $r$ will usually
denote an element of $\Rnz$. Also, we know from Corollary~\ref{action_mulr_R0}
and Proposition~\ref{action_mull_R0} that for any $r\in \Rnz$ there is a
unique rook $\w{r}\eqdef 1_n\cdot r=r\cdot 1_n$ such that $\pi_{\w{r}} =
r$. So when there is a need to distinguish, we will denote in normal letter
$r$ the elements of the monoid and in boldface as $\w{r}$ their associated
rooks.  \bigskip

We start by summarizing the main results (in particular
Corollary~\ref{action_mulr_R0}) of the Section~\ref{sec-def} which concerns
the representations:
\begin{proposition}\label{isoRegNat}
  The maps
  \begin{equation*}
    f_R : \left|\begin{array}{ccc}
        \C [\Rnz] & \longrightarrow & \C R_n\\
        x & \longmapsto & 1_n\cdot x\,,
      \end{array}\right.
    \qandq
    f_L : \left|\begin{array}{ccc}
        \C [\Rnz] & \longrightarrow & \C R_n\\
        x & \longmapsto & x\cdot 1_n\,,
      \end{array}\right.
  \end{equation*}
  extended by linearity, are two isomorphisms of representations of $\Rnz$
  between the left and right regular representations and the
  natural one (acting on $R_n$).
\end{proposition}

\subsection{Idempotents and Simple modules}

As for any algebra, the representation theory of $\C[\Rnz]$ (or equivalently
$\Rnz$) is largely governed by its idempotents, however since~$\Rnz$ is a
$\JJ$-trivial monoid, as shown in~\cite{DHST}, it is sufficient to look for
idempotents in the monoid $\Rnz$ itself.

\begin{proposition}
  For any $S\subset \interv{0}{n-1}$, we denote $\pi_S$ the zero
  of the so-called parabolic submonoid generated by $\{\pi_i \mid i\in S\}$.
\end{proposition}
\begin{proof}
  This submonoid is finite since the ambient monoid~$\Rnz$ is finite. By
  Proposition~\ref{Jtrivial} it contains a unique minimal element for the
  $\JJ$-order, which is a zero.
\end{proof}

\begin{proposition}\label{formeidempotentmatrice}
  For any $S\subset\interv{0}{n-1}$, write
  $S^c\eqdef\interv{0}{n-1} \setminus S$ its complement and
  $I=C(S^c)=(i_1,\dots,i_\ell)$ its associated extended composition. Then
  $\pi_S=\pi_{\w{r}}$ where $\w{r}$ is the block diagonal rook matrix of size
  $n$ whose block are anti diagonal matrices of $1$ of size
  $(i_1,\dots,i_\ell)$, except the first block which is a zero matrix.
\end{proposition}
Note that if $0\notin S$ then the first part of $I$ is zero, so that the first
zero block is of size $0$ and therefore vanishes.
\begin{example}
If $n = 12$ and $S=\{0,1,2, 5, 7,8, 11\}$. Then $S^c=\{3,4,6,9,10\}$ so that
$I=C(S^c)=(3,1,2,3,1,2)$. Similarly, if $T=\{2, 4,5, 7,8,9,\}$, then
$T^c=\{0,1,3,6,10,11\}$ so that $J=C(T^c)=(0,1,2,3,4,1,1)$. Therefore the
associated matrices are:
$$
\setlength{\arraycolsep}{0.5mm}
\newcommand\Tstrut{\rule{0pt}{1.7ex}}         
\pi_S = \left(\begin{array}{c|c|c|c|c|c}
\begin{smallmatrix}
0 & 0 & 0\\
0 & 0 & 0\\
0 & 0 & 0
\end{smallmatrix} & & & & & \\[6pt]
\hline
& \begin{smallmatrix}1\end{smallmatrix} & & & & \\[0pt]
\hline
& & \begin{smallmatrix}
 0 & 1\Tstrut\\
 1 & 0
\end{smallmatrix} & & & \\[3pt]
\hline
 & & & \begin{smallmatrix}
 0 & 0 & 1\Tstrut \\
 0 & 1 & 0\\
 1 & 0 & 0
\end{smallmatrix}& & \\[6pt]
\hline
 & & & & \begin{smallmatrix}1\end{smallmatrix} & \\[-1pt]
\hline
 & & & & & \begin{smallmatrix}
 0 & 1\Tstrut \\
 1 & 0
 \end{smallmatrix}
\end{array}\right)
\qandq
\pi_T = \left(\begin{array}{c|c|c|c|c|c}
\begin{smallmatrix}1\end{smallmatrix} & & & & & \\[-1pt]
\hline
& \begin{smallmatrix}
0 & 1\Tstrut \\
1 & 0 \\
\end{smallmatrix} & & & & \\[3pt]
\hline
& & \begin{smallmatrix}
 0 & 0 & 1\Tstrut \\
 0 & 1 & 0\\
 1 & 0 & 0
\end{smallmatrix} & & & \\[6pt]
\hline
 & & & \begin{smallmatrix}
 0 & 0 & 0 & 1\Tstrut \\
 0 & 0 & 1 & 0 \\
 0 & 1 & 0 & 0\\
 1 & 0 & 0 & 0
\end{smallmatrix}& & \\[10pt]
\hline
 & & & & \begin{smallmatrix}1\end{smallmatrix} & \\[-1pt]
\hline
 & & & & & \begin{smallmatrix}1\end{smallmatrix}\\[-1pt]
\end{array}\right)
$$
\end{example}
\begin{proof}
  We fix some $S$ and consider $\w{r}$ the associated rook matrix. The block
  diagonal structure ensures that $\pi_{\w{r}}$ belongs to the parabolic
  submonoid $\langle\pi_i \mid i\in S\rangle$. Indeed, suppose that there is a
  reduced word $\un{w}$ for $\pi_{\w{r}}$ with some $w_i\notin S$. Recall, that from
  Corollary~\ref{action_mulr_R0}, this means that $1_n\cdot \un{w}=\w{r}$. Choose
  the smallest such $i$. There are two cases whether $w_i=\pi_0$ or not.
  \begin{itemize}
  \item if $w_i=\pi_0$ with $0\notin S$, then when computing $1_n\cdot
    w_1\cdots w_{i-1} \cdot w_i$, the action of $\pi_0$ will be to kill a
    column. In this case, the killed column will never appear again so that
    there is no way to get the correct matrix.
  \item if $w_i=\pi_i$ with $i\neq 0$, when computing $1_n\cdot w_1\cdots
    w_{i-1} \cdot w_i$, the action of $w_i$ is to exchange two columns from
    two different blocks. However, acting by any $\pi_j$ will never exchange
    those two columns again, so that it is not possible to get them back in the
    correct order.
  \end{itemize}
  Hence, we have proven that $\un{w}$ only contains $\pi_i$ with $i\in S$ that
  is $r\in \langle\pi_i \mid i\in S\rangle$. Furthermore, using the action
  on matrices one sees that $r\cdot\pi_i=r$ or equivalently that
  $\pi_{\w{r}}\pi_i=\pi_{\w{r}}$ if and only if $i\in S$. This shows that
  $\pi_{\w{r}}$ is the zero of $\langle\pi_i \mid i\in S\rangle$.
\end{proof}
\begin{remark}
  If we decompose the set $S$ into its maximal components of consecutive letters $S_1
  \cup S_2 \cup \dots \cup S_r$, then $\pi_S = \prod_{1\leq i\leq r}
  \pi_{S_i}$ where the product commutes. Moreover, if $0\in S$ then
  $\pi_{S_1}=P_m$ where $m$ is the size of the first block.
\end{remark}
During the proof, we got the following Lemma:
\begin{lemma}\label{actiongenerateurssurpiJ}
  Let $S \subset \interv{0}{n-1}$. Then $\pi_S\pi_i = \pi_S = \pi_i\pi_S$ if
  $i\in S$, and $\pi_S\pi_i\neq\pi_S$ and $\pi_i\pi_S\neq\pi_S$ otherwise.
\end{lemma}

\begin{proposition}\label{prop-Rn0-idemp}
  The monoid $\Rnz$ has exactly $2^n$ idempotents: these are the zeros of
  every parabolic submonoid.
\end{proposition}
\begin{proof}
  We already know that $\Rnz$ has at least $2^n$ idempotents. We now have
  to prove this exhaust the idempotents of $\Rnz$.

  Let $e$ an idempotent of $\Rnz$. Recall that $\cont(e)$ is the set of the
  $\pi_i$ with $i\in\interv{0}{n-1}$ appearing in any reduced word of $e$. Let
  us show that $e=\pi_{\cont(e)}$, that is the zero of the parabolic submonoid
  $\langle\pi_i \mid i\in\cont(e)\rangle$. Indeed for $a\in\cont(e)$, one can
  write $e=\un{u}a\un{v}$ for some $\un{u}$ and $\un{v}$ in $\Rnz$. By
  definition of the $\JJ$-order, this means that $e\leq_\JJ
  a$. Using~\cite[Lemma~3.6]{DHST}, this is equivalent to $ea=e$ and to
  $ae=e$, so that $e$ is stable under all its support.
\end{proof}

\begin{theorem}\label{theo-Rn0-simple}
  The monoid $\Rnz$ has $2^n$ left (and right) simple modules, all
  one-dimensional, indexed by the subsets of $\interv{0}{n-1}$.  Let $S
  \subset\interv{0}{n-1}$. Its associated simple module $S_S$ is the
  one-dimensional module generated by $\varepsilon_S$ with the following
  action of generators:
  \begin{equation}
    \pi_i\cdot\varepsilon_S
    = \begin{cases}
      \varepsilon_S & \text{ if }i\in S,\\
      0 & \text{ otherwise.}
    \end{cases}
  \end{equation}
\end{theorem}
\begin{proof}
  We apply \cite[Proposition~3.1]{DHST} using Lemma
  \ref{actiongenerateurssurpiJ}.
\end{proof}

Recall that we write $x^\omega$ any sufficiently large power of $x$ which
becomes idempotent, and that the star product of two idempotents is defined as
$e \ast f = (ef)^{\omega}$. This endows the set of idempotents with a
structure of a lattice where $\ast$ is the meet~\cite[Theorem~3.4]{DHST}. We
now explicitely describe this lattice:
\begin{proposition}\label{prod_of_idemp}
  Let $S, T \subset \interv{0}{n-1}$. Then $\pi_S \ast \pi_T = \pi_{S\cup T}$.
\end{proposition}
\begin{proof}
  First we note that $\pi_S\ast\pi_T$ is inside the parabolic $S\cup T$.  It
  is enough to show that it is a zero of this submonoid, and then conclude by
  unicity. The product formual is is a consequence of~\cite[Lemma~3.6]{DHST}.
\end{proof}
\begin{corollary}
  The quotient $\C[\Rnz]/\rad(\C[\Rnz])$ is isomorphic to the
  algebra of the lattice of the $n$-dimensionnal cube.
\end{corollary}

\subsection{Indecomposable projective modules}

Recall that the indecomposable projective $\Hnz$-modules are spanned by descent
classes (see Section~\ref{subsec-reptheo-hn0} and reference therein for more
details). Extending the definition from the Hecke monoid, we define left
and right $R$-descents sets of a rook as:
\begin{equation}
  D_R(r) = \lbrace 0 \leq i \leq n-1 \, \mid  \, r\pi_i = r\rbrace
  \text{\quad (resp.\ }
  D_L(r) = \lbrace 0 \leq i \leq n-1 \, \mid \, \pi_i r = r\rbrace)\,.
\end{equation}
\begin{example}
  Let $r = \w{0423007}\in \Rnz$. We have $0<4\geq 2<3\geq 0\geq 0 <7$, and the
  first letter is $0$. So $D_R(r) = \lbrace 0, 2, 4,5\rbrace$.
\end{example}

\begin{notation}
  We choose to represent an element $r\in \Rnz$ by a ribbon notation the usual
  way, with the difference that two zeros are vertical and not horizontal:
  $\gyoung(;0,;0)$ and not $\gyoung(;0;0)$.
\end{notation}
This change of convention compared to \textit{e.g.}~\cite{KrobThibon.1997} is
due to our choice of taking the $\pi$'s and not the $T_i$'s for generators of
the Hecke algebra. As a consequence, the eigenvalues $0$ and $1$ are
exchanged.

For example, $r = \w{0423007}$ is represented by the ribbon
\scalebox{0.7}{\raisebox{20pt}{$\gyoung(;0;4,:;2;3,::;0,::;0;7)$ }}
Figure~\ref{Rdescclass4} shows the ribbon together with their associated
descent sets. Figure~\ref{Rdescclass4Bool} depicts the associated boolean
lattice.
\begin{figure}[ht]
\centering
\includegraphics[scale=0.7]{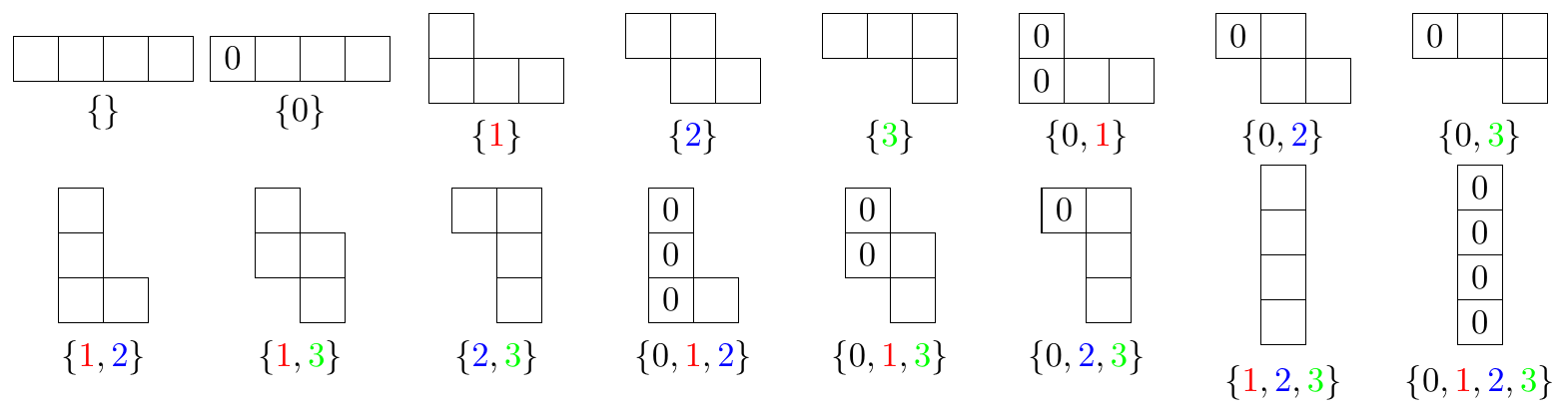}
\caption{\label{Rdescclass4}$R$-descent sets for $R_4^0$.}
\end{figure}
\begin{figure}[ht]
\centering
\includegraphics[scale=0.8]{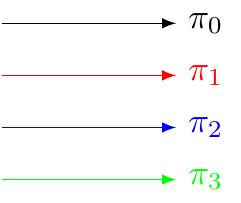}
\hspace{-2cm}
\includegraphics[scale=1]{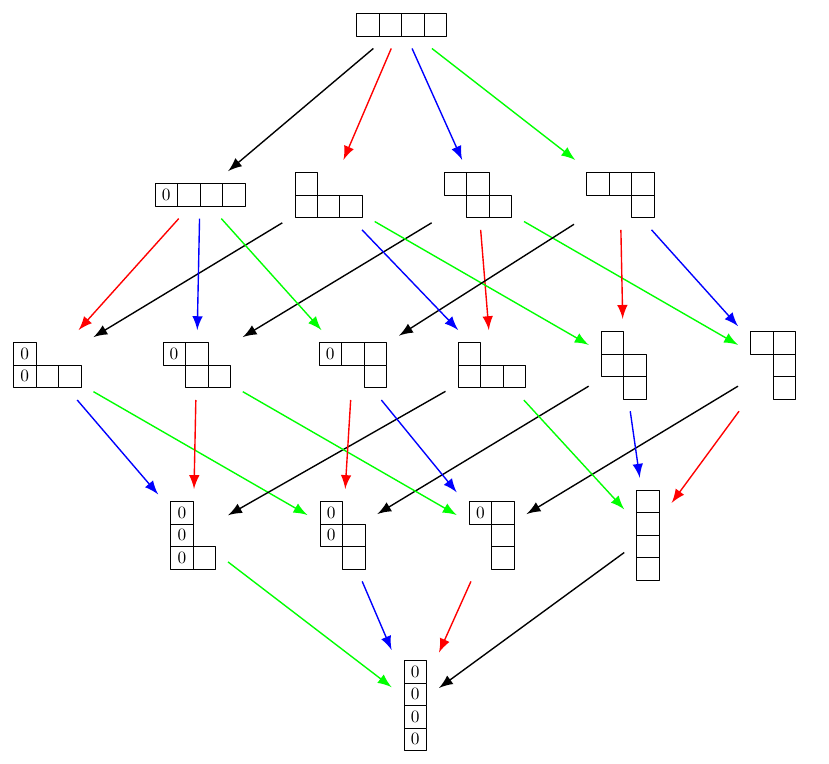}
\caption{\label{Rdescclass4Bool}The lattice of $R$-descent sets for $R_4^0$.}
\end{figure}
With this notation we can easily find the idempotents of each $R$-descent
set:
\begin{proposition}\label{forme_idempotent_Rnz}
  In each $R$-descent class there is a unique idempotent. It is obtained by
  filling ribbon shape by numbers $1$ to $n$ in this order, going through the
  columns left to right and bottom to up. Then if $0$ is in the descent
  class, fill the first column with zeros.
\end{proposition}
\begin{proof}
  The existence and the uniqueness come from
  Corollary~\ref{actiongenerateurssurpiJ}. The way to fill in comes from
  Proposition~\ref{formeidempotentmatrice}.
\end{proof}
\begin{example}
  Consider the $R$-descent set $\lbrace 0, 1, 2, 5, 6, 7\rbrace$ in size
  $8$. We show below its associated ribbon shape and its idempotent
  $$\gyoung(;0,;0,;0;;,::;;,:::;)\qquad\qquad\gyoung(;0,;0,;0;4;6,::;5;8,:::;7)\,.$$
\end{example}
We now follow~\cite[Section 3.4]{DHST}, specializing it to the
combinatorics of $\Rnz$. Recall that the automorphism sub-monoid $\rAut(x)$
and $\lAut(x)$ are defined by
\begin{equation}
  \rAut(x)\eqdef\{u\in \tMonoid \suchthat xu=x\}
  \qandq
  \lAut(x) \eqdef \{u\in \tMonoid \suchthat ux=x\}\,.
\end{equation}
\begin{proposition}
Let $r\in \Rnz$.
\begin{equation}
\rAut(r) = \langle \pi_i \mid i\in D_R(r)\rangle
\qandq
\lAut(r) = \langle \pi_i\mid i\in D_L(r)\rangle.
\end{equation}
\end{proposition}
\begin{proof}
We do the proof for $\rAut$. The first inclusion
$\langle \pi_i \mid i\in D_R(r)\rangle \subseteq \rAut(r)$ is clear.

Let $u \in \rAut(r)$. So $ru = r$. Assume that $u\notin \langle \pi_i\mid i\in
D_R(r)\rangle $. Let $\pi_{i_1}\dots \pi_{i_m}$ a reduced expression of
$u$. Let $j$ be the smallest index such that $i_j \notin D_R(r)$. Then $ru =
r\pi_{i_j}\dots \pi_{i_m}$ by minimality. Since $i_j\notin D_R(r)$,
$r\pi_{i_j} <_{\JJ} r$ and by $\JJ$-triviality we get $ru<_{\JJ}
r$. This contradict the minimality.
\end{proof}
From~\cite[Proposition 3.16]{DHST}, we get the following corollary:
\begin{corollary}
Let $r\in \Rnz$
\begin{equation}
  \rfix(r) = \pi_{D_R(r)} \qandq \lfix(r) = \pi_{D_L(r)}.
\end{equation}
\end{corollary}

Then, applying Theorem~\ref{theorem.projective_modules}, we get:
\begin{theorem}\label{theo-Rn0-projective}
  The indecomposable projective $\Rnz$-modules are indexed by the $R$-descents
  sets and isomorphic to the quotient of the associated $R$-descent class by
  the finer $R$-descent classes.
\end{theorem}
\begin{figure}[ht]
\centering
\raisebox{5.5cm}{\includegraphics[scale=0.8]{figures/legende4.pdf}}
\hspace{-2cm}
\raisebox{25pt}{\includegraphics[scale=0.45]{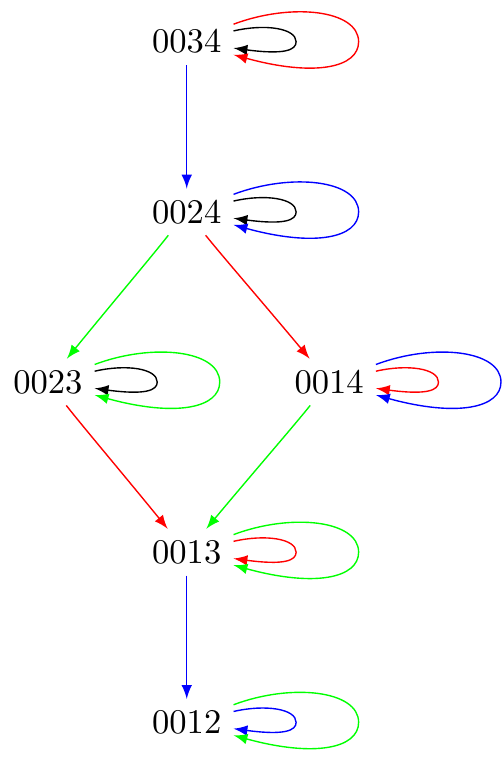}}
\quad
\raisebox{15pt}{\includegraphics[scale=0.45]{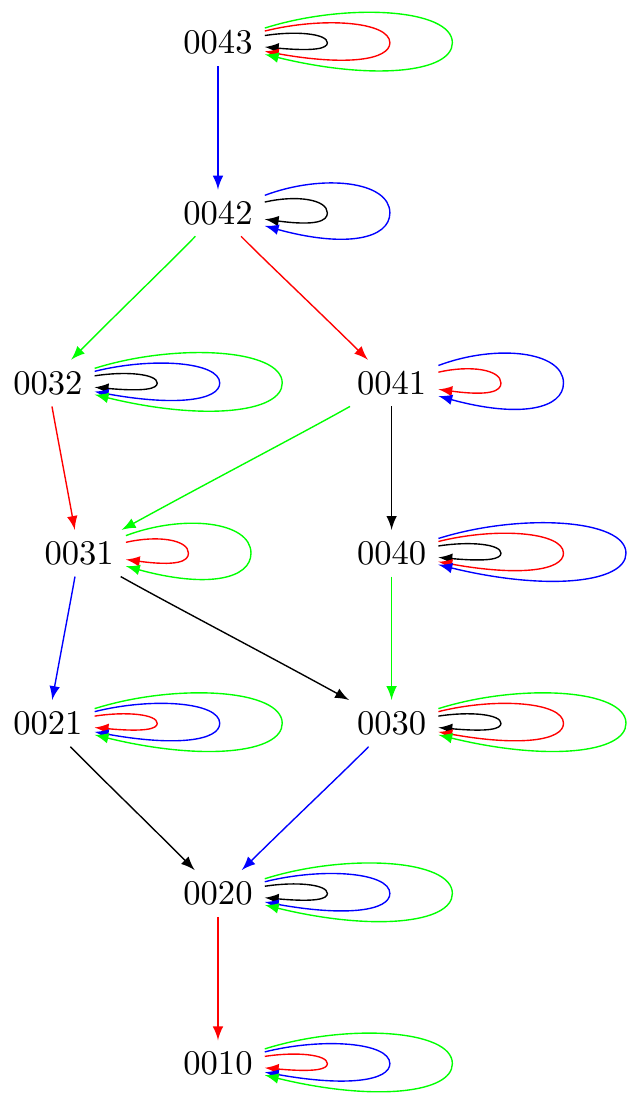}}
\quad
\includegraphics[scale=0.45]{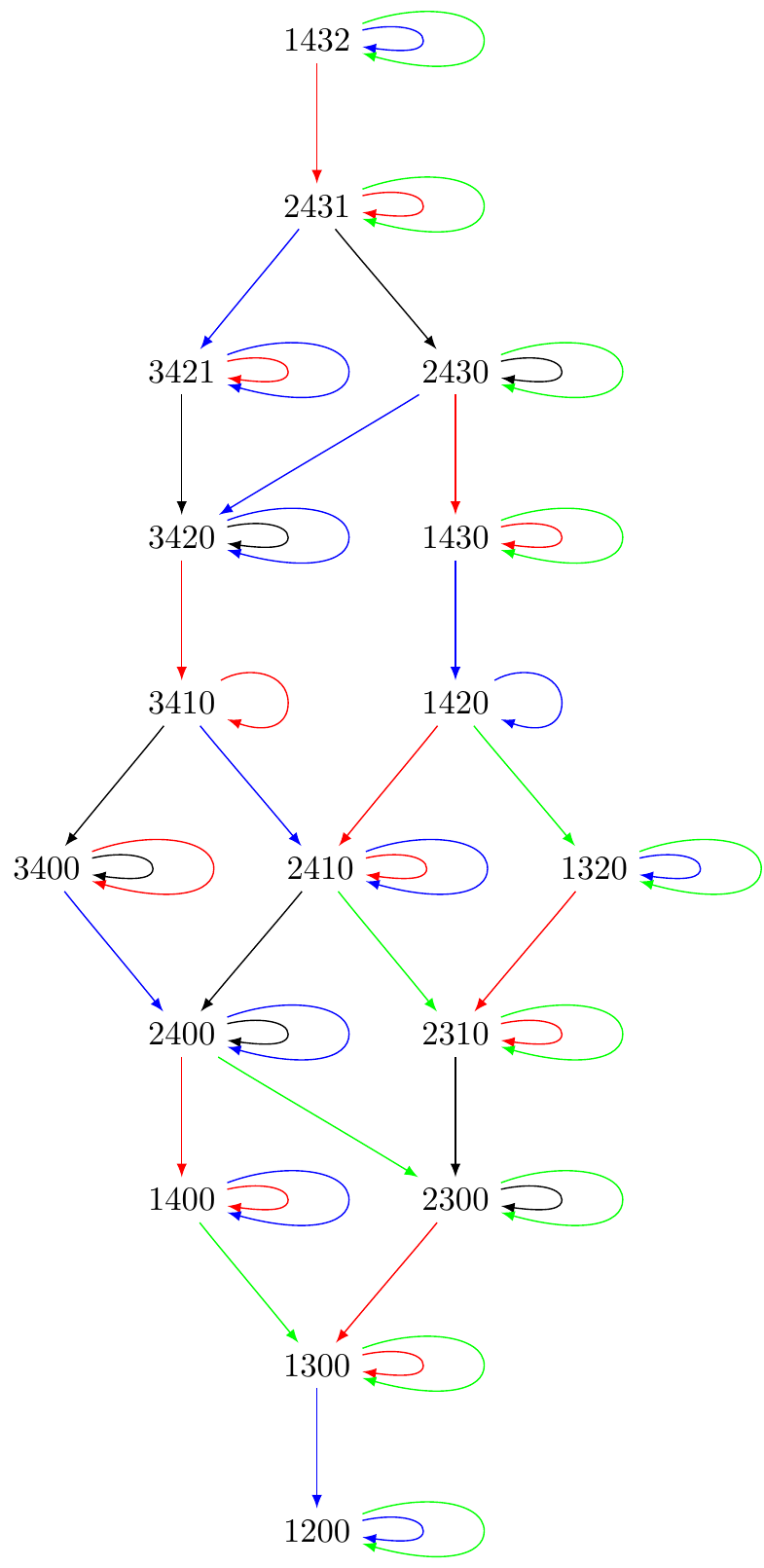}
\quad
\raisebox{5pt}{\includegraphics[scale=0.45]{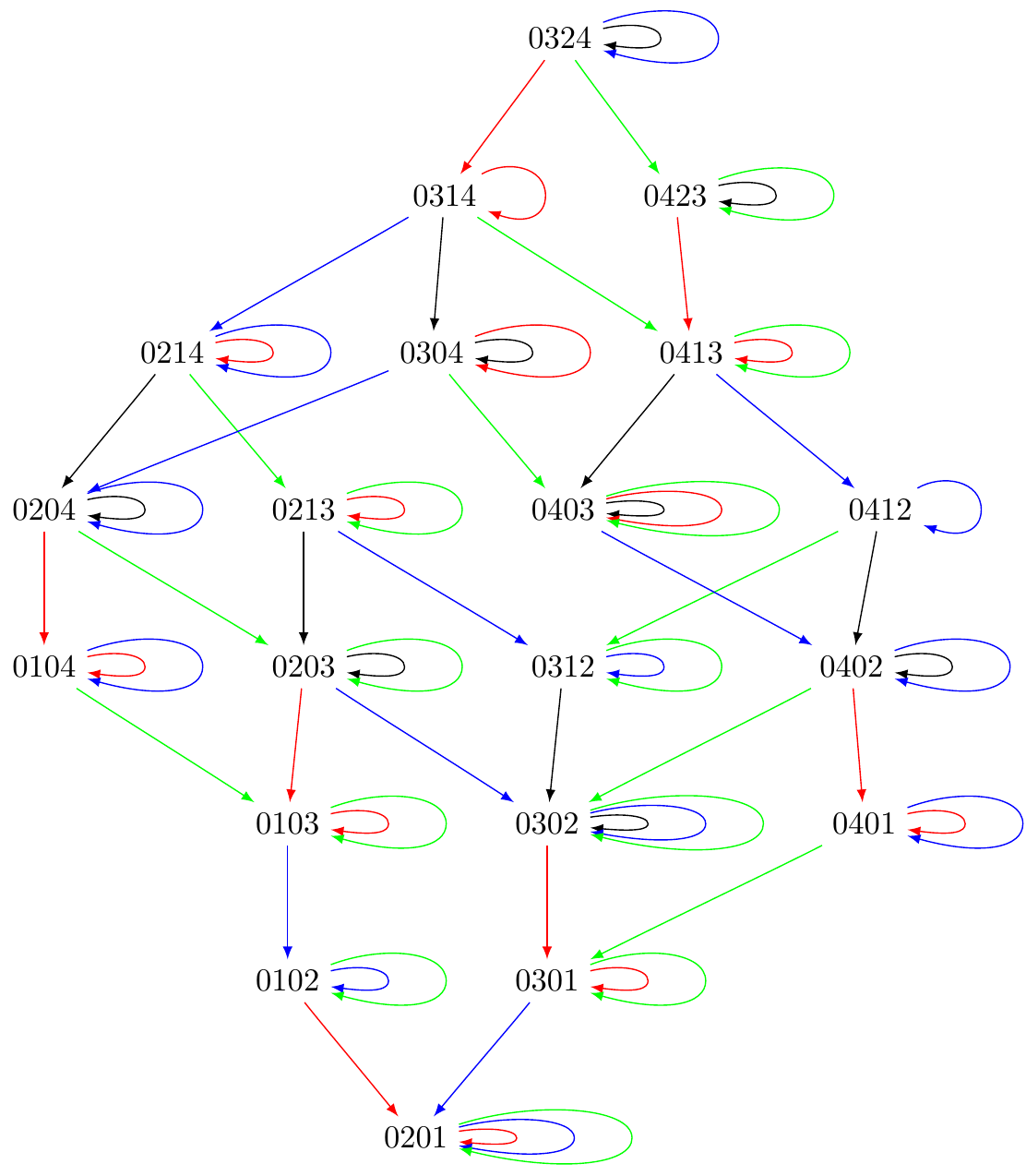}}
\caption{The $R$-descent classes $\lbrace 0,1\rbrace$, $\lbrace 0,1,3\rbrace$, $\lbrace  2,3\rbrace$ and $\lbrace 0,2\rbrace$. }
\end{figure}
Finally \cite[Theorem 3.20]{DHST} gives the coefficients of the Cartan matrix
of $\Rnz$ as the number of rooks with a given left and right descent set. We
give it in Annex~\ref{decFunctorTable}, for $n=1,2,3,4$.

\begin{remark}
  Contrary to the classical case~\cite{KrobThibon.1997} these quotients are
  not intervals of the $\RR$-order: the descent class depicted in
  Figure~\ref{fig_not_interval} has two bottom elements.
\end{remark}
\begin{figure}[!ht]
\centering
\raisebox{8cm}{\includegraphics[scale=0.9]{figures/legende4.pdf}}
\hspace{-2cm}
\includegraphics[scale=0.7]{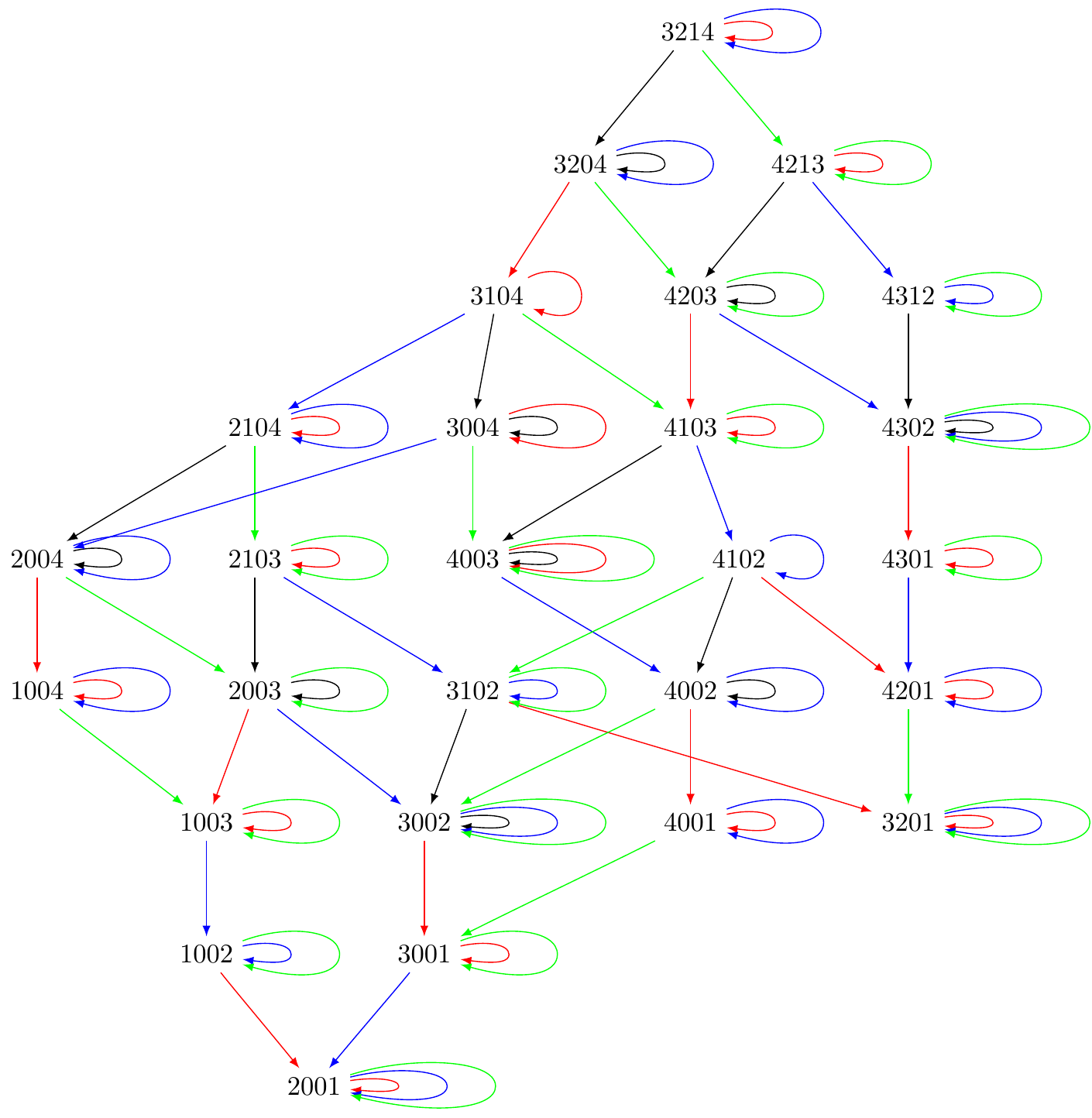}
\caption{An example of a $R$-descent class which is not an interval of the $\RR$-order.}\label{fig_not_interval}
\end{figure}

\subsection{Ext-Quivers}\label{sec-quiv_rnz}

The Ext-quiver of $H_n^0$ were first computed
in~\cite{DuchampHivertThibon.2002} in type $A$, and later in~\cite{Fayers.2005} in
the other types. Moreover, \cite{DHST} describes an algorithm to compute the
quiver of any $\JJ$-trivial monoid. This algorithm is implemented in
\texttt{sage-semigroups} from the second author, Franco~Saliola and
Nicolas~M.~Thiéry~\cite{sage-semigroup}. It turns out that the quiver of rook
monoids are not different from $0$-Hecke monoids:
\begin{theorem}\label{theo-Rn0-quiver}
  The kernel of the two following algebra morphisms
  \begin{equation}
    \C[H^0(B_{n})] \twoheadrightarrow \C[\Rnz]
    \qandq
    \C[\Rnz] \twoheadrightarrow \C[H^0(A_{n+1})]
  \end{equation}
  are included in the square radical of their respective domains. As a
  consequence, these three algebras share the same quiver.
\begin{proof}
  Recall that all of these algebras are monoid algebras of $\JJ$-trivial
  monoids. Thanks to~\cite[Corollary~3.8]{DHST}, their radical is generated by
  commutators. Therefore, the following non zero elements:
  $\pi_0\pi_1\pi_0-\pi_0\pi_1$ and $\pi_0\pi_1\pi_0-\pi_1\pi_0$ lie in the
  radical of each of these three algebras.  The first map has for kernel the
  ideal generated by the relation
  \begin{equation*}
    \pi_0\pi_1\pi_0-\pi_0\pi_1\pi_0\pi_1=
    (\pi_0\pi_1\pi_0-\pi_0\pi_1)(\pi_0\pi_1\pi_0-\pi_1\pi_0)\,.
  \end{equation*}
  which thus lies in the square radical.  Similarly the kernel of the second map is
  the ideal generated by
  \begin{equation*}
    \pi_1\pi_0\pi_1-\pi_1\pi_0\pi_1\pi_0 =
    (\pi_0\pi_1\pi_0-\pi_1\pi_0)(\pi_0\pi_1\pi_0-\pi_0\pi_1)\,.\qedhere
  \end{equation*}
\end{proof}
\end{theorem}
We refer the reader who want to see actual picture of the quiver
to~\cite{DuchampHivertThibon.2002}. Except for trivial cases, they are not of
known type so that the representation theory of $R^0_n$ starting from $n=3$ is
wild.

\subsection{Restriction functor to \texorpdfstring{$\Hnz$}{Hn0}}

We now further examine the links between representations of $\Hnz$ and
$\Rnz$. Indeed, since $\Hnz$ is a submonoid of $\Rnz$, it acts by
multiplication on $\Rnz$. We can see $\Rnz$ as an $\Hnz$-module. Moreover,
we can transport modules back and between $\Hnz$ and $\Rnz$ trough the
induction and restriction functors.

We first look at simple modules whose restriction rule is immediate:
\begin{proposition}\label{restsimpleRH}
  Let $J\subset \interv{0}{n-1}$, with associated simple $\Rnz$-module $S_J$. Then:
  \begin{equation}
    \Res_{\Hnz}^{\Rnz} S_J = S^H_{J\setminus \lbrace 0\rbrace}\,,
  \end{equation}
  where $S^H_I$ is the simple $\Hnz$-module generated by the parabolic $I \subseteq \llbracket 1, n-1\rrbracket$.
\end{proposition}

The rule of induction for simple $\Hnz$-modules to $\Rnz$-modules is otherwise
quite intricate and would be very technical. It would be very similar to what
we will do in section \ref{mod_simple_tour} for the induction of simple
modules of $\Rnz$ to another $R_{m}^0$, which is already very technical.

We now look at indecomposable $\Rnz$-projective modules.
\begin{proposition}\label{restindecRH}
  Let $I\subset \interv{1}{n-1}$ and $P^H_I$ the associated indecomposable
  $\Hnz$-projective module. Then:
  \begin{equation}
    \Ind_{\Hnz}^{\Rnz} P^H_I = P_I\oplus P_{I\cup\lbrace 0\rbrace}\,.
  \end{equation}
\end{proposition}
\begin{proof}
  This is a consequence of Proposition~\ref{restsimpleRH}, thanks to Frobenius
  reciprocity~(see \textit{e.g.}~\cite{CurtisReiner.1990}). Indeed, since the
  simple module $S^R_J$ is the quotient of the indecomposable projective
  $P^R_J$ by its radical, the multiplicity of $P^R_J$ in a projective module
  $P$ is equal to $\dim\Hom_R(P, S^R_J)$. By Frobenius reciprocity,
  \begin{equation}
    \operatorname{Mult}_{P^R_J}(\Ind_H^R P_I^H) =
    \dim\Hom_R(\Ind_H^R P_I^H, S_J^R) = \dim\Hom_H(P_I^H, \Res S_J^R)\,.
  \end{equation}
  Now, Proposition~\ref{restsimpleRH} says that this dimension is $1$ only if
  $I=J\setminus\{0\}$, otherwise it is $0$.
\end{proof}

The restriction of projective modules from $\Rnz$ to $\Hnz$ is much more
interesting. We will show that $\Rnz$-projective modules are still projective
when restricted to $\Hnz$, and give a precise combinatorial rule.

\begin{definition}\label{fI}
  Let $I\subset\{1,\dots,n\}$ of size $k$, and $\sigma =i_1\dots i_n \in
  \mathfrak{S}_n$. We define $\varphi_I(\sigma)$ to be the rook obtained by
  removing the first $k$ entries of $\sigma$ and inserting zeros in positions
  indexed by the elements of $I$.

  We also denote $\psi : R_n \rightarrow \mathfrak{S}_n$ the map which takes a
  rook, put all zeros at the beginning of the word and replace them by the
  missing letters in decreasing order. 
\end{definition}
\begin{example}
  For instance $\varphi_{\lbrace 1,3\rbrace}(14235) = 02035$ and $\psi(02410)
  = 53241$.
\end{example}

For the next results, we will consider $\Rnz$ to be a left $\Hnz$-module by
left multiplication. Thus the action is on values as in
Definition~\ref{def_left_action}.
\begin{theorem}\label{Rn_proj_Hn}
  $\C\Rnz$ is projective over $\Hnz$. As a consequence any projective
  $\Rnz$-module remains projective when restricted to $\Hnz$.
\end{theorem}
\begin{proof}
  The main remark is that according to Definition~\ref{def_left_action}, the
  left action of $\pi_i$ for $i>0$ on any rook does not change the zeros: they
  remain at the same positions and no one are added.

  For any $I\subset\interv{0}{n-1}$, let $C_{I}$ the set of rooks with zeros
  in the positions indexed by $I$. Since the action of $\Hnz$ does not move
  zeros, we have the following decomposition in $\Hnz$-modules:
  \begin{equation}
    \C\Rnz \simeq \bigoplus_{I\subset\interv{0}{n-1}} \C C_{I}\,.
  \end{equation}
  It is enough to prove that each summand $\C C_J$ are projective since
  direct sums of projective modules are projective.

  For such a summand where zeros are in positions $i\in I$, the linearization
  of the map $\psi$ of Definition~\ref{fI} is an injective $\Hnz$-module
  morphism. Its image is the set of permutations which start with $|I|-1$
  descents which is a well known projective $\Hnz$-module. Indeed, it is the
  $\Hnz$-module generated by the element
  $i, i-1, \dots, 2, 1, i+1, i+2, \dots , n$. This element is the zero of the
  parabolic submonoid generated by $\lbrace \pi_1,\dots \pi_{i-1}\rbrace$,
  hence idempotent. Consequently it generates a projective modules. This shows
  that $C_{I}$ is projective on $\Hnz$.
\end{proof}

We now describe explicitely the restriction functor.  Recall from
Equation~\ref{induct_proj_hn0} that the induction product of two
indecomposable projective $H_m^0$-Module (resp. $H_n^0$-Module) $P_I$ and
$P_J$ is given by $P_I\star P_J \eqdef \Ind_{m,n}(P_I\otimes P_J) \simeq P_{I
  \cdot J} \oplus P_{I \triangleright J}$.
\begin{definition}
  Let $I$ be an extended composition of $n$. A zero-filling of $I$ is a ribbon
  of shape $I$ with boxes either empty, either with $0$ inside according to
  the following rules:
  \begin{itemize}
  \item In the first column, either every box contains $0$ if $0\in\Des(I)$, or
    none otherwise.
  \item Outside of the first column, if a box contains $0$ then there is no box
    on its left, and all the boxes below in the same column also contain zeros.
  \end{itemize}
  To each of these fillings $f$ we associate a tuple $\Split(f)$ of ribbon as
  follows
  \begin{itemize}
  \item the first entry of the tuple is a column whose size is the total
    number of zeros in $f$
  \item the other entries of the tuple are the (down-right) connected components
    of $I$ where the boxes containing a $0$ in $f$ are removed.
  \end{itemize}
  To each splitting, it therefore makes sense to consider the $\star$-product
  $\prod_{r\in\Split(f)} P_r$.
\end{definition}

\begin{example}
  The following picture shows an extended composition followed by some of its
  $0$-fillings. There are $3*3*2$ of them.
  \[
  \rsbox{20pt}{0.5}{\gyoung(;0,;0;;,::;,::;;,:::;,:::;;,::::;;;)}
  \qquad
  \rsbox{20pt}{0.5}{\gyoung(;0,;0;;,::;,::;;,:::;,:::;;,::::;0;;)}
  \rsbox{20pt}{0.5}{\gyoung(;0,;0;;,::;,::;;,:::;,:::;0;,::::;;;)}
  \rsbox{20pt}{0.5}{\gyoung(;0,;0;;,::;,::;;,:::;,:::;0;,::::;0;;)}
  \rsbox{20pt}{0.5}{\gyoung(;0,;0;;,::;,::;0;,:::;0,:::;0;,::::;0;;)}
  \rsbox{20pt}{0.5}{\gyoung(;0,;0;;,::;0,::;0;,:::;,:::;;,::::;0;;)}
  \rsbox{20pt}{0.5}{\gyoung(;0,;0;;,::;0,::;0;,:::;,:::;0;,::::;0;;)}
  \rsbox{20pt}{0.5}{\gyoung(;0,;0;;,::;0,::;0;,:::;0,:::;0;,::::;0;;)}
  \]
  We now consider two particular $0$-fillings and show the ribbons appearing
  in the associated respective products (the colors are just to show what
  happens of each box):
  \[
  \rsbox{20pt}{0.5}{\gyoung(!\limeB;0,;0!\ylwB;;,::;!\limeB,::;0!\bluB;!\limeB,:::;0,:::;0!\redB;,::::;;;)}
  \mapsto
  \prod\ 
  \rsbox{25pt}{0.4}{\gyoung(!\limeB;,;,;,;,;,!\ylwB:;;,::;,!\bluB:::;,!\redB::::;,::::;;;)}
  \approx\ 
  \rsbox{25pt}{0.4}{\gyoung(!\limeB;,;,;,;,;!\ylwB;;,::;!\bluB;!\redB;,::::;;;)}
  \rsbox{25pt}{0.4}{\gyoung(!\limeB;,;,;,;,;!\ylwB;;,::;!\bluB;,!\redB:::;,:::;;;)}
  \rsbox{25pt}{0.4}{\gyoung(!\limeB;,;,;,;,;!\ylwB;;,::;,!\bluB::;!\redB;,:::;;;)}
  \rsbox{25pt}{0.4}{\gyoung(!\limeB;,;,;,;,;!\ylwB;;,::;,!\bluB::;,!\redB::;,::;;;)}
  \rsbox{25pt}{0.4}{\gyoung(!\limeB;,;,;,;,;,!\ylwB;;,:;!\bluB;!\redB;,:::;;;)}
  \rsbox{25pt}{0.4}{\gyoung(!\limeB;,;,;,;,;,!\ylwB;;,:;!\bluB;,!\redB::;,::;;;)}
  \rsbox{25pt}{0.4}{\gyoung(!\limeB;,;,;,;,;,!\ylwB;;,:;,!\bluB:;!\redB;,::;;;)}
  \rsbox{25pt}{0.4}{\gyoung(!\limeB;,;,;,;,;,!\ylwB;;,:;,!\bluB:;,!\redB:;,:;;;)}
  \]
  \[
  \rsbox{20pt}{0.5}{\gyoung(!\limeB;0,;0!\ylwB;;,::;!\limeB,::;0!\bluB;,:::;,:::;;,!\limeB::::;0!\redB;;)}
  \mapsto
  \prod\ 
  \rsbox{25pt}{0.4}{\gyoung(!\limeB;,;,;,;,!\ylwB:;;,::;,!\bluB:::;,:::;,:::;;,!\redB:::::;;)}
  \approx\ 
  \rsbox{25pt}{0.4}{\gyoung(!\limeB;,;,;,;!\ylwB;;,::;!\bluB;,:::;,:::;;!\redB;;)}
  \rsbox{25pt}{0.4}{\gyoung(!\limeB;,;,;,;!\ylwB;;,::;!\bluB;,:::;,:::;;,!\redB::::;;)}
  \rsbox{25pt}{0.4}{\gyoung(!\limeB;,;,;,;!\ylwB;;,::;,!\bluB::;,::;,::;;!\redB;;)}
  \rsbox{25pt}{0.4}{\gyoung(!\limeB;,;,;,;!\ylwB;;,::;,!\bluB::;,::;,::;;,!\redB:::;;)}
  \rsbox{25pt}{0.4}{\gyoung(!\limeB;,;,;,;,!\ylwB;;,:;!\bluB;,::;,::;;!\redB;;)}
  \rsbox{25pt}{0.4}{\gyoung(!\limeB;,;,;,;,!\ylwB;;,:;!\bluB;,::;,::;;,!\redB:::;;)}
  \rsbox{25pt}{0.4}{\gyoung(!\limeB;,;,;,;,!\ylwB;;,:;,!\bluB:;,:;,:;;!\redB;;)}
  \rsbox{25pt}{0.4}{\gyoung(!\limeB;,;,;,;,!\ylwB;;,:;,!\bluB:;,:;,:;;,!\redB::;;)}
  \]
\end{example}

\begin{theorem}\label{decompo_H_R}
  The indecomposable projective $\Rnz$-module $P_I^R$ associated to an
  extended composition $I$ splits as a $\Hnz$-module as
  \begin{equation}
    P_I^R\simeq \bigoplus_{f} \prod_{r\in\Split(f)} P_r\,,
  \end{equation}
  where the direct sum spans along all the zero-fillings of $I$, and the product
  is for the induction product $\star$.
\end{theorem}

Before giving a proof, we give a full example.
\begin{example}\label{excalc_decompo_H_R} We decompose restriction of the
  indecomposable projective $R_4^0$-module $P_{\{0,2\}}$ into indecomposable projective
  $H_4^0$-modules. The colors indicate the different products of
  zero-filling. Figure~\ref{fig:decoupage_proj} depict the action of the
  generators.
\begin{center}
\includegraphics[scale=0.8]{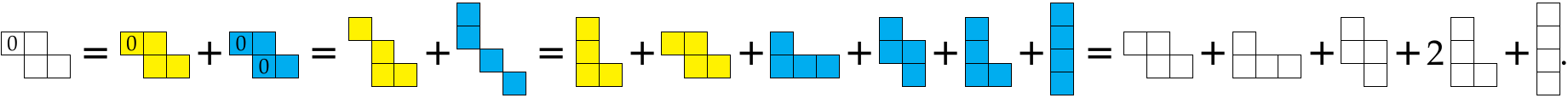}
\end{center}
\end{example}

\begin{proof}Let $P_I$ be an indecomposable projective $\Rnz$-module and look
  at it inside the regular representation. We proceed as in the proof of
  Theorem \ref{Rn_proj_Hn}: we cut $P_I$ according to the positions of zeros,
  which comes down to cutting along the zero-fillings. Indeed the conditions of
  zero-fillings give us only valid fillings, because they still have the good
  descent set. Moreover, we see all of them appearing in the descent class:
  for a given zero-filling $f$, we fill the diagram of $I$ column after
  column, left to right, down to up, by the entries starting from the number
  of zeros in the zero-filling plus $1$ to $n$. We get a rook of descent set
  $I$ with zero in the positions given by $f$. 

  Let $F$ be a zero-filling of shape $I$ with $i$ zeros in positions indexed
  by elements of $D\subset \interv{1}{n}$. Let $M_D \subset \Rnz$ be the
  associated $\Hnz$-projective indecomposable module. We consider the
  restriction $\psi_{F} \eqdef \psi_{\mid M_D}$. We need to describe the image of
  $\psi_F$. First they start with $i$ descents including zeros. We consider
  the connected components of $\interv{1}{n} \setminus I$: the letters at
  these positions are moved to the right by $\psi_F$, but keep their relative
  order. It is only between the connected components that we can have either a
  rise or a descent. Then we are getting a subset from a product associated to
  $F$. And we get them all: take one of them, and fill it with the same rule
  as before; one gets a permutation and then apply $\varphi_I$ defined in
  \ref{fI} to get an element with the good descent set and positions of zeros
  which will be sent by $\psi_F$ to an element of the product.
\end{proof}

\begin{figure}[ht]
\centering
\includegraphics[scale = 0.58]{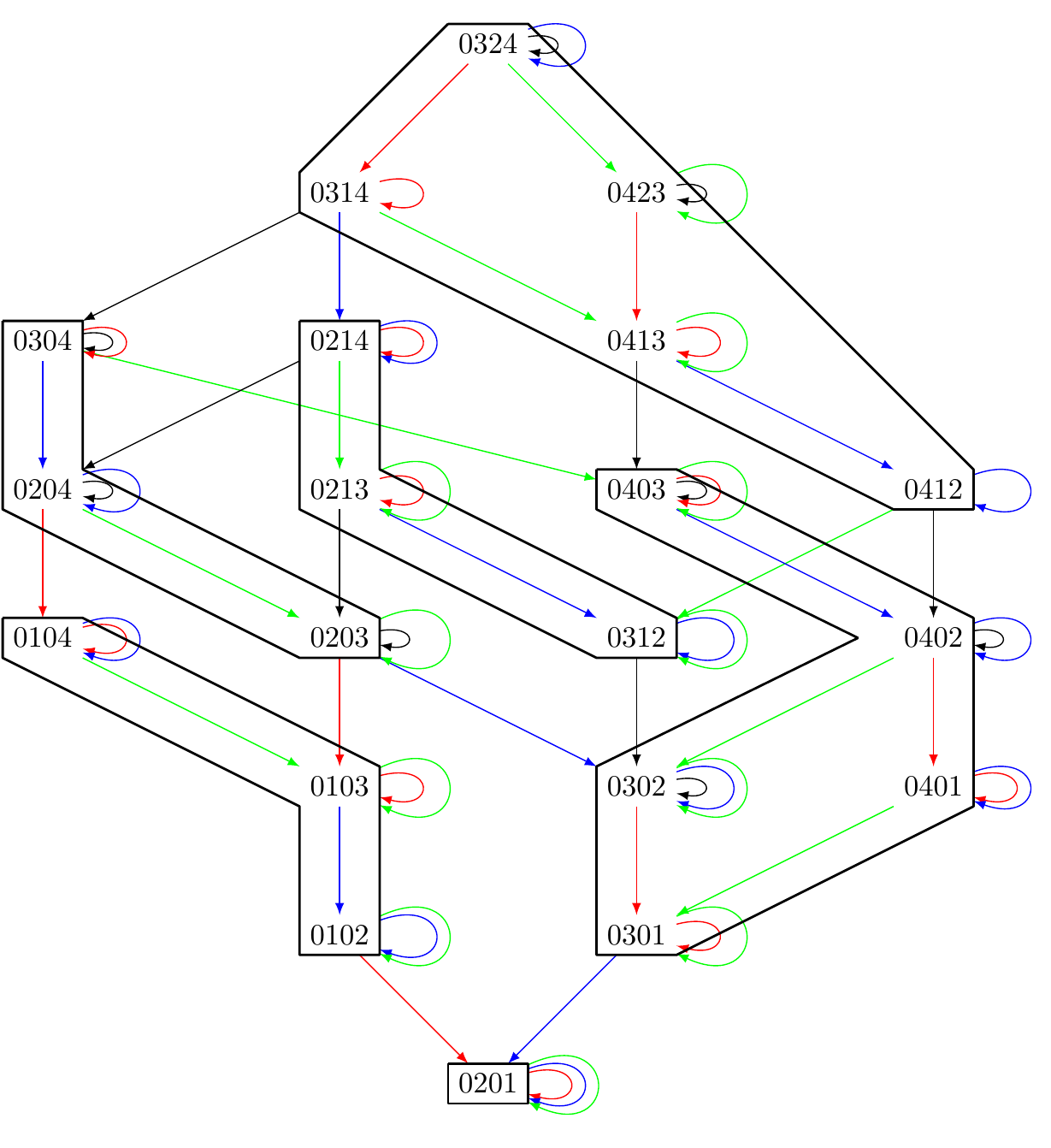}
\quad
\hspace{-1.5cm}
\raisebox{6cm}{\includegraphics[scale=0.8]{figures/legende4.pdf}}
\hspace{-1cm}
\includegraphics[scale = 0.58]{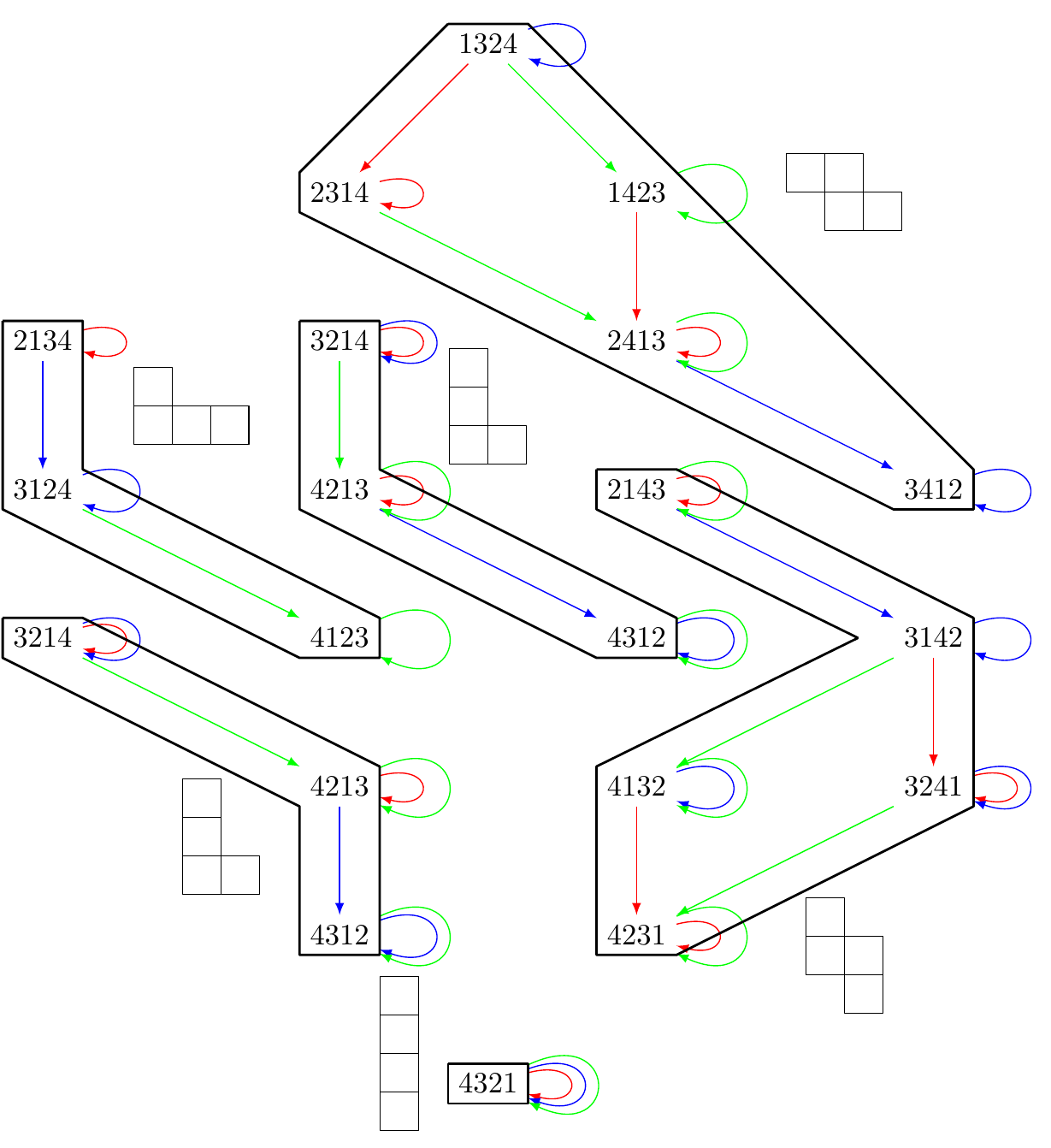}
\caption{The decomposition of a $R_4^0$-projective module associated to $\lbrace 0, 2\rbrace$ into $H_4^0$-projective modules.}
\label{fig:decoupage_proj}
\end{figure}
In Annex~\ref{decFunctorTable}, we give the decomposition
functor from $\Rnz$ to $\Hnz$ for $n=1,2,3,4$.

We can be a little more precise:
\begin{proposition}
  Let $P^R$ be an indecomposable projective module of $\Rnz$. Write
  $P^R=\bigoplus P_I^H$ its decomposition into indecomposable $\Hnz$-projective
  modules. Then the isomorphism of $\Hnz$-module $\tilde{\varphi}: \bigoplus
  P_I^H \rightarrow P^R$ is triangular:
  $\tilde{\varphi}(e)=\varphi_{I}(e)+\sum_{e'<e}e',$ with $\varphi_I$ defined
  in \ref{fI} and $I$ the zero-set linked to $P_I^H$.
\end{proposition}
\begin{proof}
  We consider a $\Rnz$ indecomposable projective $P^R$, pick a
  $D\subset\interv{1}{n}$ and denote as in the proof of
  Theorem~\ref{decompo_H_R} the $\Hnz$ submodule $M_D$ of rooks whose zeros
  are in positions indexed by the elements of $D$. The setwise map $\psi_{\mid
    M_D}$ extends linearly to an isomorphism to the projective but not
  necessarily indecomposable permutation module $\prod_{r\in\Split(f)} P_r$.
  Using~\cite[Theorem 3.11 and Corollary 3.19]{DHST}, we know that the basis
  change decomposing this module to its indecomposable component is
  uni-triangular. The statement follows by inverting this map.
\end{proof}
\begin{example}
  We know from Example \ref{excalc_decompo_H_R} that there is a
  module~$\raisebox{10pt}{\scalebox{0.6}{\gyoung(;,;,;;)}}$ inside the Figure
  \ref{fig:decoupage_proj}, coming from the zero-filling
  $\raisebox{8pt}{\scalebox{0.6}{\gyoung(0;,:;0;)}}$. This $\Hnz$-module is
  well-known to have the elements $3214$, $4213$ and $4312$. So ours must
  contains $\varphi_{\lbrace 0,2\rbrace}(3214) = 0104$, $\varphi_{\lbrace
    0,2\rbrace}(4213) = 0103$ and ${\varphi_{\lbrace 0,2\rbrace}(4312) =
    0102}$. See Figure \ref{fig:decoupage_proj}.
\end{example}

\subsection{Tower of monoids}\label{sec-tower}
The goal of this section is to investigate if the chain of submonoids
\begin{equation}
  R_1^0\subset R_2^0 \subset R_3^0 \subset \dots
  \subset R_{n-1}^0 \subset R_n^0 \subset R_{n+1}^0 \subset \cdots
\end{equation}

can be endowed with a structure of a tower of monoids~\cite{BergeronLi.2009}.

Recall that an associative tower of monoids is a sequence~$(M_i)_{i\in\N}$
where~$M_0=\{1\}$ together with a collection of monoid morphisms
$\rho_{n,m}:M_n\times M_m\rightarrow M_{n+m}$ such that the product
$a\cdot b \eqdef \rho_{n,m}(a,b)$ defined on the disjoint union
$\sqcup_{i\in\N} M_i$
is associative.

\begin{proposition}\label{plongement}
The maps
\begin{equation}
  \begin{array}{cccccc}
    \rho_{n,m} : & \Rnz & \times & R_m^0 & \longrightarrow & R_{n+m}^0 \\
    &\pi_0, \dots \pi_{n-1} &  & & \longmapsto & \pi_0, \dots \pi_{n-1}\\
    &P_i & &  & \longmapsto & P_i\\
    && & \pi_1, \dots \pi_{n-1} & \longmapsto & \pi_{n+1}, \dots \pi_{n+m-1} \\
    && & P_i & \longmapsto & P_{i+n}
  \end{array}
\end{equation}
defines an associative tower of monoids.
\end{proposition}
\begin{notation}
  If $a\in \Rnz$ and $b\in R_m^0$ we denote $a\cdot b \eqdef \rho_{n,m}(a,b)$.

  Furthermore, if $w$ is a word on nonnegative integers, $\decn{w}$ denotes the
  word $w$ where all nonzero entries have been increased by $n$.
\end{notation}
\begin{proof}
  We first show that $\rho_{n,m}$ are morphisms.  Let $a \in \Rnz$ et $b\in
  R_m^0$. Then, by relation of commutation and absorption we get $\rho(a,b) =
  \rho(a,1)\cdot\rho(1,b) = \rho(1,b)\cdot\rho(a,1)$.

  The proof of the associativity rely on the following lemma:
  \begin{lemma}\label{plongelt}
    Let $a\in \Rnz$ and $b\in R_m^0$. Then
    \begin{equation}
      a \cdot b = \begin{cases}
        \w{a}\decn{\w{b}} & \text{ if } 0\notin \w{b},\\
        0\dots 0\decn{\w{b}} & \text{ otherwise.}
      \end{cases}
    \end{equation}
  \end{lemma}
  \begin{proof}
    Indeed $\rho(a,b) = \rho(a,1)\rho(1,b)$. If $0\notin b$ then $\pi_0$
    does not appear in any reduced expression of $b$, thus the reduced
    expressions of $a$ and $b$ contain generators which do not act on
    $1_{n+m}$ on the same positions. Otherwise $P_{n+1}$ appear in
    $\rho(1,b)$, and since all elements of $\rho(1,a)$ commute with those
    of $\rho(1,b)$, $P_{n+1}$ absorbs all the $\rho(a)$.
  \end{proof}
  We now can compute explicitely the products $(a\cdot b)\cdot c$ and $a\cdot
  (b\cdot c)$, do the four cases whether $0\in B$ or not and $0\in C$ and
  check associativity.
\end{proof}

\begin{remark}
  The embedding $\rho$ is not injective since $\forall a,a' \in \Rnz$, and
  $b\in R_m^0$ with $0\in b$ : $a\cdot b = a'\cdot b$ by Lemma \ref{plongelt}. So we do not
  have a tower of monoid in the sense of~\cite{BergeronLi.2009}.
\end{remark}

\begin{remark}
  To map $\Rnz\times R_m^0$ to $R_{n+m}^0$, Remark~\ref{kill-other-0} prevents
  us to use the trivial map $(a, b)\mapsto \w{a}\decn{\w{b}}$: it is not a
  monoid morphism.
\end{remark}

\subsubsection{Restriction and induction of simple modules}\label{mod_simple_tour}

The goal of this section is to describe the restriction and induction rule of
the tower of the $0$-rook monoids. Recall that for $H_n^0$, this gives the
multiplication and comultiplication rule of the Hopf algebra of
quasi-symmetric functions in the fundamental basis~\cite{KrobThibon.1997}.

\paragraph{Restriction of simples modules}

\begin{theorem}\label{rest_simple_RR}
  Let $J \subset \interv{0}{n+m-1}$ a parabolic of $R_{n+m}^0$. Then:
  \begin{equation}
    \Res_{\Rnz\times R_m^0}^{R_{n+m}^0} S_J =
    \begin{cases}
      S_{J\cap \interv{0}{n-1} } \otimes
      S_{\overline{J\cap \interv{ n+1}{n+m-1} }} & \text{ if }J\cap \interv{0}{n} \neq \interv{0}{n},\\
      S_{\interv{0}{n-1} } \otimes  S_{\lbrace 0\rbrace \cup \overline{J\setminus \interv{0}{n} }}  & \text{ otherwise.}
    \end{cases}
  \end{equation}
  where $\overline{X}\eqdef\{x-n\mid x\in X\}$.
\end{theorem}
\begin{proof}
  We know that $S_J = \langle \varepsilon_J\rangle$, and that $\varepsilon_J
  \cdot \pi_i = \varepsilon_J$ if $i\in J$, and $0$ otherwise.
  The action of $\Rnz\otimes 1_m$ on $S_J$ gives us $S_{J\cap \interv{0}{n-1}
  }$. The generators $1_n\otimes\pi_1,\dots ,1_n\otimes\pi_{m-1}$ of
  $1_n\otimes R_m^0$ act as $\pi_{n+1},\dots, \pi_{n+m-1}$. It remains only to
  see how $1_n\otimes \pi_0 = P_{n+1}$ acts on $S_J$.
  By Lemma \ref{Pn} we have that
  $P_{n+1}=\pi_0\pi_1\pi_0\pi_2\pi_1\pi_0\dots\pi_n\dots\pi_2\pi_1\pi_0$. So
  if there is $0\leq i\leq n$ with $i\notin J$, $\varepsilon_J\cdot \pi_i = 0$
  thus $\varepsilon_J\cdot P_{n+1} = 0$. Otherwise, for all
  $i\in\interv{0}{n}$, $\varepsilon_J\cdot \pi_i = \varepsilon_J$ and so
  $\varepsilon_J\cdot P_{n+1} = \varepsilon_J$.
\end{proof}

\paragraph{Induction of simple modules}

We can compute the induction of simple module thanks to Virmaux~\cite[Theorem
4.3]{Virmaux.2014}, which we reformulate in our context here. The comparisons
are done with the $\RR$-order in $\Rnz$, which we described in Theorem
\ref{ordreR0}.
\begin{theorem}[{\cite[Theorem 4.3]{Virmaux.2014}}]\label{Aladin}
  If $e \in E(\Rnz)$ and $f\in E(R_m^0)$, then
  \begin{equation}\label{aladin_qef}
    \Ind_{\Rnz \times R_m^0}^{R_{n+m}^0} S_e \otimes S_f  =
    \faktor{(e\cdot f) R_{n+m}^0 }%
    {\left[(R_{<e}\cdot f) + (e\cdot R_{<f})\right] R_{n+m}^0 },
  \end{equation}
  where $R_{<e}$ is the set of elements of $\Rnz$ strictly smaller than $e$,
  and $R_{<f}$ those of $R_m^0$ strictly smaller than $f$.
\end{theorem}

\begin{notation}
  In Equation~\ref{aladin_qef}, we will denote by $Q(e,f)$ the right hand side
  of the equality.  It is a $R_{n+m}^0$-module.
  It is also a quotient which is compatible with the canonical basis. By abuse
  of language, we will say that an element $r\in R_{n+m}^0$ remains in
  $Q(e,f)$ and write ${r\in Q(e,f)}$ if $r$ is not mapped to zero in the
  quotient.
\end{notation}

Our first goal is to rephrase Theorem~\ref{Aladin} in a more combinatorial way.
\begin{notation}
  Until now, we used the notation $\pi_I$ to design the idempotent of the
  parabolic submonoid associated to $I$ in $\Rnz$. In order to avoid
  confusion, we will now denote it by $\pi_{I,n}$. Note that as long as $n,
  m\geq \max I+1$, then $\pi_{I,n}$ and $\pi_{I,m}$ have the same
  reduced expressions and thus the same action of the first
  $\min(n,m)$-letters on the identity of size $\max(n,m)$.
\end{notation}
In the sequel of this section, we fix $I\subseteq\interv{0}{n-1}$ and
$J\subseteq\interv{0}{m-1}$. They encode the data of two simple modules of
$\Rnz$ and $R_m^0$ respectively, or equivalently of two idempotents. We denote
$e\eqdef\pi_{I, n}$ and $f\eqdef\pi_{J, m}$ these two idempotents.

Before giving the induction of the simple modules, we go for a serie of lemmas.
\begin{lemma}\label{inclplong}
  The image of $(e,f) \in \Rnz\times R_m^0$ in $R_{n+m}^0$ is the element of
  $R_{n+m}^0$ associated to $\w{e}\decn{\w{f}}$ if $0\notin J$ and to $0\dots
  0\decn{\w{f}}$ otherwise. In particular we have the following cases:
  \begin{itemize}
  \item If $J = \emptyset$ then $e\cdot f = \w{e}\decn{12\dots m} = \pi_{I,
      n+m}$. 
  \item If $I = \emptyset$ and $0\notin J$ then $e\cdot f = 1\dots n
    \decn{\w{f}} = \pi_{\decn{J},
      n+m}$. 
  \item If $I = \interv{0}{n-1}$ and $0\in J$ then $e\cdot f = 0 \dots 0
    \decn{\w{f}} = \pi_{\left(\interv{ 0}{n } \cup \decn{J\setminus\lbrace
        0\rbrace}\right), n+m}$.
  \end{itemize}
\end{lemma}
\begin{proof}
It is straightforward application of Lemma~\ref{plongelt}.
\end{proof}

\begin{remark}
  Note that because of the form of idempotents, $0\notin I \Leftrightarrow
  0\notin \w{e}$.
\end{remark}

\begin{lemma}\label{lemcasnulinduct}
  Assume that $0 \in J$ and $I \neq \interv{0}{n-1}$. Then $Q(e,f) = 0$.
\end{lemma}
\begin{proof}
  Since $0 \in J$ then $e\cdot f = 0\dots 0 \w{f}$ according to
  Lemma~\ref{inclplong}. On the other hand, let $j \in
  \interv{0}{n-1}\setminus I$. Then in $Q(e,f)$ we are doing a quotient by
  $(e\cdot \pi_j)\cdot 1_m$ which is above $e\cdot f$ by Theorem
  \ref{ordreR0}. Hence $Q(e,f) = 0$.
\end{proof}

We are now considering cases where $0\notin J$, writing $\w{f}=f_0\dots f_m$.
\begin{lemma}\label{leminduction}
  Assume $0\notin J$. Let $r$ be an element of $R_{n+m}^0$ which does not
  vanish in the quotient $Q(e,f)$. Let $a$ and $b$ be two letters of $\w{r}$,
  not both zero.  If $a$ and $b$ appear both in $\w{e}$ (resp.\ $a - n$ and
  $b-n$ appear both in $\w{f}$) then they appear in $\w{r}$ in the same order
  as in $\w{e}$ (resp.\ $\w{f}$).  Furthermore, all the nonzero letters of
  $\w{e}$ appear in $\w{r}$.  Finally, if $f_i+n$ is not in $\w{r}$ then
  $f_j+n$ is not in $\w{r}$, for all $j<i$.
\end{lemma}
\begin{example}
  If $\w{e} = \w{023}$ and $\w{f} = \w{213}$ then neither $\w{042356}$, or
  $\w{005463}$, or $\w{025306}$ remain in $Q(e,f)$, respectively because of
  the first, second and third rule.
\end{example}
\begin{proof}
  For the first point, it is sufficient to do the proof when the two letters
  are consecutive in~$\w{e}$. Let $\w{r}\in Q(e,f)$. So
  $r\leq \w{e}\decn{\w{f}}$. Assume $\w{e}=\un{L}ab\un{R}$ with $a$ and $b$ non
  both zero, and both present in $\w{r}$.

  Suppose first that $a>b$, so that $a\neq 0$ and $b \neq 0$ since $0\notin
  J$. Since $r<\w{e}\decn{\w{f}}$ we deduce that $a$ is before $b$ in
  $\w{r}$.

  Otherwise, $a<b$. Let $i\eqdef \ell(\un{L})$ be the position of $a$ in
  $\w{e}$. So $i\notin I$. Then $e\cdot \pi_i < e$. Also $\Inv(e\cdot \pi_i) =
  \Inv(e)\cup \lbrace (b,a) \rbrace$. Thus, $\Inv((\w{e}\cdot
  \pi_i)\decn{\w{f}}) = \Inv(\w{e}\decn{\w{f}})\cup \lbrace (b,a) \rbrace$
  while $(b,a) \notin \Inv(\w{e}\decn{\w{f}})$. Since $r < \w{e}\decn{\w{f}}$
  we get $\lbrace (r_i, r_j) \in \Inv(\w{e}\decn{\w{f}}) \, \vert \,
  r_i \in \w{r}\rbrace \subseteq \Inv(r)$. Assume that $b$ is left to
  $a$ in $\w{r}$. In this case we have $\lbrace (r_i, r_j) \in
  \Inv(\w{e}\decn{\w{f}})\cup\lbrace(b,a)\rbrace \, \vert \, r_i \in
  \w{r}\rbrace \subseteq \Inv(r)$, so $r<(\w{e}\cdot
  \pi_i)\decn{\w{f}}$, the latter being an element by which we quotient in
  $Q(e,f)$. It is a contradiction.
  \medskip

  The proof is the same when $a$ and $b$ both come from $\w{f}$ once
  decreased. The only change are that both letter are nonzero, and that we
  have to decrease by $n$.
  \medskip

  Let us prove the second point by contradiction, assuming that a nonzero
  letter $b$ in $\w{e}$ is not in $\w{r}$. We first show that the first
  nonzero letter of $\w{e}$, say $a$, is not in $\w{r}$ either. By
  contradiction, assume that $a\in \w{r}$. If $a>b$ then $e$ has descent
  $(a, b)$. So $r$ must also have it since $r<\w{e}\decn{\w{f}}$ and
  $a \in \w{r}$, but it is not the case since $b\notin \w{r}$, which is a
  contradiction. Otherwise $a<b$. Since $a\in \w{r}$ and $b\notin \w{r}$, and
  that the generator $\pi_0$ can only delete the first letter, ${r}$ is in the
  $\RR$-order between $\w{e}\decn{\w{f}}$ and a rook ${r}'$ in which $a$ is
  there and $b$ is in first position. Because of the first point, this element
  ${r}'$ has been sent to $0$ in the quotient, and thus ${r}$ which is below
  as well. So ${r}=0$, again this is a contradiction.

  The same argument also apply to the third case, with some minor
  adaptation.
  \medskip

  Thus if there is a nonzero letter of $\w{e}$ lacking in $\w{r}$, the first
  one at least is lacking. We now look at $\w{e}$. If $0\notin I$, $\w{e}$
  begins with $a$. Then $\w{q} \eqdef (\w{e}\cdot \pi_0)\decn{\w{f}}$ is an
  element by which we quotient. We have ${r}<\w{e}\decn{\w{f}}$ and $a \notin
  \w{r}$ so ${r}< q$, thus ${r} = 0$, and we get a contradiction.

  Otherwise $0\in I$ so $\w{e} = 0\dots 0a\dots$. We denote by $i$ the
  position of the last $0$ and $q \eqdef (\w{e}\cdot \pi_i)\decn{\w{f}}$ is an
  element by which we quotient. Since $a\notin \w{r}$, ${r}$ is in the
  $\RR$-order between $\w{e}\decn{\w{f}}$ and a rook ${r}'$ in which $a$ is
  there in first position. In particular in $\w{r}$, we have a $0$ right to
  $a$. So ${r}<{r}'<\w{e}\decn{\w{f}}$ and $(a,0)\in \Inv({r}')$, so ${r}'<q$
  and thus ${r} = 0$, for a final contradiction.
\end{proof}

\begin{remark}\label{rem_forme_idem}
  Let $K \subset \interv{1}{n-1}$ and $g\in \Rnz$ the associated idempotent
  (hence $0 \notin g$). We write $\w{g} = g_1g_2\dots g_n$. Because of
  Proposition \ref{forme_idempotent_Rnz} we have that if $g_1 = \ell$ then
  $g_2 = \ell-1$, $g_3 = \ell-2$, \dots , $g_{\ell-1} = 2$ and $g_\ell =
  1$. Furthermore $\ell \notin K$ (since $g_{\ell+1} > g_\ell$) and $\ell =
  \min\left( \interv{1}{n-1} \setminus I\right)$.
\end{remark}

We are now in position to state the formula giving the induction of simple
modules. Recall that $\shuffle$ denote the so-called shuffle product
introduced in Definition~\ref{def:shuffle}. We also denote $0^i$ the word
$00\dots0$ with $i$ letters $0$.

\begin{theorem}\label{ind_simple_tour}
  \newcommand\Qef{\Ind_{\Rnz \times R_m^0}^{R_{n+m}^0} S_I \otimes S_J} For
  $n, m\in \N$, we fix $I\subseteq\interv{0}{n-1}$ and
  $J\subseteq\interv{0}{m-1}$. Denoting $e\eqdef\pi_{I, n}$ and $f\eqdef\pi_{J, m}$,
  the induction of simple modules $S_I=S_e$ and $S_J=S_f$ is given by
  \begin{enumerate}
  \item If $0\in J$ and $I\neq \interv{0}{n-1}$ then $\Qef=0$.
  \item If $0\in J$ and $I = \interv{0}{n-1}$ then
    $\Qef = \left\langle \w{e} \decn{\w{f}} \right\rangle
    \simeq S_{\interv{0}{n} \cup \decn{J\setminus\lbrace 0\rbrace}}$.
  \item If $0\notin J$ and $I =  \interv{0}{n-1}$ then
    $\Qef = \ev{0^n \shuffle \decn{\w{f}}}$.
  \item If $0\notin J$ and $0\in I$, $I\neq \interv{0}{n-1}$, 
    let $\ell:=f_1$ be the first letter of $\w{f} = f_1\dots f_m$. Then:
    \begin{equation}
      \Qef = \left\langle\
        0^i\w{e}\shuffle \decn{f_{i+1}\dots f_m}\mid i=0,\dots, \ell
      \ \right\rangle\,.
    \end{equation}
  \item If $0\notin J$ and $0\notin I$, let $\ell:=f_1$ be the first letter of $\w{f}
    = f_1\dots f_m$. Then
    \begin{equation}
      \Qef = \left\langle\
        0^i\shuffle\w{e}\shuffle \decn{f_{i+1}\dots f_m}\mid i=0,\dots, \ell
      \ \right\rangle\,.
    \end{equation}
  \end{enumerate}
\end{theorem}
\begin{proof}
  \begin{enumerate}
  \item This case follows directly from Lemma~\ref{lemcasnulinduct}.
  \item Let $K \eqdef \interv{0}{n} \cup \decn{J\setminus\lbrace 0\rbrace}$. Then
    by Lemma \ref{inclplong},$e\cdot f = \pi_{{K,n+m}}.$ Since
    $I=\interv{0}{n-1}$ then
    \begin{equation}
      Q(e,f) =
      \faktor{\pi_{{K, n+m}} R_{n+m}^0 }%
      {\left[(0\dots 0\cdot R_{<f})\right] R_{n+m}^0}\,.
    \end{equation}
    On the other hand, let $g\eqdef\pi_{K, n+m}$ be the idempotent associated to
    $K$ in $R^0_{n+m}$. By Theorem~\ref{Aladin},
    \begin{equation}
      S_g = \Ind_{1\times R_{n+m}^0}^{R_{n+m}^0} 1\otimes S_g =
      \faktor{g R_{n+m}^0 }{\left[R_{<g}\right] R_{n+m}^0}\,.
    \end{equation}
    But since $I=\interv{0}{n-1}$ on has $R_{<g} = 0\dots 0 \cdot
    R_{<f}$, so that $Q(e,f) \simeq S_g$.

  \item Since $I= \interv{0}{n-1}$ then $\w{e} = 0\dots 0$ and $e\cdot f =
    0\dots 0\un{\w{f}}$. Let $\w{r}\in 0\dots 0 \shuffle
    \decn{\w{f}}$. Clearly ${r} < e\cdot f$. We know that $\w{r}$ has the same
    number of zeros than $\w{e\cdot f}$ and also that its inversions are
    those of $f$ increased by $n$. We
    deduce that ${r}$ is not below $\w{e}\decn{(\w{f}\cdot \pi_j)}$ in the
    $\RR$-order for $j \in \interv{0}{m-1} \setminus J$. Thus ${r}\in Q(e,f)$.

    Conversely let ${r}\in Q(e,f)$. Since $0\notin J$ then ${r}\not<
    \w{e}\decn{(\w{f}\cdot \pi_0)}$. So the first letter of $\w{f}$ increased
    by $n$ is in $\w{r}$. By Lemma \ref{leminduction} all the letters of
    $\w{f}$ increased by $n$ are in $\w{r}$. Again by Lemma
    \ref{leminduction} they are in the same order, and so $\w{r} \in 0\dots 0
    \shuffle \decn{\w{f}}$.

  \item Denote $S_{ef} \eqdef \w{e}\shuffle \decn{\w{f}} + 0\w{e}\shuffle
    \decn{f_2\dots f_m} +\dots + 0\dots 0\w{e}\shuffle \decn{f_{\ell+1}\dots
      f_m}$ and let $\w{r}\in S_{ef}$. The same argument than the third point shows
    that ${r}\in Q(e,f)$.

    Conversely, let ${r}\in Q(e,f)$. Since $0\in I$ (or equivalently, $0\in
    \w{e}$) Lemma~\ref{leminduction} tells us that the eventual new zeros of
    $\w{r}$ are before the nonzero letters of $\w{e}$. By the same lemma, the letters
    of $\w{f}$ disappear in the same order than in $\w{f}$. So that we have proven:
    \begin{equation}
      \w{r}\in T_{ef} \eqdef
      \w{e}\shuffle \decn{\w{f}} +
      0\w{e}\shuffle \decn{f_2\dots f_m} +\dots +
      0\dots 0\w{e}\shuffle \decn{f_m} + 0\dots 0\w{e}.
    \end{equation}
    We recall that $\ell = f_1$. We have to show
    that elements of $T_{ef}\setminus S_{ef}$ are not in $Q(e,f)$. A first
    immediate remark is that all these elements are below $t = 0\dots
    0\w{e}\decn{f_{\ell+2}\dots f_m}$.  But $t < \w{e}f_1\dots
    f_{\ell-1}f_{\ell+1}f_\ell f_{\ell+2}\dots f_m = (\w{e}\cdot(\w{f}\cdot
    \pi_\ell))$. Thus, since $\ell \notin J$ (by Remark \ref{rem_forme_idem}),
    $t = 0$ in $Q(e,f)$, and so all $T_{ef}\setminus S_{ef}$ also, hence the result.

  \item Denote $S_{ef} \eqdef \w{e}\shuffle \decn{\w{f}} + 0\shuffle\w{e}\shuffle
    \decn{f_2\dots f_m} + \dots + 0\dots 0 \shuffle\w{e}\shuffle
    \decn{f_{\ell+1}\dots f_m}.$ Let $\w{r}\in S_{ef}$. The argument of the
    third point proves that ${r}\in Q(e,f)$.

    Conversely, for ${r}\in Q(e,f)$, the argument of the fourth point shows
    that $\w{r}\in S_{ef}$.\qedhere
\end{enumerate}
\end{proof}

Recall that the corresponding rule for $H_n^0$ is the multiplication of the
fundamental basis $(F_I)$ of quasi-symmetric
function~\cite{KrobThibon.1997}. This rule can be computed as follows
\cite{Gessel.1984,DuchampHivertThibon.2002}. Let $I$ and $J$ be two
compositions. Choose any permutation $\sigma\in\SG{n}$ whose descent
composition is $\cset(\sigma)=I$, for example $\pi_I$ whose corresponding
$H_n^0$ element is idempotent, and $\mu$ such that $\cset(\mu)=J$. Then
\begin{equation}
  F_I\,F_j = \sum_{\nu\in \sigma\shuffle \decn{\mu}} F_{\cset(\nu)}\, .
\end{equation}
As explained by Virmaux~\cite{Virmaux.2014} this is a direct consequence of
Theorem~\ref{Aladin}.

To get the analogue of the product of quasi-symmetric functions, one has to use
the Theorem~\ref{ind_simple_tour} and then get the projection of the induced
module in the Grothendieck ring. This amounts to compute the $R$-descent of
every rook vector appearing in the sum $Q(e,f)$ according to Jordan-Hölder's
theorem.

\begin{example}
  If $n = 2$, $m = 3$, $I = \lbrace 0, 1\rbrace$ and $J = \lbrace 1
  \rbrace$. Then $\w{e} = 00$ and $\w{f} = 213$.
  Theorem~\ref{ind_simple_tour} says that
  {\small$$Q(e,f) = \ev{00\shuffle 435} =
  \ev{\w{00435, 04035, 04305, 04350, 40035, 40305, 40350, 43005, 43050,
      43500}}.$$}%
  This gives the following $R$-descent classes:
  \begin{center}
    \footnotesize
    \begin{tabular}{r|cccccccccc}
      Element & $\w{00435}$&$\w{04035}$&$\w{04305}$&$\w{04350}$&$\w{40035}$&$\w{40305}$&$\w{40350}$&$\w{43005}$&$\w{43050 }$&$\w{43500}$\\
      Descents & 0,1,3 & 0,2 & 0,2,3 & 0,2,4 & 1,2 & 1,3 & 1,4 & 1,2,3 & 1,2,4 & 1,3,4
    \end{tabular}
  \end{center}
  \raisebox{10pt}{Finally:}
  $\begin{aligned}\Ind \Si{0,1}^2 \times \Si{1}^3 & = \Si{0,1,3}^5 + \Si{0,2}^5 + \Si{0,2,3}^5 + \Si{0,2,4}^5 + \Si{1,2}^5 + \Si{1,3}^5 \\
& \quad + \Si{1,4}^5 + \Si{1,2,3}^5 + \Si{1,2,4}^5 + \Si{1,3, 4}^5 \end{aligned}$
\end{example}

\begin{example}\label{ex_ind_2}
  If $n = 3$, $m = 2$, $I = \lbrace 0, 1\rbrace$ and $J = \lbrace 1
  \rbrace$. Then $\w{e} = 003$ and $\w{f} = 21$. Theorem~\ref{ind_simple_tour}
  says that
\[\begin{aligned}
  Q(e,f) & = \ev{003\shuffle 21 + 0003\shuffle 1 + 00003}\\
  &  = \langle\lbrace \w{00321}, \w{00231, 00213, 02031, 02013, 02103, 20031, 20013, 20103, 21003}\rbrace \\
  & \qquad \cup \lbrace \w{00031, 00013, 00103, 01003, 10003}\rbrace \cup \lbrace \w{00003}\rbrace\rangle
\end{aligned}
\]
Then:
\begin{center}
  \footnotesize
  \begin{tabular}{r|cccccccc}
    Element & $\w{00321}$ & $\w{00231}$ & $\w{00213}$ & $\w{02031}$ & $\w{02013}$ & $\w{02103}$ & $\w{20031}$ & $\w{20013}$ \\
    Descents & 0,1,3,4 & 0,1,4 & 0,1,3 & 0,2,4 & 0,2 & 0,2,3 & 1,2,4 & 1,2\\
    \hline
    Element & $\w{20103}$ & $\w{21003}$ & $\w{00031}$ & $\w{00013}$ & $\w{00103}$ & $\w{01003}$ & $\w{10003}$ & $\w{00003}$\\
    Descents & 1,3 & 1,2,3 & 0,1,2,4 & 0,1,2 & 0,1,3 & 0,2,3 & 1,2,3 & 0,1,2,3
  \end{tabular}
\end{center}
\begin{multline*}
  \Ind \Si{0,1}^3\times \Si{1}^2 = \Si{0,1,3,4}^5 + \Si{0,1,4}^5 + 2\Si{0,1,3}^5 + \Si{0,2,4}^5+\Si{0,2}^5 + 2\Si{0,2,3}^5 + \Si{1,2,4}^5 \\
  \quad + \Si{1,2}^5 + \Si{1,3}^5 + \Si{1,2}^5+
  \Si{0,1,2,4}^5+\Si{0,1,2}^5+2\Si{1,2,3}^5 + \Si{0,1,2,3}^5
\end{multline*}
\end{example}
This defines the left (resp. right) dual branching graph, where the arrows
$I \mapsto J$ are labelled by the multiplicity of $S_J$ in the induction of $S_I$
along the morphism $\rho_{1,n}$ (resp. $\rho_{n,1}$). The beginning of those
two graphs are illustrated in Figures~\ref{branching_gauche} and
\ref{branching_droit}.
\begin{figure}[ht]
  $$
  \vcenter{\hbox{\includegraphics[scale=0.6]{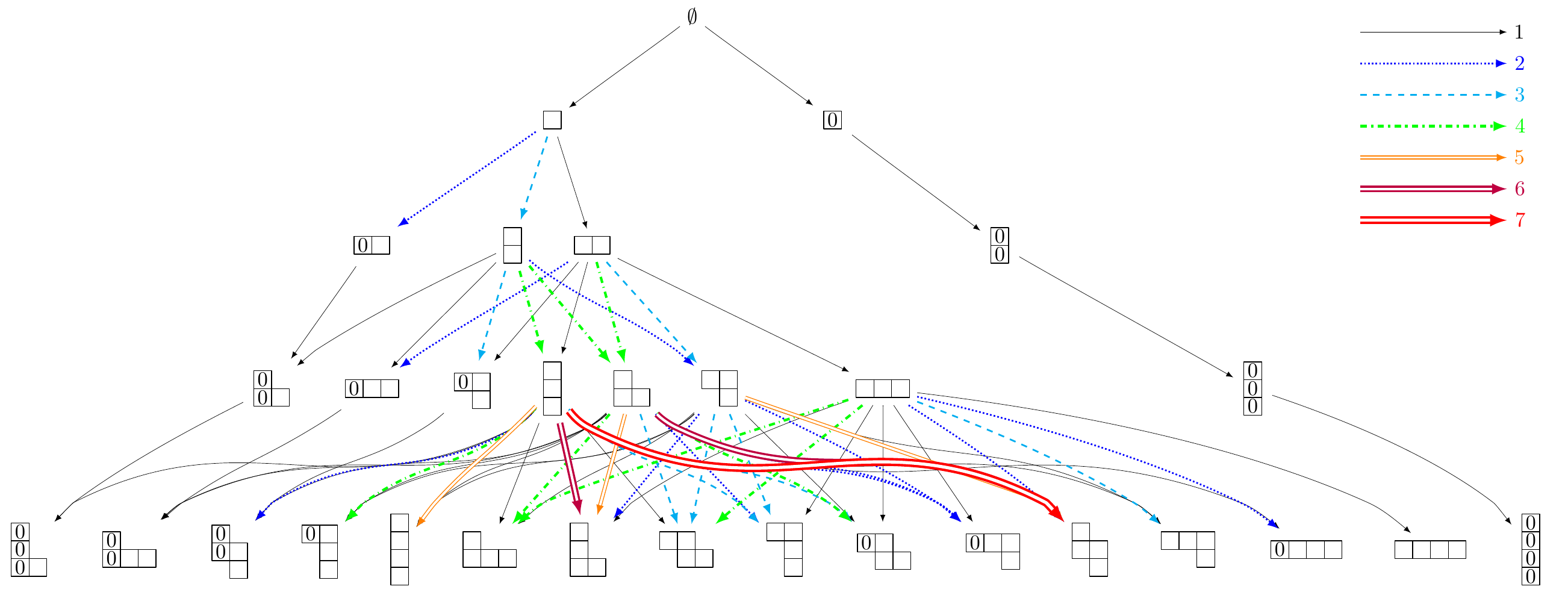}}}
  $$
  \caption{\label{branching_gauche} The left dual branching graph of $\Rnz$.}
\end{figure}
\begin{figure}[ht]
  $$
  \vcenter{\hbox{\includegraphics[scale=0.6]{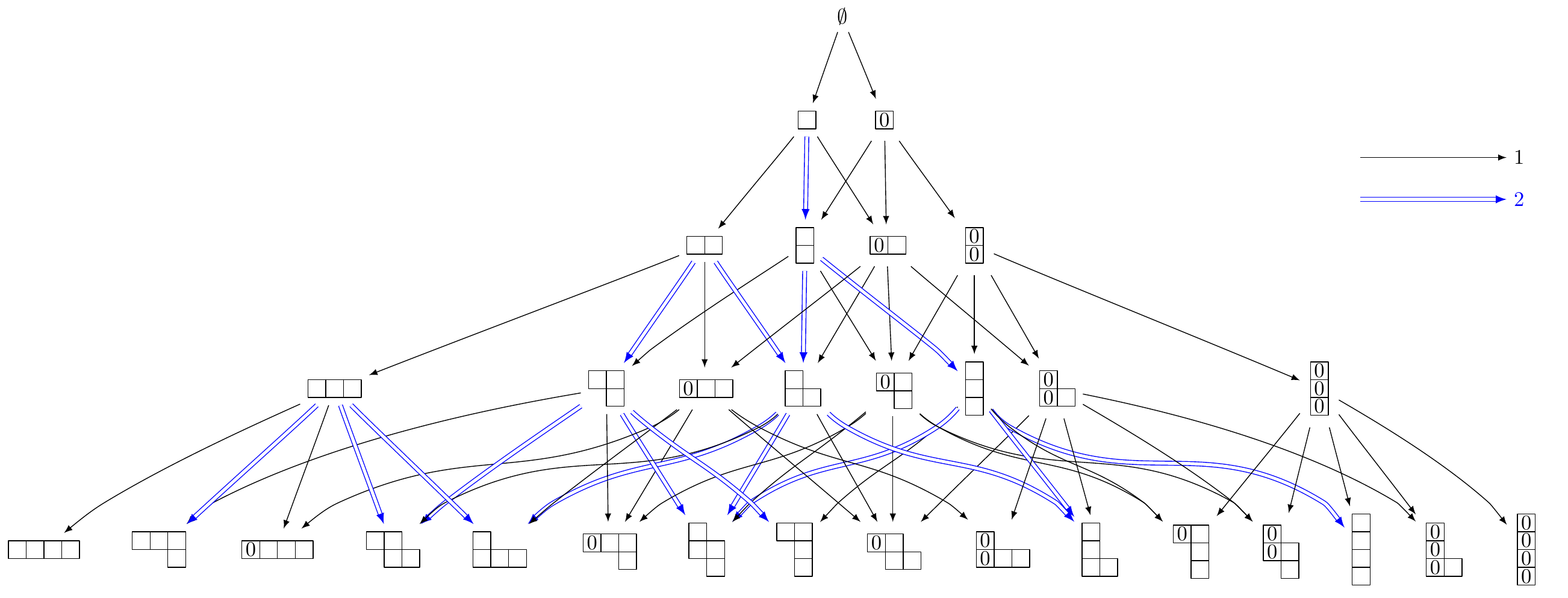}}}
  $$
  \caption{\label{branching_droit} The right dual branching graph of $\Rnz$.}
\end{figure}

\paragraph{Hopf algebra}
On the contrary to $\Hnz$, we do not get a Hopf algebra. Indeed, the following
diagram that express the compatibility of the product with the co-product
does not commute:
$$
\begin{array}{ccc}
R_{a+b}^0 \times R_{c+d}^0 & \overset{\Ind}{\longrightarrow} & R_{a+b+c+d}^0 \\
\raisebox{-30pt}{\rotatebox{90}{$\overset{\Res \times \Res}{\longleftarrow}$}} & & \raisebox{-5pt}{\rotatebox{270}{$\overset{\Res}{\longrightarrow}$}}\\
R_a^0\times R_b^0 \times R_c^0 \times R_d^0 & \overset{\Ind\times \Ind}{\longrightarrow} & R_{a+c}^0 \times R_{b+d}^0
\end{array}
$$
Here is a counter example: Using Theorem~\ref{ind_simple_tour}, we get
$\Res_{R_1^0 \times R_2^0}^{R_3^0} \Si{0,1}^3 = \Si{0}^1\otimes \Si{0}^2$ and
$\Res_{R_1^0\times R_1^0}^{R_2^0} \Si{1}^2 = \Si{}^1\otimes \Si{}^1 $. Then
\begin{equation}
  \begin{aligned}
    \Ind \times \Ind \left(\Res\times \Res \Si{0,1}^3 \otimes \Si{1}^2\right) &
    = \Ind(\Si{0}^1\otimes \Si{}^1) \otimes \Ind(\Si{0}^2\otimes \Si{}^1) \\ & =
    (\Si{0}^2 + \Si{1}^2) \otimes
    (\Si{0}^3+\Si{0,1}^3+\Si{0,2}^3+\Si{1}^3)\,.
  \end{aligned}
\end{equation}
Hence this sum has $8$ elements, with multiplicity. On the other hand, we saw
in Example~\ref{ex_ind_2} that $\Ind \Si{0,1}^3\times \Si{1}^2$ is a sum of
$16$ elements (with multiplicity) and Theorem~\ref{rest_simple_RR} shows
that the multiplicity does not change by restriction. Hence the result is
false.

\paragraph{Induction with $\Hnz$}
One can wonder what would happen if we rather consider the induction and
restriction along the inclusion $\Rnz\times H_m^0 \rightarrow R_{n+m}^m$. It
is not a tower of monoids, but the morphisms $\tilde{\rho}_{n,m} \eqdef
(\rho_{n,m})_{| \Rnz\times H_m^0}$ are injective. We just give the result of
the induction of simple modules:
\begin{theorem}
  \newcommand\Qef{\Ind_{\Rnz \times H_m^0}^{R_{n+m}^0} S_I \otimes S_J} 
  For $n, m\in \N$, let $I\subseteq\interv{0}{n-1}$ and
  $J\subseteq\interv{1}{m-1}$. Denoting $e\eqdef\pi_{I, n}\in \Rnz$ and $f\eqdef\pi_{J, m} \in H_m^0$,
  the induction of simple modules $S_I=S_e$ and $S_J=S_f$ is given by
  \begin{enumerate}
  \item If $0\in I$, let $\ell$ be the first letter of $\w{f} = f_1\dots f_m$. Then:
    \begin{multline}
      \Qef = \left\langle
        \w{e}\shuffle \decn{\w{f}} +
        0\w{e}\shuffle \decn{f_2\dots f_m} +
        00\w{e}\shuffle \decn{f_3\dots f_m} +
      \right.\\ \left.\dots\qquad +
        0\dots 0\w{e}\shuffle \decn{f_{\ell+1}\dots f_m}
      \right\rangle,
    \end{multline}
    where the last term begins with $\ell$ letters $0$.
  \item If $0\notin I$, let $\ell$ be the first letter of $\w{f}
    = f_1\dots f_m$. Then
    \begin{multline}
      \Qef = \left\langle
        \w{e}\shuffle \decn{\w{f}} +
        0\shuffle\w{e}\shuffle\decn{f_2\dots f_m} +
        00\shuffle\w{e}\shuffle \decn{f_3\dots f_m} +
      \right.\\ \left.\dots\qquad +
        0\dots 0 \shuffle\w{e}\shuffle \decn{f_{\ell+1}\dots f_m}
      \right\rangle,
    \end{multline}
    where the last term begins with $\ell$ letters $0$.
  \end{enumerate}
\end{theorem}
\begin{proof}
  This is a consequence of Theorem~\ref{ind_simple_tour}.
\end{proof}

\subsubsection{Projective indecomposable modules}

\paragraph{Restriction of indecomposable projective modules}

In order to get a co-product on the Grothendieck ring of projective modules,
$\mathcal{K}_0$, we need that $R^0_{m+n}$ is projective over
$R^0_{m}\times\Rnz$. Unfortunately, this is not the case. We will moreover
give counterexamples to the fact that $R_n^0$ is projective over $R_{n-1}^0$
for both embedding $\rho_{n-1,1}$ and $\rho_{1,n-1}$. This forbids to have any
analogues of Bratelli diagrams for projective modules.

Let us take $P_{\lbrace 0, 2, 3\rbrace}$. We want to restrict this projective
indecomposable module of $R_4^0$ to $R_2^0\times R_2^0$.
\begin{figure}[ht]
  \centering
  \includegraphics[scale=0.5]{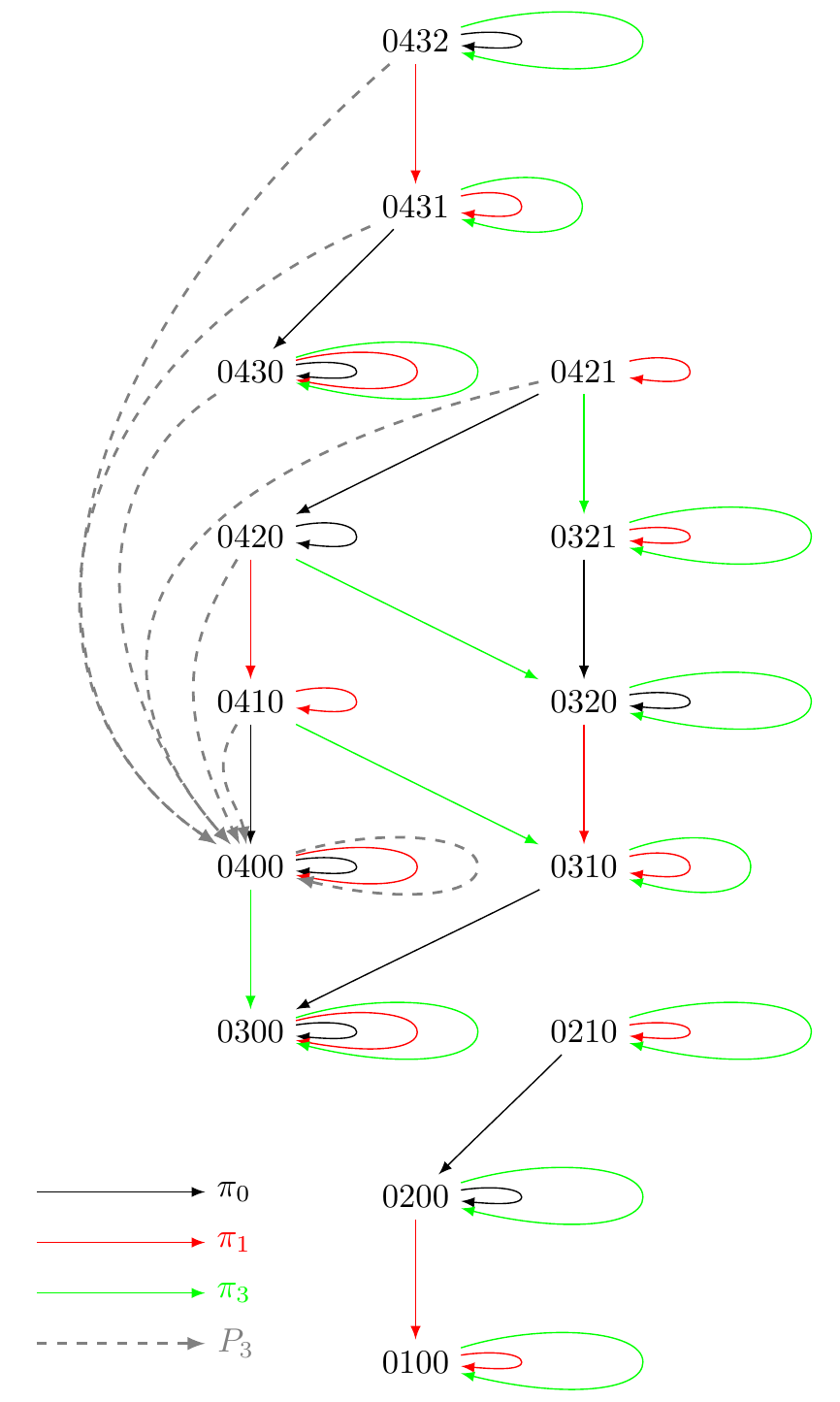}
  \includegraphics[scale=1]{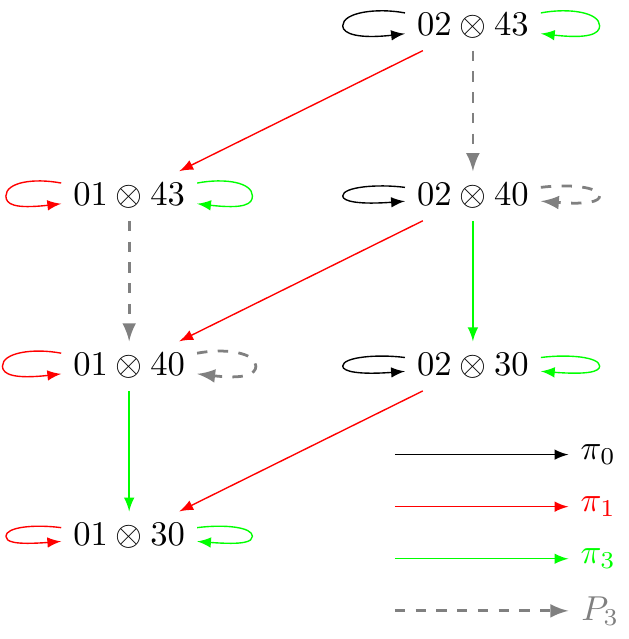}
  \caption{First counterexample for the restriction of projective
    modules.}\label{contreex_restr_proj_1}
\end{figure}
In Figure~\ref{contreex_restr_proj_1} we have on the left the module
$P_{\lbrace 0,2,3\rbrace}^4$ where we deleted the arrows of $\pi_2$ and showed
the action of $P_3$. Here we see that $P_3$ has a stable subspace of dimension
$1$. On the right we represent what would be a necessary part of the
decomposition of $P_{\lbrace 0,2,3\rbrace}^4$, that is $P_{\lbrace 0\rbrace}^2
\otimes P_{\lbrace 1\rbrace}^2$. Here we see that $P_3$ (that is the $\pi_0$
of the right $R_2^0$ according to the embedding \ref{plongement}) as a stable
subspace of dimension $2$. Hence it is impossible to cut the left one to get a
sum of projective indecomposable modules since the right one must be there and
can not be.

We give now two counterexamples which show that it does not work also for the
restriction along both embeddings $\rho_{n-1,1}$ and $\rho_{1,n-1}$.
\begin{figure}[!ht]
\centering
\includegraphics[scale=0.5]{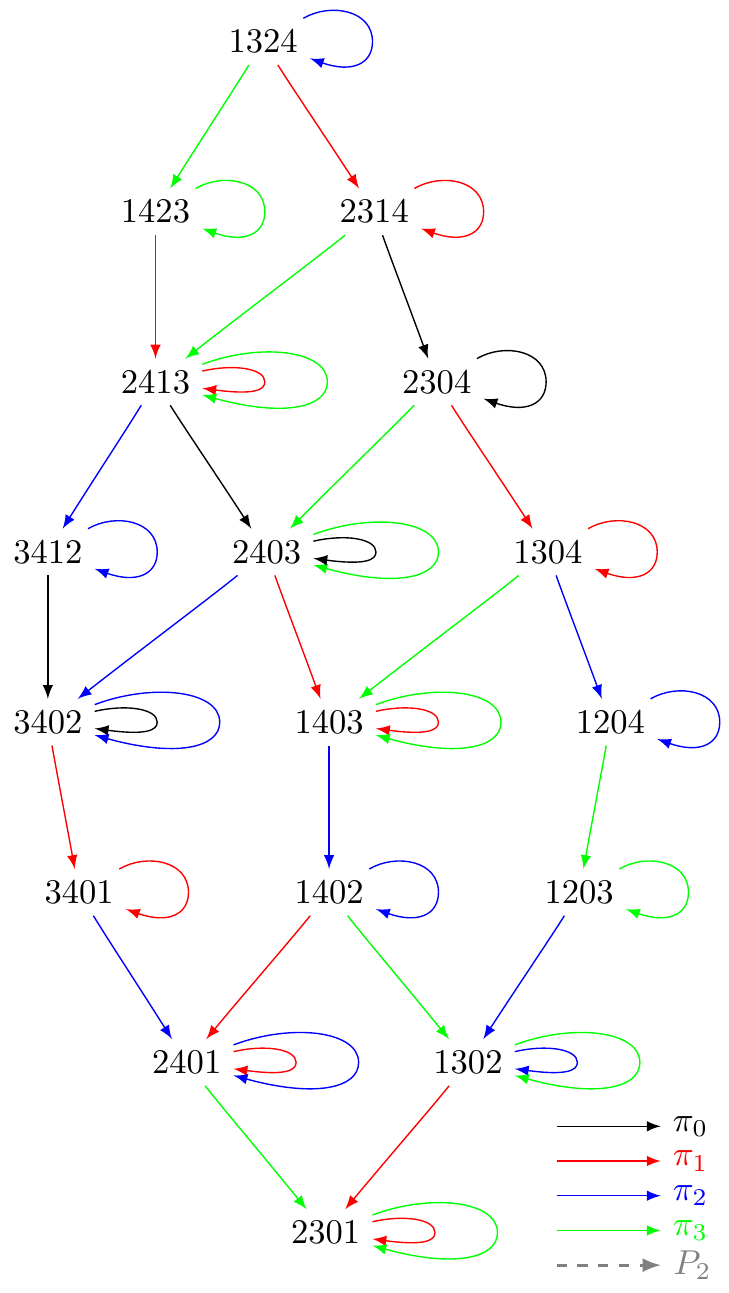}
\includegraphics[scale=0.5]{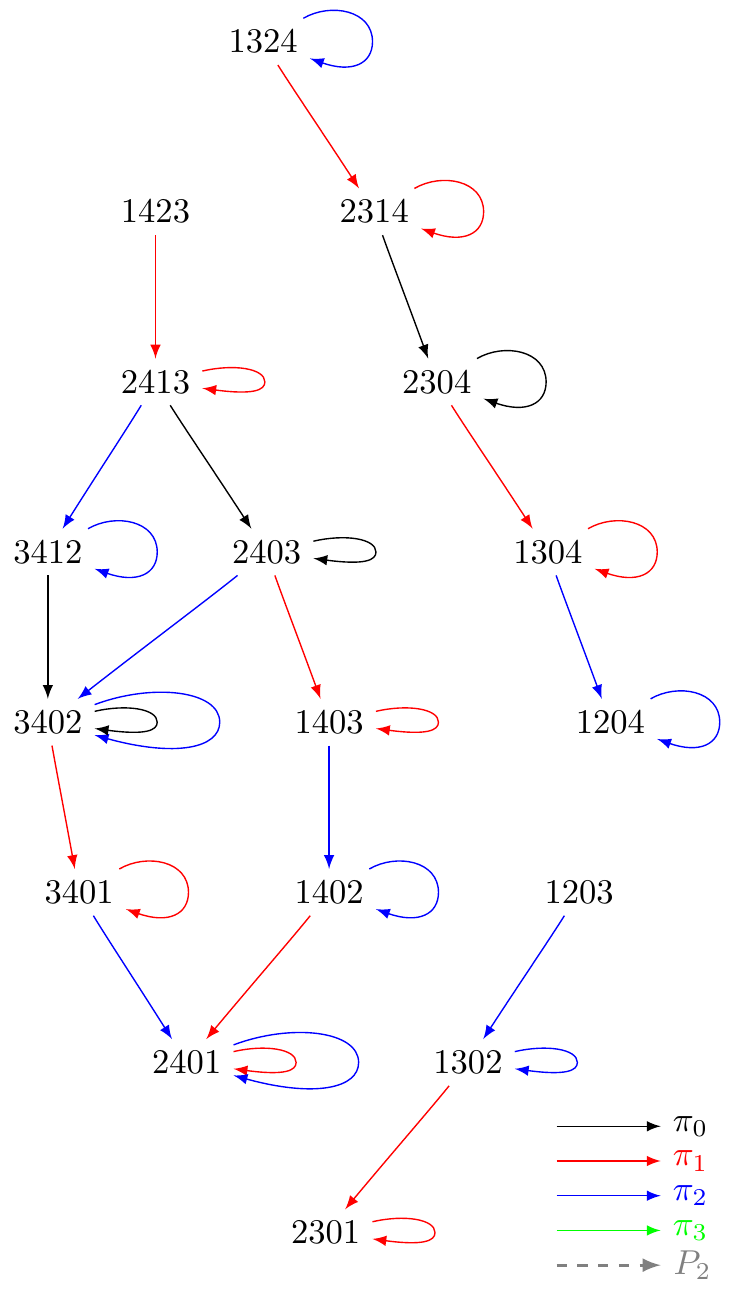}
\includegraphics[scale=0.5]{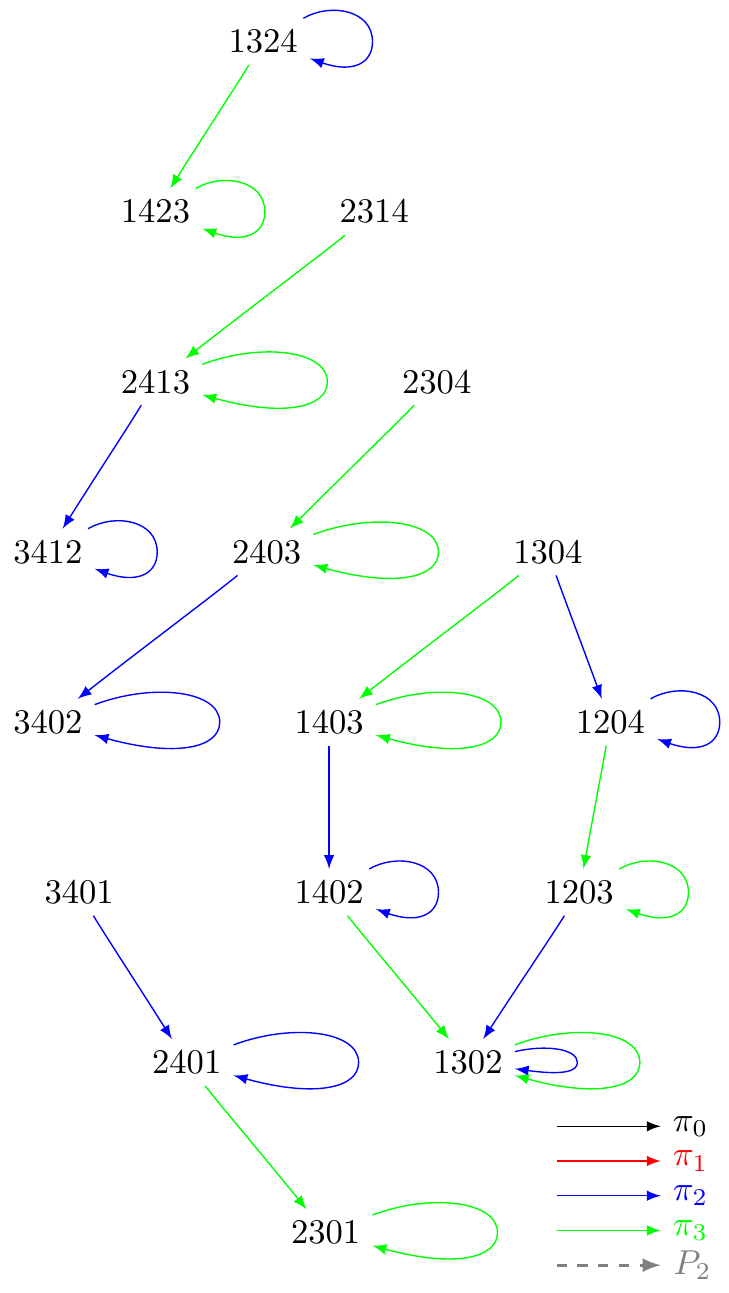}
\caption{Second counterexample for the restriction of projective modules.}\label{contreex_restr_proj_2}
\end{figure}
On the left of Figure \ref{contreex_restr_proj_2} we have the projective
module $P_{\lbrace 2\rbrace}^4$. We see that no element of this module has two
zeros, hence $P_2$ send every element to zero. In the middle of the figure we
have the same module where we forgot the action of $\pi_3$, that is we are
looking at the restriction $R_4^0\rightarrow R_{3}^0\otimes R_1^0$. In the
left one we forgot the action of both $\pi_0$ and $\pi_1$ but put the action
of $P_2$ (none here): we are looking at the restriction along
$R_4^0\rightarrow R_1^0\otimes R_{3}^0$.
If the middle and right modules were projective, these figures could be cut as
projective modules of $R_3^0$. We proceed step by step on the middle
one. First we recognise the first chain of five elements which is $P_{\lbrace
  2 \rbrace}^3$. Then the element $1423$ is $P_{\lbrace \rbrace}^3$. All the
cycles below with element on top $2413$ is $P_{\lbrace 1 \rbrace}^3$. The
element $1203$ is again $P_{\lbrace \rbrace}^3$. But the last two elements do
not correspond to any projective modules of $R_3^0$ (it should correspond to
$P_{\lbrace 2 \rbrace}^3$ since $1302$ only has the loop of $\pi_2$, which is
not the case).

We proceed the same way for the right module. We immediatly have a
contradiction with the first element which should generate $P_{\lbrace 1
  \rbrace}^3$ (be careful of the labels!) which is not the case.

As a conclusion of this paragraph, since we do not have the restriction of
indecomposable projective modules, we will not be able to have a tower of
monoids as for the case of $\Hnz$ to get $\NCSF$ and
$\QSym$~\cite{KrobThibon.1997}.

\paragraph{Induction of indecomposable projective modules}
For this one we can use Frobenius reciprocity as we did in Proposition \ref{restindecRH}, using Theorem \ref{rest_simple_RR}:
\begin{theorem}
Let $I\subset \interv{0}{n-1}$ and $J\subset \interv{0}{m-1}$. Then
$$ \Ind_{R_n^0\times R_m^0}^{R_{n+m}^0} P_I\otimes P_J =
\begin{cases}
P_{I\cup \decn{J}}\oplus P_{I\cup\lbrace n\rbrace\cup  \decn{J}}  & \text{ if } 0\notin J \\
P_{\interv{0}{n}\cup  \decn{J \setminus \lbrace 0\rbrace}} & \text{ if } 0\in J \text{ and } I = \interv{0}{n-1} \\
 0 & \text{ otherwise.}
\end{cases}.$$
\end{theorem}

\begin{proof}
  We reason as in the proof of Proposition~\ref{restindecRH}, using Frobenius
  reciprocity:
  \begin{equation}
    \Hom_{R_{n+m}^0}\left(
      \Ind_{R_n^0\times R_m^0}^{R_{n+m}^0} P_I\otimes P_J, S_K \right)
    =
    \Hom_{\Rnz\otimes R_m^0}\left(
      P_I\otimes P_J, \Res_{R_n^0\times R_m^0}^{R_{n+m}^0} S_K \right)\,.
  \end{equation}
  We are looking for sets $K\subset \interv{0}{n+m-1}$ such that the simple
  $R_{n+m}^0$-module $S_K$ restricts to $S_I\otimes S_J$ over $R_n^0\times
  R_m^0$. If $0\notin J$ then $K\cap \interv{0}{n-1} = I$ and $K\cap
  \interv{n+1}{n+m-1} = \decn{J}$. We conclude considering the two cases
  whether $n\in K$ or not. On the contrary, if $0\in J$ then we are in the
  second case of Theorem \ref{rest_simple_RR}. So either
  $K\cap\interv{0}{n}=\interv{0}{n} $ that is $I = \interv{0}{n-1}$, and we
  have the second case, either it is wrong and in this case no restriction can
  be obtained.
\end{proof}

As we have seen, the natural tower of monoids structure of $(\Rnz)_{n\in\N}$
described here does not have a very nice representation theory. However, this
is not the only tower structure, and they may be nice tower structure on their
algebras involving linear combination.

\appendix

\section{Implementation}

A large part of the algorithms here are implemented in
\texttt{Sagemath}~\cite{Sagemath}. The representation theory where computed
using \texttt{sage\_semigroups}~\cite{sage-semigroup} from the second author,
F.~Saliola and N.~Thiéry. The code is freely
accessible at \[\text{\url{https://github.com/hivert/Jupyter-Notebooks}}\] Thanks to
the binder technology, one can experiment with it online
at \[\text{\small\url{https://mybinder.org/v2/gh/hivert/Jupyter-Notebooks/master?filepath=rook-0.ipynb}}\]

\section{Tables}

\paragraph{Decomposition functor}\label{decFunctorTable}

We give the decomposition functor from projective $\Rnz$-modules into
$\Hnz$-modules. They where computed according to Theorem~\ref{decompo_H_R}.
\newcommand{\isom}{\simeq}
\begin{alignat*}{2}
  P^R_{(1)} &\isom P_{(1)}&
  P^R_{(0, 1)} &\isom P_{(1)}\\
  P^R_{(2)} &\isom P_{(2)}&
  P^R_{(0, 2)} &\isom P_{(1, 1)} + P_{(2)}\\
  P^R_{(1, 1)} &\isom 2P_{(1, 1)} + P_{(2)}&
  P^R_{(0, 1, 1)} &\isom P_{(1, 1)}\\
  P^R_{(3)} &\isom P_{(3)}&
  P^R_{(0, 3)} &\isom P_{(1, 2)} + P_{(3)}\\
  P^R_{(2, 1)} &\isom P_{(1, 2)} + P_{(2, 1)} + P_{(3)}&
  P^R_{(0, 2, 1)} &\isom 2P_{(1, 1, 1)} + P_{(1, 2)} + P_{(2, 1)}\\
  P^R_{(1, 2)} &\isom P_{(1, 1, 1)} + 2P_{(1, 2)} + P_{(2, 1)} + P_{(3)}\qquad&
  P^R_{(0, 1, 2)} &\isom P_{(1, 1, 1)} + P_{(1, 2)}\\
  P^R_{(1, 1, 1)} &\isom 3P_{(1, 1, 1)} + P_{(1, 2)} + P_{(2, 1)}\qquad&
  P^R_{(0, 1, 1, 1)} &\isom P_{(1, 1, 1)}\\
\end{alignat*}
\begin{alignat*}{1}
  P^R_{(4)} &\isom P_{(4)}\\
  P^R_{(0, 4)} &\isom P_{(1, 3)} + P_{(4)}\\
  P^R_{(3, 1)} &\isom P_{(1, 3)} + P_{(3, 1)} + P_{(4)}\\
  P^R_{(0, 3, 1)} &\isom P_{(1, 1, 2)} + P_{(1, 2, 1)} + P_{(1, 3)} + P_{(3, 1)}\\
  P^R_{(2, 2)} &\isom P_{(1, 2, 1)} + P_{(1, 3)} + P_{(2, 2)} + P_{(3, 1)} + P_{(4)}\\
  P^R_{(0, 2, 2)} &\isom P_{(1, 1, 1, 1)} + 2P_{(1, 1, 2)} + P_{(1, 2, 1)} + P_{(1, 3)} + P_{(2, 2)}\\
  P^R_{(2, 1, 1)} &\isom P_{(1, 1, 2)} + P_{(1, 2, 1)} + P_{(1, 3)} + P_{(2, 1, 1)} + P_{(3, 1)}\\
  P^R_{(0, 2, 1, 1)} &\isom 3P_{(1, 1, 1, 1)} + P_{(1, 1, 2)} + P_{(1, 2, 1)} + P_{(2, 1, 1)}\\
  P^R_{(1, 3)} &\isom P_{(1, 1, 2)} + 2P_{(1, 3)} + P_{(2, 2)} + P_{(4)}\\
  P^R_{(0, 1, 3)} &\isom P_{(1, 1, 2)} + P_{(1, 3)}\\
  P^R_{(1, 2, 1)} &\isom 2P_{(1, 1, 1, 1)} + 2P_{(1, 1, 2)} + 3P_{(1, 2, 1)} + P_{(1, 3)} + P_{(2, 1, 1)} + P_{(2, 2)} + P_{(3, 1)}\\
  P^R_{(0, 1, 2, 1)} &\isom 2P_{(1, 1, 1, 1)} + P_{(1, 1, 2)} + P_{(1, 2, 1)}\\
  P^R_{(1, 1, 2)} &\isom 2P_{(1, 1, 1, 1)} + 3P_{(1, 1, 2)} + P_{(1, 2, 1)} + P_{(1, 3)} + P_{(2, 1, 1)} + P_{(2, 2)}\\
  P^R_{(0, 1, 1, 2)} &\isom P_{(1, 1, 1, 1)} + P_{(1, 1, 2)}\\
  P^R_{(1, 1, 1, 1)} &\isom 4P_{(1, 1, 1, 1)} + P_{(1, 1, 2)} + P_{(1, 2, 1)} + P_{(2, 1, 1)}\\
  P^R_{(0, 1, 1, 1, 1)} &\isom P_{(1, 1, 1, 1)}
\end{alignat*}

\paragraph{Cartan matrices}
We show below the first Cartan matrices of the $0$-rook monoids $R_n$ for
$n=2,3,4,5$. The column on the left shows the associated idempotents.
\[
\begin{smallmatrix}
12\\02\\21\\00
\end{smallmatrix}
\begin{psmallmatrix}
 1 & . & . & . \\
 . & 1 & 1 & . \\
 . & 1 & 2 & . \\
 . & . & . & 1
 \end{psmallmatrix}
\qquad
\begin{smallmatrix}
123\\023\\213\\003\\132\\032\\321\\000
\end{smallmatrix}
\begin{psmallmatrix}
 1 & . & . & . & . & . & . & . \\
 . & 1 & 1 & . & 1 & . & . & . \\
 . & 1 & 3 & . & 2 & 1 & 1 & . \\
 . & . & . & 1 & . & 1 & 1 & . \\
 . & 1 & 2 & . & 2 & . & . & . \\
 . & . & 1 & 1 & . & 2 & 2 & . \\
 . & . & 1 & 1 & . & 2 & 3 & . \\
 . & . & . & . & . & . & . & 1
\end{psmallmatrix}
\qquad
\begin{smallmatrix}
1234\\0234\\2134\\0034\\1324\\0324\\3214\\0004\\
1243\\0243\\2143\\0043\\1432\\0432\\4321\\0000
\end{smallmatrix}
\begin{psmallmatrix}
 1 & . & . & . & . & . & . & . & . & . & . & . & . & . & . & . \\
 . & 1 & 1 & . & 1 & . & . & . & 1 & . & . & . & . & . & . & . \\
 . & 1 & 3 & . & 3 & 1 & 1 & . & 2 & 1 & 2 & . & 1 & . & . & . \\
 . & . & . & 1 & . & 1 & 1 & . & . & 1 & 1 & . & 1 & . & . & . \\
 . & 1 & 3 & . & 4 & 1 & 1 & . & 2 & 1 & 3 & . & 1 & . & . & . \\
 . & . & 1 & 1 & 1 & 3 & 3 & . & . & 2 & 4 & 1 & 2 & 1 & 1 & . \\
 . & . & 1 & 1 & 1 & 3 & 5 & . & . & 2 & 6 & 1 & 3 & 2 & 2 & . \\
 . & . & . & . & . & . & . & 1 & . & . & . & 1 & . & 1 & 1 & . \\
 . & 1 & 2 & . & 2 & . & . & . & 2 & . & . & . & . & . & . & . \\
 . & . & 1 & 1 & 1 & 2 & 2 & . & . & 2 & 3 & . & 2 & . & . & . \\
 . & . & 2 & 1 & 3 & 4 & 6 & . & . & 3 & 9 & 1 & 4 & 2 & 2 & . \\
 . & . & . & . & . & 1 & 1 & 1 & . & . & 1 & 2 & . & 2 & 2 & . \\
 . & . & 1 & 1 & 1 & 2 & 3 & . & . & 2 & 4 & . & 3 & . & . & . \\
 . & . & . & . & . & 1 & 2 & 1 & . & . & 2 & 2 & . & 3 & 3 & . \\
 . & . & . & . & . & 1 & 2 & 1 & . & . & 2 & 2 & . & 3 & 4 & . \\
 . & . & . & . & . & . & . & . & . & . & . & . & . & . & . & 1
\end{psmallmatrix}\]

\[\begin{smallmatrix}
12345\\02345\\21345\\00345\\13245\\03245\\32145\\00045\\
12435\\02435\\21435\\00435\\14325\\04325\\43215\\00005\\
12354\\02354\\21354\\00354\\13254\\03254\\32154\\00054\\
12543\\02543\\21543\\00543\\15432\\05432\\54321\\00000
\end{smallmatrix}
\begin{psmallmatrix}
 1 & . & . & . & . & . & . & . & . & . & . & . & . & . & . & . & . & . & . & . & . & . & . & . & . & . & . & . & . & . & . & . \\
 . & 1 & 1 & . & 1 & . & . & . & 1 & . & . & . & . & . & . & . & 1 & . & . & . & . & . & . & . & . & . & . & . & . & . & . & . \\
 . & 1 & 3 & . & 3 & 1 & 1 & . & 3 & 1 & 2 & . & 1 & . & . & . & 2 & 1 & 2 & . & 2 & . & . & . & 1 & . & . & . & . & . & . & . \\
 . & . & . & 1 & . & 1 & 1 & . & . & 1 & 1 & . & 1 & . & . & . & . & 1 & 1 & . & 1 & . & . & . & 1 & . & . & . & . & . & . & . \\
 . & 1 & 3 & . & 5 & 1 & 1 & . & 4 & 2 & 5 & . & 2 & . & . & . & 2 & 1 & 4 & . & 4 & 1 & 1 & . & 1 & . & 1 & . & . & . & . & . \\
 . & . & 1 & 1 & 1 & 3 & 3 & . & 1 & 3 & 5 & 1 & 3 & 1 & 1 & . & . & 2 & 4 & 1 & 4 & 2 & 2 & . & 2 & 1 & 2 & . & 1 & . & . & . \\
 . & . & 1 & 1 & 1 & 3 & 5 & . & 1 & 3 & 8 & 1 & 5 & 2 & 2 & . & . & 2 & 6 & 1 & 6 & 3 & 4 & . & 3 & 2 & 4 & . & 2 & . & . & . \\
 . & . & . & . & . & . & . & 1 & . & . & . & 1 & . & 1 & 1 & . & . & . & . & 1 & . & 1 & 1 & . & . & 1 & 1 & . & 1 & . & . & . \\
 . & 1 & 3 & . & 4 & 1 & 1 & . & 4 & 1 & 3 & . & 1 & . & . & . & 2 & 1 & 3 & . & 3 & . & . & . & 1 & . & . & . & . & . & . & . \\
 . & . & 1 & 1 & 2 & 3 & 3 & . & 1 & 4 & 7 & 1 & 4 & 1 & 1 & . & . & 2 & 5 & 1 & 5 & 3 & 3 & . & 2 & 1 & 3 & . & 1 & . & . & . \\
 . & . & 2 & 1 & 5 & 5 & 8 & . & 3 & 7 & 21 & 2 & 11 & 5 & 5 & . & . & 3 & 13 & 2 & 14 & 9 & 13 & . & 4 & 4 & 12 & 1 & 4 & 1 & 1 & . \\
 . & . & . & . & . & 1 & 1 & 1 & . & 1 & 2 & 3 & 1 & 3 & 3 & . & . & . & 1 & 2 & 1 & 4 & 4 & 1 & . & 2 & 4 & 1 & 2 & 1 & 1 & . \\
 . & . & 1 & 1 & 2 & 3 & 5 & . & 1 & 4 & 11 & 1 & 7 & 2 & 2 & . & . & 2 & 7 & 1 & 8 & 4 & 6 & . & 3 & 2 & 6 & . & 2 & . & . & . \\
 . & . & . & . & . & 1 & 2 & 1 & . & 1 & 5 & 3 & 2 & 5 & 5 & . & . & . & 2 & 2 & 2 & 6 & 8 & 1 & . & 3 & 7 & 2 & 3 & 2 & 2 & . \\
 . & . & . & . & . & 1 & 2 & 1 & . & 1 & 5 & 3 & 2 & 5 & 7 & . & . & . & 2 & 2 & 2 & 6 & 10 & 1 & . & 3 & 9 & 2 & 4 & 3 & 3 & . \\
 . & . & . & . & . & . & . & . & . & . & . & . & . & . & . & 1 & . & . & . & . & . & . & . & 1 & . & . & . & 1 & . & 1 & 1 & . \\
 . & 1 & 2 & . & 2 & . & . & . & 2 & . & . & . & . & . & . & . & 2 & . & . & . & . & . & . & . & . & . & . & . & . & . & . & . \\
 . & . & 1 & 1 & 1 & 2 & 2 & . & 1 & 2 & 3 & . & 2 & . & . & . & . & 2 & 3 & . & 3 & . & . & . & 2 & . & . & . & . & . & . & . \\
 . & . & 2 & 1 & 4 & 4 & 6 & . & 3 & 5 & 13 & 1 & 7 & 2 & 2 & . & . & 3 & 10 & 1 & 10 & 4 & 5 & . & 4 & 2 & 5 & . & 2 & . & . & . \\
 . & . & . & . & . & 1 & 1 & 1 & . & 1 & 2 & 2 & 1 & 2 & 2 & . & . & . & 1 & 2 & 1 & 3 & 3 & . & . & 2 & 3 & . & 2 & . & . & . \\
 . & . & 2 & 1 & 4 & 4 & 6 & . & 3 & 5 & 14 & 1 & 8 & 2 & 2 & . & . & 3 & 10 & 1 & 11 & 4 & 6 & . & 4 & 2 & 6 & . & 2 & . & . & . \\
 . & . & . & . & 1 & 2 & 3 & 1 & . & 3 & 9 & 4 & 4 & 6 & 6 & . & . & . & 4 & 3 & 4 & 9 & 11 & 1 & . & 4 & 10 & 2 & 4 & 2 & 2 & . \\
 . & . & . & . & 1 & 2 & 4 & 1 & . & 3 & 13 & 4 & 6 & 8 & 10 & . & . & . & 5 & 3 & 6 & 11 & 18 & 1 & . & 5 & 16 & 3 & 6 & 4 & 4 & . \\
 . & . & . & . & . & . & . & . & . & . & . & 1 & . & 1 & 1 & 1 & . & . & . & . & . & 1 & 1 & 2 & . & . & 1 & 2 & . & 2 & 2 & . \\
 . & . & 1 & 1 & 1 & 2 & 3 & . & 1 & 2 & 4 & . & 3 & . & . & . & . & 2 & 4 & . & 4 & . & . & . & 3 & . & . & . & . & . & . & . \\
 . & . & . & . & . & 1 & 2 & 1 & . & 1 & 4 & 2 & 2 & 3 & 3 & . & . & . & 2 & 2 & 2 & 4 & 5 & . & . & 3 & 5 & . & 3 & . & . & . \\
 . & . & . & . & 1 & 2 & 4 & 1 & . & 3 & 12 & 4 & 6 & 7 & 9 & . & . & . & 5 & 3 & 6 & 10 & 16 & 1 & . & 5 & 15 & 2 & 6 & 3 & 3 & . \\
 . & . & . & . & . & . & . & . & . & . & 1 & 1 & . & 2 & 2 & 1 & . & . & . & . & . & 2 & 3 & 2 & . & . & 2 & 3 & . & 3 & 3 & . \\
 . & . & . & . & . & 1 & 2 & 1 & . & 1 & 4 & 2 & 2 & 3 & 4 & . & . & . & 2 & 2 & 2 & 4 & 6 & . & . & 3 & 6 & . & 4 & . & . & . \\
 . & . & . & . & . & . & . & . & . & . & 1 & 1 & . & 2 & 3 & 1 & . & . & . & . & . & 2 & 4 & 2 & . & . & 3 & 3 & . & 4 & 4 & . \\
 . & . & . & . & . & . & . & . & . & . & 1 & 1 & . & 2 & 3 & 1 & . & . & . & . & . & 2 & 4 & 2 & . & . & 3 & 3 & . & 4 & 5 & . \\
 . & . & . & . & . & . & . & . & . & . & . & . & . & . & . & . & . & . & . & . & . & . & . & . & . & . & . & . & . & . & . & 1
\end{psmallmatrix}\]

\bibliographystyle{plainnat}
\bibliography{article}
\Addresses

\end{document}